\documentclass[12pt,a4paper,reqno]{amsart}

\usepackage[all,2cell,arrow,color]{xy}
\usepackage[english]{babel}
\usepackage{mathtools}
\usepackage{stmaryrd}
\usepackage{tensor}
\numberwithin{equation}{section}
\usepackage{graphicx} 
\usepackage{rotating}
\usepackage[pagebackref,colorlinks,urlcolor=darkgreen,citecolor=darkgreen,linkcolor=darkgreen]{hyperref}
\usepackage{enumerate,xspace}
\usepackage{xcolor}
\usepackage{floatpag}

\definecolor{darkgreen}{rgb}{0,0.45,0} 

  \newtheorem{proposition}{Proposition}[section]
  \newtheorem{lemma}[proposition]{Lemma}
  \newtheorem{corollary}[proposition]{Corollary}
  \newtheorem{theorem}[proposition]{Theorem}
  
  \theoremstyle{definition}
  \newtheorem{definition}[proposition]{Definition}
  \newtheorem{example}[proposition]{Example}

\theoremstyle{remark}
  \newtheorem{remark}[proposition]{Remark}
 
  \newcounter{c}
  \renewcommand{\[}{\setcounter{c}{1}$$}
  \newcommand{\etyk}[1]{\vspace{-7.4mm}$$\begin{equation}\Label{#1}
  \addtocounter{c}{1}}
  \renewcommand{\]}{\ifnum \value{c}=1 $$\else \end{equation}\fi}
\makeatletter
\newcommand*{\inlineequation}[2][]{%
  \begingroup
    \refstepcounter{equation}%
    \ifx\\#1\\%
    \else
      \label{#1}%
    \fi
    \relpenalty=10000 %
    \binoppenalty=10000 %
    \ensuremath{%
      #2%
    }%
    ~\@eqnnum
  \endgroup
}

\makeatletter
\def\@settitle{\begin{center}%
  \baselineskip14\p@\relax
  \bfseries
  \uppercasenonmath\@title
  \@title
  \ifx\@subtitle\@empty\else
     \\[1ex]\uppercasenonmath\@subtitle
     \footnotesize\mdseries\@subtitle
  \fi
  \end{center}%
}
\def\subtitle#1{\gdef\@subtitle{#1}}
\def\@subtitle{}
\makeatother

\newenvironment{amssidewaysfigure}
  {\begin{sidewaysfigure}\vspace*{.5\textwidth}\begin{minipage}{\textheight}\centering}
  {\end{minipage}\end{sidewaysfigure}}

\newcommand{\updot}{\raisebox{3pt}{$\cdot$}}
\newcommand{\vop}{{\raisebox{-3pt}{\hspace{1pt}\rotatebox{90}{$^\mathsf{op}$}\!\!}}}
\newcommand{\verti}{{\scalebox{.45}{$\mathbf \vert$}}}
\newcommand{\hori}{{\!-\!}}
\newcommand{\Longdownarrow}[1]{\mbox{\rotatebox{270}{$\Longrightarrow$} \raisebox{-12pt}{$#1$}}}
\newcommand{\Mnd}{\mathbb M\mathsf{nd}}
\newcommand{\Sqr}{\mathbb S\mathsf{qr}}

\begin{document}

\title[Multimonoidal monads]{The formal theory of multimonoidal monads}

\author{Gabriella B\"ohm} 
\address{Wigner Research Centre for Physics, H-1525 Budapest 114,
P.O.B.\ 49, Hungary}
\email{bohm.gabriella@wigner.mta.hu}
\date{April 2019}
 
\begin{abstract}
Certain aspects of Street's formal theory of monads in 2-categories are extended to multimonoidal monads in symmetric strict monoidal 2-categories.
Namely, any symmetric strict monoidal 2-category $\mathcal M$ admits a symmetric strict monoidal 2-category of pseudomonoids, monoidal 1-cells and monoidal 2-cells in $\mathcal M$.
Dually, there is a symmetric strict monoidal 2-category of pseudomonoids, opmonoidal 1-cells and opmonoidal 2-cells in $\mathcal M$.
Extending a construction due to Aguiar and Mahajan for $\mathcal M=\mathsf{Cat}$, we may apply the first construction $p$-times and the second one $q$-times (in any order). It yields a 2-category $\mathcal M_{pq}$.
A 0-cell therein is an object $A$ of $\mathcal M$ together with $p+q$ compatible pseudomonoid structures; it is termed a $(p+q)$-oidal object in $\mathcal M$. A monad in $\mathcal M_{pq}$ is called a $(p,q)$-oidal monad in $\mathcal M$; it is a monad $t$ on $A$ in $\mathcal M$ together with $p$ monoidal, and $q$ opmonoidal structures in a compatible way. 
If $\mathcal M$ has monoidal Eilenberg-Moore construction, and certain (Linton type) stable coequalizers exist, then a $(p+q)$-oidal structure on the Eilenberg-Moore object $A^t$ of a $(p,q)$-oidal monad $(A,t)$  is shown to arise via a symmetric strict monoidal double functor to Ehresmann's double category $\Sqr (\mathcal M)$ of squares in $\mathcal M$, from the double category of monads in $\Sqr (\mathcal M)$ in the sense of Fiore, Gambino and Kock. While $q$ ones of the pseudomonoid structures of $A^t$ are lifted along the `forgetful' 1-cell $A^t \to A$, the other $p$ ones are lifted along its left adjoint. 
In the particular example when $\mathcal M$ is an appropriate 2-subcategory of $\mathsf{Cat}$, this yields a conceptually different proof of some recent results due to Aguiar, Haim and L\'opez Franco.
\end{abstract}
  
\maketitle


\section*{Introduction} \label{sec:intro}

Classically, a {\em monad} on a category $A$ is a monoid in the category of endofunctors on $A$; that is,  a functor $t:A\to A$ together with natural transformations from the twofold iterate $t.t$ and from the identity functor, respectively, to $t$, regarded as an associative multiplication with a unit. A popular example is the monad $T\otimes -$ induced by an associative unital algebra $T$ on the category of vector spaces.

Any adjunction $l\dashv r:B \to A$ induces a monad $r.l$ on $A$, with unit provided by the unit of the  adjunction, and multiplication induced by the counit. Conversely, any monad is induced by some adjunction in this sense. There is no unique adjunction in general, but a terminal one can easily be described. An {\em Eilenberg-Moore algebra} of a monad $t$ on a category $A$ consists of an object $X$ of $A$ and an associative and unital action $tX \to X$. A morphism of Eilenberg-Moore algebras is a morphism in $A$ which commutes with the actions. (For the monad $T\otimes -$ induced by an algebra $T$ on the category of vector spaces, these are just left $T$-modules and their morphisms.) The evident forgetful functor $u^t$ (forgetting the actions) from the so defined {\em Eilenberg-Moore category} $A^t$ to $A$ has a left adjoint (sending an object $X$ to $tX$ with action provided by the multiplication of the monad) and this adjunction induces the monad $t$. Moreover, any other adjunction $l\dashv r:B \to A$ inducing the same monad factorizes through a unique functor $B \to A^t$. 

Having functors $f:A\to B$ between the base categories of some respective monads $t$ and $s$, and natural transformations between them, it is often a relevant question if they lift to the Eilenberg-Moore categories in the sense of the commutative diagram
\begin{equation} \label{eq:EMlift}\tag{EM}
\xymatrix{
A^t \ar[d]_-{u^t} \ar@{-->}[r] &
B^s \ar[d]^-{u^s} \\
A \ar[r]_-f &
B}
\end{equation}
where $u^t$ and $u^s$ are the respective forgetful functors.
It is not hard to see that such liftings of a functor $f$ correspond bijectively to natural transformations $\varphi:s.f\to f.t$ which are compatible with both monad structures. Such a pair $(f,\varphi)$ is called a {\em monad functor}. A natural transformation between functors $A\to B$ has at most one lifting to a natural transformation between the lifted functors, and the condition for its existence is a compatibility with the natural transformations $\varphi$. A natural transformation satisfying this compatibility condition is termed a {\em monad transformation}.

The {\em formal theory of monads} \cite{Street:FTM} due to Ross Street provides a wide generalization of the above picture and gives it a conceptual interpretation. The 2-category $\mathsf{Cat}$ of categories, functors and natural transformations is replaced by an arbitrary 2-category $\mathcal M$. A {\em monad} in $\mathcal M$ on a 0-cell $A$ is defined as a monoid in the hom category $\mathcal M(A,A)$.
Monads are the 0-cells in a 2-category $\mathsf{Mnd}(\mathcal M)$ whose 1-cells and 2-cells are the analogues of monad functors, and monad transformations, respectively.
Regarding any 0-cell of $\mathcal M$ as a trivial monad (with identity 1, and 2-cell parts), regarding any 1-cell of $\mathcal M$ as their monad morphism (with identity 2-cell part), and regarding any 2-cell of $\mathcal M$ as a monad transformation, there is an inclusion 2-functor $\mathcal M \to \mathsf{Mnd}(\mathcal M)$. For $\mathcal M =\mathsf {Cat}$ it has a right 2-adjoint. The right 2-adjoint sends a monad to its Eilenberg-Moore category, a monad functor to the corresponding lifted functor, and a monad transformation to the lifted natural transformation. When for some 2-category $\mathcal M$ the above inclusion 2-functor $\mathcal M \to \mathsf{Mnd}(\mathcal M)$ possesses a right 2-adjoint, $\mathcal M$ is said to {\em admit Eilenberg-Moore construction}. By the theory worked out in \cite{Street:FTM}, the right adjoint describes an analogous lifting theory. (More will be said in Section \ref{sec:q-EM}.)

For a monad $t$ on a {\em monoidal category} $A$, in \cite{Moerdijk} the additional structure was described which is equivalent to a monoidal structure on the Eilenberg-Moore category $A^t$ such that the forgetful functor $A^t\to A$ is strict monoidal.
The explicit computations of \cite{Moerdijk}  in $\mathsf{Cat}$ (considered with the Cartesian product of categories as the monoidal structure) were replaced in \cite{McCrudden}, \cite{ChikhladzeLackStreet} by abstract arguments about more general monoidal bicategories. 
Beyond a wide generalization, thereby also a conceptually different proof was obtained. Namely, the structure described in \cite {Moerdijk} was interpreted as an {\em opmonoidal monad}; that is, a monad in the 2-category $\mathsf{Cat}_{01}$ of monoidal categories, opmonoidal functors and opmonoidal natural transformations. 
Now  $\mathsf{Cat}_{01}$ admits Eilenberg-Moore construction in the sense of \cite{Street:FTM}, see \cite{Lack} and \cite{Zawadowski} for conceptually different proofs. Hence there is a 2-functor from Street's 2-category of monads $\mathsf{Mnd}(\mathsf{Cat}_{01})$ above to $\mathsf{Cat}_{01}$ whose object map sends an opmonoidal monad to its monoidal Eilenberg-Moore category.

Recently in \cite{AguiarHaimLopezFranco} a similar analysis to that of \cite{Moerdijk} was carried out for {\em multimonoidal monads} on {\em multimonoidal categories}. (In \cite{AguiarHaimLopezFranco} the term {\em higher monoidal} was used. However, we prefer to call the same thing {\em multimonoidal} and reserve the term {\em higher} to be used only for dimensionality of categorical structures.)
The aim of this paper is to extend the results of \cite{AguiarHaimLopezFranco} to multimonoidal monads in {\em symmetric strict monoidal 2-categories} and, more importantly, derive them from a suitable `formal theory'. With this we not only gain a deeper understanding of the origin of the formulae in \cite{AguiarHaimLopezFranco} but, as a byproduct, also open a way to some practical applications. For example, we obtain sufficient conditions under which the 2-category of pseudomonoids, monoidal 1-cells and monoidal 2-cells in a strict monoidal 2-category admits Eilenberg--Moore construction.

The development of this formal theory will require a move away from 2-categories to {\em symmetric strict monoidal double categories}.
By a {\em strict monoidal 2-category} we mean a monoid in the category of 2-categories and 2-functors considered with the Cartesian product of 2-categories as the monoidal product. This definition occurred e.g. on page 69 of \cite{Jay} where also the notion of {\em symmetry} was introduced as a suitable 2-natural isomorphism between the monoidal product 2-functor and its reversed mate.
Similarly restrictively, we adopt the definition of symmetric strict monoidal double category in \cite{BruniMeseguerMontanari}. Here again, a {\em strict monoidal double category} means a monoid in the category of double categories and double functors considered with the Cartesian product of double categories as the monoidal product. A {\em symmetry} in \cite{BruniMeseguerMontanari} can be interpreted then as a  suitable double natural isomorphism between the monoidal product double functor and its reversed mate.
(Note that in order for Ehresmann's square, or quintet construction \cite{Ehresmann} to yield a symmetric strict monoidal double category in this sense, we need to apply it to a symmetric strict monoidal 2-category.)

The notions recalled in the previous paragraph are very restrictive (by being so strict). One can ask about various levels of generalization whether they are possible.
Instead of monoids, one may consider pseudomonoids in the 2-category of 2-categories, 2-functors and 2-natural transformations; and correspondingly, pseudomonoids in the 2-category of double categories, double  functors and double natural transformations --- considered in both cases with the monoidal structure provided by the Cartesian product.
Although we expect that it should be possible, it is not motivated by our examples. Also, the technical complexity resulting from the tedious checking of all coherence conditions could divert attention from the key ideas.
It looks more challenging to extend our considerations to monoidal bicategories in the sense of 
\cite{KapranovVoevodsky,BaezNeuchl,SchommerPries} --- or at least to their semistrict version known as {\em Gray monoids} \cite{GordonPowerStreet,DayStreet}. These are monoids in the category of 2-categories and 2-functors considered with  the monoidal structure provided by the {\em Gray tensor product} \cite{Gray}. 
Since the corresponding Gray tensor product of double categories seems not yet available in the literature, this problem does not look to be within reach. We plan to address it elsewhere \cite{Bohm:Gray}.

For any symmetric strict monoidal 2-category $\mathcal M$, there is a symmetric strict monoidal 2-category $\mathcal M_{10}$ whose 0-cells are the pseudomonoids (also called {\em monoidal objects} e.g. in \cite{Zawadowski} or {\em monoidales} e.g. in \cite{ChikhladzeLackStreet}), the 1-cells are the monoidal 1-cells, and the 2-cells are the monoidal 2-cells in $\mathcal M$. Symmetrically, there is a symmetric strict monoidal 2-category $\mathcal M_{01}$ whose 0-cells are again  the pseudomonoids, but the 1-cells are the opmonoidal 1-cells, and the 2-cells are the opmonoidal 2-cells in $\mathcal M$. Moreover, these constructions commute with each other (see Section \ref{sec:Mpq}). Thus applying $p$ times the first construction and $q$ times the second one (in an arbitrary order), we get a 2-category $\mathcal M_{pq}$. For a fixed non-negative integer $n$ and every  $0\leq p \leq n$, the 0-cells of $\mathcal M_{p,n-p}$ are the same gadgets. They consist of a 0-cell of $\mathcal M$ together with $n$ pseudomonoid structures and compatibility morphisms between them (constituting suitable monoidal structures on the structure morphisms of the pseudomonoids). We call a 0-cell of  $\mathcal M_{p,n-p}$ an {\em $n$-oidal object} of $\mathcal M$. (In the particular case of $\mathcal M=\mathsf{Cat}$, in \cite{AguiarHaimLopezFranco,AguiarMahajan} it was called an {\em $n$-monoidal} category.) A $1$-oidal object is a pseudomonoid, in particular a monoidal category for $\mathcal M=\mathsf{Cat}$. So we re-obtain the classical terminology if ``1" is pronounced as ``mono". A 2-oidal object of $\mathsf{Cat}$ is a duoidal category (in the sense of \cite{Street:Belgian}, termed {\em 2-monoidal} in \cite{AguiarMahajan}). Again we agree with the established terminology if ``2" is pronounced as ``duo".
A 1-cell in $\mathcal M_{pq}$ consists of a 1-cell of $\mathcal M$ together with monoidal structures with respect to $p$ ones of the pseudomonoid structures; and opmonoidal structures with respect to the remaining $q$ ones of the pseudomonoid structures of the domain and the codomain. They are subject to suitable compatibility conditions. We term a 1-cell of $\mathcal M_{pq}$ a {\em $(p,q)$-oidal 1-cell} in $\mathcal M$ (rather than {\em $(p,q)$-monoidal} as in \cite{AguiarHaimLopezFranco,AguiarMahajan}; where in particular a $(2,0)$-oidal 1-cell of $\mathsf{Cat}$ was called a {\em double monoidal functor}, a (0,2)-oidal 1-cell was called {\em double comonoidal} and a (1,1)-oidal 1-cell was called {\em bimonoidal}).
A 2-cell in $\mathcal M_{pq}$ --- called a {\em $(p,q)$-oidal 2-cell} in $\mathcal M$ --- is a 2-cell in $\mathcal M$ which is compatible with all of the (op)monoidal structures of the domain and the codomain.
A monad in $\mathcal M_{pq}$ is termed a {\em $(p,q)$-oidal monad}.

As recalled above from \cite{McCrudden}, the monoidal structure of the base category $A$ of any opmonoidal (or $(0,1)$-oidal) monad $t$ in $\mathsf{Cat}$ lifts to the Eilenberg-Moore category $A^t$ along the forgetful functor $u^t:A^t\to A$. (This means that the functors and natural transformations constituting the monoidal structures of $A$ and $A^t$  fit in commutative diagrams as in \eqref{eq:EMlift}).
Furthermore, if reflexive coequalizers exist in $A$ and they are preserved by the monoidal product of $A$ and by the functor $t$, then also the  monoidal structure of the base category $A$ of a monoidal (or $(1,0)$-oidal) monad $t$ lifts to the Eilenberg-Moore category $A^t$. However, at this time it is a lifting along the left adjoint $f^t$ of the forgetful functor $u^t:A^t \to A$.  This means that the functors and natural transformations constituting the monoidal structures of $A$ and $A^t$ fit in commutative diagrams obtained from that in \eqref{eq:EMlift} replacing the forgetful functors with their left adjoints in the opposite direction. This is a result of \cite{Seal}; see also \cite{AguiarHaimLopezFranco}.
In \cite{AguiarHaimLopezFranco} it was proven, moreover, that under the same assumptions also the $(p+q)$-oidal structure of the base category $A$ of any $(p,q)$-oidal monad $t$ in $\mathsf{Cat}$ lifts to the Eilenberg-Moore category $A^t$. This is a lifting of mixed kind, though. While $q$ ones of the monoidal structures are lifted along the forgetful functor $u^t:A^t\to A$; the remaining $p$ ones are lifted along its left adjoint $f^t$. 

Because of this mixed nature of lifting; that is, since the different ingredients are lifted along different functors $u^t$ and $f^t$, we do not expect it to be described by some 2-functor (as in the situations of \cite{Street:FTM} and \cite{ChikhladzeLackStreet}). 
Instead, in this paper we deal with symmetric strict monoidal double categories and define their  $(p,q)$-oidal objects (see Section \ref{sec:Dpq}). These $(p,q)$-oidal objects are shown to be preserved by symmetric strict monoidal double functors.
In Ehresmann's double category $\Sqr (\mathcal M)$ of squares (or quintets \cite{Ehresmann}) in a symmetric strict monoidal 2-category $\mathcal M$, the $(p,q)$-oidal objects are the same as the $(p+q)$-oidal objects in $\mathcal M$.

Taking the double category of monads \cite{FioreGambinoKock} in the particular double category $\Sqr (\mathcal M)$, we obtain a symmetric strict monoidal double category which we denote by $\Mnd (\mathcal M)$. Its horizontal 2-category is $\mathsf{Mnd}(\mathcal M)$ and its vertical 2-category is $\mathsf{Mnd}(\mathcal M_{\mathsf{op}})_{\mathsf{op}}$ (where $(-)_{\mathsf{op}}$ refers to the horizontally opposite 2-category, see \cite[Section 4]{Street:FTM}). The $(p,q)$-oidal objects in $\Mnd (\mathcal M)$ are identified with the $(p,q)$-oidal monads in $\mathcal M$ (see Section \ref{sec:pq-monad}). Consequently, any symmetric strict monoidal double functor $\Mnd (\mathcal M) \to \Sqr (\mathcal M)$ takes $(p,q)$-oidal monads in $\mathcal M$ to $(p+q)$-oidal objects in $\mathcal M$.

Whenever a symmetric strict monoidal 2-category $\mathcal M$ admits monoidal Eilen\-berg-Moore construction (in the sense that the comparison 1-cells $I^1\to I$ and $A^tB^s \to (AB)^{ts}$ are identities for the identity monad on the monoidal unit $I$ and all monads $t$ on $A$ and $s$ on $B$, see Definition \ref{def:mon-EM}) and furthermore some Linton type coequalizers (see \eqref{eq:Linton}) exist and are preserved by the horizontal composition, we construct a symmetric strict monoidal double functor $\Mnd (\mathcal M) \to \Sqr (\mathcal M)$ (see Section \ref{sec:K}). In the terminology of \cite{FioreGambinoKock:free} it provides Eilenberg--Moore construction for the double category $\Sqr (\mathcal M)$. Its horizontal 2-functor is the Eilenberg-Moore 2-functor $\mathsf{Mnd}(\mathcal M) \to \mathcal M$ and its vertical 2-functor $\mathsf{Mnd}(\mathcal M_{\mathsf{op}})_{\mathsf{op}}\to \mathcal M$ is obtained via Linton type coequalizers. 
Its taking $(p,q)$-oidal monads in $\mathcal M$ to their $(p+q)$-oidal Eilenberg-Moore objects provides a `formal theory' in the background of the liftings in \cite{AguiarHaimLopezFranco}.

\subsection*{Acknowledgement} \ 
It is a pleasure to thank Ross Street and two anonymous referees for helpful comments, insightful questions and mentioning some relevant references.
Financial support by the Hungarian National Research, Development and Innovation Office – NKFIH (grant K124138) is gratefully acknowledged.  


\section{Multimonoidal structures in 2-categories}
\label{sec:Mpq}

We begin with a brief review of some definitions and basic constructions for later use. A more detailed introduction can be found e.g. in Chapter 7 of \cite{BorceuxI}.

A {\em 2-category} is a category enriched in the category of small categories and functors considered with the Cartesian product of categories as the monoidal product.
The explicit definition can be found in \cite[Definition 7.1.1]{BorceuxI}.
Throughout, horizontal composition in a 2-category (i.e. the composition functor) will be denoted by a lower dot $.$ and vertical composition (that is, the compositions of the hom categories) will be denoted by an upper dot $\updot$. All identity (1- or 2-) cells will be denoted by $1$.

\begin{example} \label{ex:Cat}
The most well-known 2-category is, perhaps, $\mathsf{Cat}$ \cite[Example 7.1.4.a]{BorceuxI}. 
Its 0-cells (or objects) are small categories, the 1-cells (i.e. objects of the hom categories) are the functors and the 2-cells (i.e. morphisms of the hom categories) are the natural transformations. 
\end{example}

For any 2-category $\mathcal M$ we denote the {\em horizontally opposite} 2-category in \cite[Section 4]{Street:FTM} by $\mathcal M_{\mathsf{op}}$ and we denote the {\em vertically opposite} 2-category in \cite[Section 4]{Street:FTM}  by $\mathcal M^\vop$. 

A {\em 2-functor} is a functor enriched in the category of small categories and functors considered with the Cartesian product of categories as the monoidal product. The explicit definition can be found in \cite[Definition 7.2.1]{BorceuxI}. Any 2-functor $\mathsf F:\mathcal M \to \mathcal N$ can be seen as a 2-functor $\mathsf F_{\mathsf{op}}:\mathcal M_{\mathsf{op}} \to \mathcal N_{\mathsf{op}}$ and also as a 2-functor $\mathsf F^\vop:\mathcal M^\vop \to \mathcal N^{\,\vop}$.

A {\em 2-natural transformation} is a natural transformation enriched in the category of small categories and functors considered with the Cartesian product of categories as the monoidal product.
The explicit definition can be found in \cite[Definition 7.2.2]{BorceuxI}. Any 2-natural transformation $\Theta:\mathsf F \to \mathsf G$ can be seen as a 2-natural transformation $\Theta_{\mathsf{op}}:\mathsf F_{\mathsf{op}} \to \mathsf G_{\mathsf{op}}$ and also as a 2-natural transformation $\Theta^\vop:\mathsf F^\vop \to \mathsf G^\vop$.

\begin{example}
Small 2-categories are the 0-cells, 2-functors are the 1-cells, and 2-natural transformations are the 2-cells of a 2-category $\mathsf{2Cat}$, see \cite[Proposition 7.2.3]{BorceuxI}. 
\end{example}

\begin{definition} \cite[page 69]{Jay}
A {\em strict monoidal 2-category} is a monoid in the category whose objects are the 2-categories, whose morphisms are the 2-functors and whose monoidal product is the Cartesian product of 2-categories. Explicitly, a strict monoidal 2-category consists of the following data.
\begin{itemize}
\item A 2-category $\mathcal M$.
\item A 2-functor $I$ from the singleton 2-category to $\mathcal M$. The image of the only object of the singleton category under it is called the {\em monoidal unit} and it is denoted by the same symbol $I$.
\item A 2-functor $\otimes: \mathcal M \times \mathcal M\to \mathcal M$ called the {\em monoidal product}. It must be strictly associative with the strict unit $I$. The action of the 2-functor $\otimes$ on (0-, 1- and 2-) cells will be denoted by juxtaposition.
\end{itemize}
\end{definition}

The 2-category $\mathsf{Cat}$ of Example \ref{ex:Cat} is strict monoidal via the Cartesian product of categories.

\begin{definition}
A {\em strict monoidal 2-functor} is a monoid morphism in the category whose objects are the 2-categories, whose morphisms are the 2-functors and whose monoidal product is the Cartesian product of 2-categories. Explicitly, a 2-functor $\mathsf F:\mathcal M \to \mathcal M'$ is strict monoidal whenever $\mathsf FI=I'$ and $\mathsf F.\otimes=\otimes' .(\mathsf  F \times \mathsf  F)$.
\end{definition}

\begin{definition}
A 2-natural transformation $\Theta: \mathsf F \to \mathsf G$ between strict monoidal 2-functors is said to be {\em monoidal} if the following diagrams of 2-natural transformations commute. 
$$
\xymatrix{
\otimes.(\mathsf F \times \mathsf F) \ar@{=}[r] \ar[d]_-{1.(\Theta \times \Theta)} &
\mathsf F.\otimes \ar[d]^-{\Theta.1} \\
\otimes.(\mathsf G \times \mathsf G) \ar@{=}[r]  &
\mathsf G.\otimes} \qquad 
\xymatrix@R=28pt{
I \ar@{=}[r] \ar@{=}[d] &
\mathsf F.I \ar[d]^-{\Theta.1} \\
I \ar@{=}[r]  &
\mathsf G.I}
$$
That is, for all objects $A$ and $B$, $\Theta_{AB}=\Theta_A \Theta_B$ and $\Theta_I=1$.
\end{definition}

For any 2-category $\mathcal M$, there is a 2-functor $\mathsf{flip}:\mathcal M \times \mathcal M \to \mathcal M \times \mathcal M$, sending a pair of 2-cells $(\omega,\vartheta)$ to $(\vartheta,\omega)$. It occurs in the following definition.

\begin{definition} \cite[Page 69]{Jay}
A {\em symmetric strict monoidal 2-category} consists of a strict monoidal 2-category $(\mathcal M,\otimes, I)$ together with a 2-natural transformation $\sigma:\otimes \to \otimes.\mathsf{flip}$ rendering commutative the following diagrams for all objects $A$, $B$ and $C$.
$$
\xymatrix{
AB \ar[r]^-\sigma \ar@{=}[rd] &
BA \ar[d]^-\sigma \\
&AB} \qquad \qquad
\xymatrix{
ABC \ar[r]^-{\sigma 1} \ar[rd]_-{\sigma_{12}} &
BAC \ar[d]^-{1\sigma} \\
&BCA}
$$
The subscripts $p$ and $q$ of $\sigma_{pq}$ indicate that it is a morphism $A_1 \dots  A_p B_1\dots  B_q \to B_1\dots  B_q $
$ A_1 \dots A_p$. 
\end{definition}

\begin{definition}
A strict monoidal 2-functor $\mathsf F$ between symmetric strict monoidal 2-categories is said to be {\em symmetric} if 
$\mathsf F (\xymatrix@C=13pt{AB\ar[r]^-\sigma &BA})=
\xymatrix@C=13pt{(\mathsf F A) (\mathsf F B)\ar[r]^-{\sigma'} & (\mathsf F B)(\mathsf F A)}$ for all objects $A$ and $B$ of the domain 2-category of $\mathsf F$.
\end{definition}

\begin{example} 
Symmetric strict monoidal small 2-categories are the 0-cells, symmetric strict monoidal 2-functors are the 1-cells and monoidal 2-natural transformations are the 2-cells of a 2-category $\mathsf{sm}\mbox{-}\mathsf{2Cat}$. Both vertical and horizontal compositions are given by the same formulae as in $\mathsf{2Cat}$ of \cite[Proposition 7.2.3]{BorceuxI}.
\end{example}

\begin{example} \label{ex:Mnd}
Below we describe a 2-functor $\mathsf{Mnd}:\mathsf{sm}\mbox{-}\mathsf{2Cat} \to \mathsf{sm}\mbox{-}\mathsf{2Cat}$.

For any 2-category $\mathcal M$, there is a 2-category $\mathsf{Mnd}(\mathcal M)$ as in \cite{Street:FTM}. Its 0-cells are the {\em monads} in $\mathcal M$. That is, quadruples consisting of a 0-cell $A$ of $\mathcal M$, a 1-cell $t:A \to A$ and 2-cells $\mu:t.t\to t$ and $\eta:1\to t$. They are required to satisfy the associativity condition $\mu \updot (\mu.1)=\mu\updot (1.\mu)$ and the unitality conditions $\mu \updot (\eta.1)=1=\mu\updot (1.\eta)$.
A 1-cell in $\mathsf{Mnd}(\mathcal M)$ from $(A,t,\mu,\eta)$ to $(A',t',\mu',\eta')$ is a pair consisting of a 1-cell $f:A\to A'$ in $\mathcal M$ and a 2-cell $\varphi:t'.f \to f.t$ satisfying the multiplicativity condition $(1.\mu)\updot (\varphi.1)\updot (1.\varphi)=\varphi\updot (\mu'.1)$ and the unitality condition $1.\eta=\varphi\updot(\eta'.1)$. Such a 1-cell is called a {\em monad morphism} in $\mathcal M$. The 2-cells $(f,\varphi)\to (g,\gamma)$ in $\mathsf{Mnd}(\mathcal M)$ are those 2-cells $\omega:f\to g$ in $\mathcal M$ for which $(\omega.1)\updot\varphi=\gamma\updot (1.\omega)$. They are called {\em monad transformations} in $\mathcal M$. The vertical composite of monad transformations is their vertical composite as 2-cells in $\mathcal M$. The horizontal composite of composable 1-cells $(f,\varphi)$ and $(g,\gamma)$ is $(g.f, (1.\varphi)\updot(\gamma.1))$. The horizontal composite of monad transformations is their horizontal composite as 2-cells in $\mathcal M$. Whenever $\mathcal M$ is equipped with a strict monoidal structure $(\otimes, I)$, there is an induced strict monoidal structure on $\mathsf{Mnd}(\mathcal M)$.
The monoidal unit is $(I,1,1,1)$, the identity monad on $I$. 
The monoidal product of monads $(A,t,\mu,\eta)$ and  $(A',t',\mu',\eta')$ is 
$$
(AA', 
\xymatrix@=16pt{AA' \ar[r]^-{tt'} & AA'}, 
\xymatrix@C=16pt{tt'.tt'=(t.t)(t.'t') \ar[r]^-{\mu\mu'} & tt'}, 
\xymatrix@C=16pt{I=II \ar[r]^-{\eta\eta'} & tt'}).
$$
The monoidal product of monad morphisms $(f,\varphi):(A,t)\to (B,s)$ and  $(f',\varphi'):(A',t')\to (B',s')$ is the monad morphism
$$
(\xymatrix@=16pt{AA' \ar[r]^-{ff'} & BB'}, 
\xymatrix@C=16pt{ss'.ff'=(s.f)(s.'f') \ar[r]^-{\varphi \varphi'} & (f.t)(f'.t')=ff'.tt'}).
$$ 
The monoidal product of monad transformations is their monoidal product as 2-cells of $\mathcal M$. If in addition a strict monoidal 2-category $(\mathcal M, \otimes, I)$ is equipped with a symmetry $\sigma$, then the induced symmetry on $\mathsf{Mnd}(\mathcal M)$ has the components $(\sigma, 1)$. So far we recalled the 0-cell part of the desired 2-functor $\mathsf{Mnd}$. 

For any 2-functor $\mathsf F:\mathcal M \to \mathcal N$ there is a 2-functor $\mathsf{Mnd}(\mathsf F):\mathsf{Mnd}(\mathcal M) \to \mathsf{Mnd}(\mathcal N)$. It sends a monad $(A,t,\mu,\eta)$ to 
$$
(FA,
\xymatrix@=16pt{\mathsf F A \ar[r]^-{\mathsf F t} & \mathsf F A}, 
\xymatrix@C=16pt{(\mathsf F t).(\mathsf F t)=\mathsf F(t.t) \ar[r]^-{\mathsf F\mu} & \mathsf F t}, 
\xymatrix@C=16pt{I=\mathsf F I \ar[r]^-{\mathsf F\eta} & \mathsf F t});
$$
it sends a monad morphism $(f,\varphi):(A,t)\to (B,s)$ to
$$
(\xymatrix@=16pt{\mathsf F A \ar[r]^-{\mathsf F f} & \mathsf F B}, 
\xymatrix@C=16pt{(\mathsf F s).(\mathsf F f)=\mathsf F(s.f) \ar[r]^-{\mathsf F \varphi } & 
\mathsf F(f.t)=(\mathsf F f).(\mathsf F t)}
$$ 
and it sends a monad transformation $\omega$ to $\mathsf F \omega$. Whenever $\mathsf F$ is strict monoidal then so is $\mathsf{Mnd}(\mathsf F)$ and if in addition $\mathsf F$ is symmetric then so is $\mathsf{Mnd}(\mathsf F)$. This yields the 1-cell part of the desired 2-functor $\mathsf{Mnd}$. 

For any 2-natural transformation $\Theta$ there is a 2-natural transformation $\mathsf{Mnd}(\Theta)$ with components $(\Theta,1)$. It is strict monoidal whenever $\Theta$ is so. This finishes the description of the stated 2-functor $\mathsf{Mnd}$. (It would be straightforward to extend it also to modifications but it is not needed for the purposes of the present paper.)

Symmetrically to the above considerations, we introduce another 2-functor 
$\mathsf{Mnd}_{\mathsf{op}}:=$
$\mathsf{Mnd}((-)_{\mathsf{op}})_{\mathsf{op}}:
\mathsf{sm}\mbox{-}\mathsf{2Cat} \to \mathsf{sm}\mbox{-}\mathsf{2Cat}$.
\end{example}

\begin{example} \label{ex:Mpq}
(1) Below a 2-functor $(-)_{01}:\mathsf{sm}\mbox{-}\mathsf{2Cat} \to \mathsf{sm}\mbox{-}\mathsf{2Cat}$ is described.

For any strict monoidal 2-category $(\mathcal M,\otimes, I)$, there is a 2-category $\mathcal M_{01}$. Its 0-cells are the {\em pseudomonoids} in $\mathcal M$, also called {\em monoidal objects} or {\em monoidales} by other authors, see \cite{Zawadowski} and \cite{ChikhladzeLackStreet}, respectively. A pseudomonoid in $\mathcal M$ consists of a 0-cell $A$, 1-cells $m:AA\to A$ and $u:I\to A$ together with invertible 2-cells $\alpha:m.m1\to m.1m$, $\lambda:m.u1\to 1$ and $\varrho:m.1u\to 1$ satisfying MacLane's pentagon and triangle conditions
$$
\xymatrix{
m.m1.m11\ar[r]^-{\alpha.1} \ar[d]_-{1.\alpha 1} &
m.1m.m11\ar@{=}[r] &
m.m1.11m \ar[d]^-{\alpha.1} \\
m.m1.1m1 \ar[r]_-{\alpha.1} &
m.1m.1m1 \ar[r]_-{1.1\alpha} &
m.1m.11m} \qquad
\xymatrix{
m.m1.1u1 \ar[r]^-{\alpha.1} \ar[rd]_-{1.\varrho 1} &
m.1m.1u1 \ar[d]^-{1.1\lambda} \\
& \, m.}
$$
A 1-cell in $\mathcal M_{01}$ is a so-called {\em opmonoidal 1-cell}. An opmonoidal 1-cell from $(A,m,u,\alpha,$ $\lambda,\varrho)$ to  $(A',m',u',\alpha',\lambda',\varrho')$ consists of a 1-cell $f:A \to A'$ in $\mathcal M$ together with 2-cells $\varphi^2:f.m\to m'.ff$ (called the {\em binary part}) and $\varphi^0:f.u \to u'$ (called the {\em nullary part}) subject to the coassociativity and counitality conditions
$$
\xymatrix{
f.m.m1 \ar[r]^-{1.\alpha} \ar[d]_-{\varphi^2.1} &
f.m.1m \ar[d]^-{\varphi^2.1} \\
m'.ff.m1 \ar[d]_-{1.\varphi^21} &
m'.ff.1m \ar[d]^-{1.1\varphi^2} \\
m'.m'1.fff \ar[r]_-{\alpha'.1 } &
m'.1m'.fff} \qquad
\xymatrix{
f.m.u1 \ar[r]^-{1.\lambda} \ar[d]_-{\varphi^2.1} &
f \ar@{=}[dd] &
f.m.1u \ar[l]_-{1.\varrho}  \ar[d]^-{\varphi^2.1} \\
m'.ff.u1 \ar[d]_-{1.\varphi^01} &&
m'.ff.1u \ar[d]^-{1.1\varphi^0} \\
m'.u'1.f \ar[r]_-{\lambda'.1 } &
f &
\ m'.1u'.f\ . \ar[l]^-{\varrho'.1 } }
$$
The 2-cells $\omega:(f,\varphi^2,\varphi^0)\to (g,\gamma^2,\gamma^0)$ in $\mathcal M_{01}$ are the {\em opmonoidal 2-cells} in $\mathcal M$; that is, those 2-cells $\omega:f\to g$ in $\mathcal M$ for which the following diagrams commute.
$$
\xymatrix{
f.m \ar[r]^-{\varphi^2} \ar[d]_-{\omega.1} &
m'.ff \ar[d]^-{1.\omega\omega} \\
g.m \ar[r]_-{\gamma^2} &
m'.gg} \qquad \qquad
\xymatrix{
f.u \ar[r]^-{\varphi^0} \ar[d]_-{\omega.1} &
u' \ar@{=}[d] \\
g.u \ar[r]_-{\gamma^0} &
u'} 
$$
The vertical composite of opmonoidal 2-cells is their vertical composite as 2-cells in $\mathcal M$. The horizontal composite of 1-cells $(f,\varphi^2,\varphi^0):(A,m,u,\alpha,\lambda,\varrho)\to (A',m',u',\alpha',\lambda',\varrho')$ and $(g,\gamma^2,\gamma^0):(A',m',u',\alpha',\lambda',\varrho')\to (A^{\prime\prime},m^{\prime\prime},u^{\prime\prime},\alpha^{\prime\prime},\lambda^{\prime\prime},\varrho^{\prime\prime})$ is
$$
(g.f, 
\xymatrix@C=16pt{
g.f.m\ar[r]^-{1.\varphi^2} & 
g.m'.ff \ar[r]^-{\gamma^2.1} &
m^{\prime\prime}.gg.ff=m^{\prime\prime}.(g.f)(g.f)},
\xymatrix@C=16pt{
g.f.u\ar[r]^-{1.\varphi^0} & 
g.u' \ar[r]^-{\gamma^0} &
u^{\prime\prime}}).
$$
The horizontal composite of opmonoidal 2-cells is their horizontal composite as 2-cells in $\mathcal M$. 

If the strict monoidal 2-category $(\mathcal M,\otimes,I)$ is equipped with a symmetry $\sigma$, then there is an induced symmetric strict monoidal structure on $\mathcal M_{01}$. The monoidal product of pseudomonoids $(A,m,u,\alpha,\lambda,\varrho)$ and $(A',m',u',\alpha',\lambda',\varrho')$ is 
\begin{eqnarray*}
&(& AA',
\xymatrix@C=20pt{
AA'AA' \ar[r]^-{1\sigma1} &
AAA'A'\ar[r]^-{mm'} &
AA'},
\xymatrix@C=20pt{I=II \ar[r]^-{uu'} &
AA'},\\
&&\xymatrix@C=25pt@R=5pt{
mm'.1\sigma1.mm'11.1\sigma 111\ar@{=}[d] &
mm'.1\sigma 1.11mm'.111\sigma 1 \ar@{=}[d] \\
mm'.m1m'1.11\sigma_{21}1.1\sigma111\ar[r]^-{\alpha\alpha'.1.1} &
mm'.1m1m'.1\sigma_{12}11.111\sigma 1,} \\
&&\xymatrix@C=16pt@R=5pt{
mm'.1\sigma1.uu'11=
mm'.u1u'1 \ar[r]^-{\lambda\lambda'} &
1}, 
\xymatrix@C=16pt@R=5pt{
mm'.1\sigma1.11uu'=
mm'.1u1u'\ar[r]^-{\varrho\varrho'} &
1}).
\end{eqnarray*}
The monoidal product of 1-cells $(f,\varphi^2,\varphi^0):(A,m,u,\alpha,\lambda,\varrho)\to (A',m',u',\alpha',\lambda',\varrho')$ and $(g,\gamma^2,\gamma^0):(B,m,u,\alpha,\lambda,\varrho)\to (B',m',u',\alpha',\lambda',\varrho')$ is 
$$
(\!\!\xymatrix@C=16pt{
AB \ar[r]^-{fg} & A'B'}
\!\!,\!\!
\xymatrix@C=25pt{
fg.mm.1\sigma 1\ar[r]^-{\varphi^2\gamma^2.1} &
m'm'.ffgg.1\sigma 1=
m'm'.1\sigma 1.fgfg}\!\!,\!\!
\xymatrix@C=20pt{
fg.uu\ar[r]^-{\varphi^0\gamma^0} &
u'u'}\!\!).
$$
The monoidal product of opmonoidal 2-cells is their monoidal product as 2-cells in $\mathcal M$.
The symmetry has the components $(\sigma,1, 1)$.
This gives the object map of a  2-functor $(-)_{01}:\mathsf{sm}\mbox{-}\mathsf{2Cat} \to \mathsf{sm}\mbox{-}\mathsf{2Cat}$.

A strict monoidal 2-functor $\mathsf F:\mathcal M \to \mathcal N$ induces a 2-functor $\mathsf F_{01}:\mathcal M_{01} \to \mathcal N_{01}$. It takes a 0-cell of $\mathcal M_{01}$ --- that is, a pseudomonoid $(A,m,u,\alpha,\lambda,\varrho)$ in $\mathcal M$ --- to the pseudomonoid in $\mathcal N$,
\begin{eqnarray*}
&(&\mathsf F A,
\xymatrix@C=16pt{
(\mathsf F A)(\mathsf F A)=\mathsf F (AA) \ar[r]^-{\mathsf F m} & \mathsf F A},
\xymatrix@C=16pt{
I=\mathsf FI\ar[r]^-{\mathsf F u} & \mathsf F A},\\
&&\xymatrix@C=16pt{
\mathsf F m.(\mathsf F m)1=\mathsf F (m.m1) \ar[r]^-{\mathsf F \alpha} &
\mathsf F(m.1m)=
\mathsf F m.1(\mathsf F m)},\\
&&\xymatrix@C=16pt{
\mathsf F m.(\mathsf F u)1= \mathsf F (m.u1) \ar[r]^-{\mathsf F \lambda} &
\mathsf F1 = 1},
\xymatrix@C=16pt{
\mathsf F m.1(\mathsf F u)= \mathsf F (m.1u) \ar[r]^-{\mathsf F \varrho} &
\mathsf F1 = 1}).
\end{eqnarray*}
The 2-functor $\mathsf F_{01}$ sends an opmonoidal 1-cell $(f,\varphi^2,\varphi^0)$ to 
$$
(\mathsf Ff,
\xymatrix@C=16pt{
\mathsf F m'.(\mathsf F f)(\mathsf F f)=\mathsf F (m'.ff) \ar[r]^-{\mathsf F \varphi^2} &
\mathsf F(f.m)=\mathsf F f.\mathsf F m},
\xymatrix@C=16pt{
\mathsf F u' \ar[r]^-{\mathsf F \varphi^0} &
\mathsf F(f.u)=\mathsf F f.\mathsf F u}).
$$
Finally, $\mathsf F_{01}$ sends an opmonoidal 2-cell $\omega$ to $\mathsf F \omega$.
If the strict monoidal 2-functor $\mathsf F$ is also symmetric, then $\mathsf F_{01}$ is strict monoidal and symmetric.
This gives the 1-cell part of a  2-functor $(-)_{01}:\mathsf{sm}\mbox{-}\mathsf{2Cat} \to \mathsf{sm}\mbox{-}\mathsf{2Cat}$.

For a monoidal 2-natural transformation $\Theta:\mathsf F \to \mathsf G$, there is a 2-natural transformation $\Theta_{01}:\mathsf F_{01} \to \mathsf G_{01}$ with the component at a pseudomonoid $(A,m,u,\alpha,\lambda,\varrho)$ in $\mathcal M$ 
$$
(\Theta_A,
\Theta_A.\mathsf F m = \mathsf G m.\Theta_{AA}=\mathsf G m.\Theta_A\Theta_A,
\Theta_A.\mathsf F u = \mathsf G u.\Theta_I=\mathsf G u).
$$
It is strict monoidal thanks to the strict monoidality of $\Theta$.
These maps define 
\begin{enumerate}[(i)]
\item a 2-functor $(-)_{01}$ from the 2-category of strict monoidal 2-categories, strict monoidal 2-functors and monoidal 2-natural transformations to $\mathsf{2Cat}$,
\item a 2-functor $(-)_{01}:\mathsf{sm}\mbox{-}\mathsf{2Cat} \to \mathsf{sm}\mbox{-}\mathsf{2Cat}$.
\end{enumerate}
(It would be easy to extend them to monoidal modifications --- defined in the evident way --- but this is not needed for the purposes of the present paper.)

(2) There are symmetrically constructed 2-functors $(-)_{10}:=(((-)^\vop)_{01})^\vop$ of both kinds in items (i) and (ii) of the list in part (1) above.
They send a strict monoidal 2-category $\mathcal M$ to the 2-category 
$\mathcal M_{10}=((\mathcal M^{\vop})_{01})^\vop$ whose 0-cells are again the pseudomonoids in $\mathcal M$. Its 1-cells and 2-cells are known as monoidal 1-cells and monoidal 2-cells in $\mathcal M$, respectively.

(3) Next we show that the 2-endofunctors on $\mathsf{sm}\mbox{-}\mathsf{2Cat}$ in parts (1) and (2) commute up-to an irrelevant 2-natural isomorphism (this extends a claim in  \cite[Proposition 6.75]{AguiarMahajan}). 

First we compare the symmetric strict monoidal 2-categories $(\mathcal M_{10})_{01}$ and $(\mathcal M_{01})_{10}$ for a symmetric strict monoidal 2-category $\mathcal M$.
A 0-cell of $(\mathcal M_{10})_{01}$ is a pseudomonoid in $\mathcal M_{10}$. As such, it consists of the following data.
\begin{itemize}
\item A $0$-cell of $\mathcal M_{10}$; that is, a pseudomonoid $(A,m^\verti,u^\verti,\alpha^\verti,\lambda^\verti,\varrho^\verti)$ in $\mathcal M$.
\item Multiplication and unit 1-cells in $\mathcal M_{10}$. That is, monoidal 1-cells in $\mathcal M$\\
$(m^\hori,\xi,\xi^0):(AA,m^\verti m^\verti.1\sigma1,u^\verti u^\verti,\alpha^\verti\alpha^\verti .1.1,\lambda^\verti \lambda^\verti,\varrho^\verti\varrho^\verti) \to (A,m^\verti,u^\verti,\alpha^\verti,\lambda^\verti,\varrho^\verti)$ 
and \\
$(u^\hori,\xi_0,\xi^0_0): (I,1,1,1,1,1) \to (A,m^\verti,u^\verti,\alpha^\verti,\lambda^\verti,\varrho^\verti)$.
\item Associativity and unitality 2-cells in $\mathcal M_{10}$.  They are invertible monoidal 2-cells in $\mathcal M$, $\alpha^\hori:m^\hori.m^\hori1 \to m^\hori.1m^\hori$, $\lambda^\hori:m^\hori.u^\hori1\to 1$ and $\varrho^\hori:m^\hori.1u^\hori\to 1$ satisfying MacLane's pentagon and triangle conditions.
\end{itemize}

Symmetrically, a 0-cell of $(\mathcal M_{01})_{10}$ is a pseudomonoid in $\mathcal M_{01}$, which consists of the following data.
\begin{itemize}
\item A $0$-cell of $\mathcal M_{01}$; that is, a pseudomonoid $(A,m^\hori,u^\hori,\alpha^\hori,\lambda^\hori,\varrho^\hori)$ in $\mathcal M$.
\item Multiplication and unit 1-cells in $\mathcal M_{01}$. That is, opmonoidal 1-cells in $\mathcal M$\\
$(m^\verti,\xi.1,\xi_0)\!:\!(AA,m^\hori m^\hori.1\sigma1,u^\hori u^\hori\!,\alpha^\hori\alpha^\hori.1.1,\lambda^\hori \lambda^\hori\!,\varrho^\hori\varrho^\hori)\! \to \! (A,m^\hori\!,u^\hori\!,\alpha^\hori\!,\lambda^\hori\!,\varrho^\hori)$, \\
$(u^\verti,\xi^0,\xi^0_0): (I,1,1,1,1,1) \to (A,m^\hori,u^\hori,\alpha^\hori,\lambda^\hori,\varrho^\hori)$.
\item Associativity and unitality 2-cells in $\mathcal M_{01}$.  They are invertible opmonoidal 2-cells in $\mathcal M$, $\alpha^\verti:m^\verti.m^\verti1 \to m^\verti.1m^\verti$, $\lambda^\verti:m^\verti.u^\verti 1\to 1$ and $\varrho^\verti:m^\verti.1u^\verti\to 1$ satisfying MacLane's pentagon and triangle conditions.
\end{itemize}

Both of these sets of data amount to two pseudomonoid structures $(A,m^\verti,u^\verti,\alpha^\verti,\lambda^\verti,\varrho^\verti)$ and $(A,m^\hori,u^\hori,\alpha^\hori,\lambda^\hori,\varrho^\hori)$ on the same object $A$ together with 2-cells in $\mathcal M$
$$
\xymatrix@C=16pt{
m^\verti.m^\hori m^\hori \ar[r]^-\xi & m^\hori.m^\verti m^\verti.1\sigma 1}\qquad
\xymatrix@C=16pt{
u^\verti \ar[r]^-{\xi^0} & m^-.u^\verti u^\verti}\qquad 
\xymatrix@C=16pt{
m^\verti.u^\hori u^\hori \ar[r]^-{\xi_0} & u^\hori} \qquad
\xymatrix@C=16pt{
u^\verti \ar[r]^-{\xi^0_0} & u^\hori}
$$
(giving rise to the 2-cell $\xi.1:m^\verti.m^\hori m^\hori.1\sigma1 \to m^\hori.m^\verti m^\verti.1\sigma1.1\sigma1=m^\hori.m^\verti m^\verti$).
They are subject to conditions as follows.
\begin{itemize}
\item Associativity and unitality of the monoidal 1-cell $(m^\hori,\xi,\xi^0)$; equivalently, compatibility of the opmonoidal 2-cells $\alpha^\verti$, $\lambda^\verti$ and $\varrho^\verti$ with the binary parts of the opmonoidal  structures of their source and target 1-cells.
\item Associativity and unitality of the monoidal 1-cell $(u^-,\xi_0,\xi^0_0)$; equivalently, compatibility of the opmonoidal 2-cells $\alpha^\verti$, $\lambda^\verti$ and $\varrho^\verti$ with the nullary parts of the opmonoidal  structures of their source and target 1-cells.
\item Compatibility of the monoidal 2-cells $\alpha^\hori$, $\lambda^\hori$ and $\varrho^\hori$  with the binary parts of the monoidal  structures of their source and target 1-cells; equivalently, coassociativity and counitality of the opmonoidal 1-cell $(m^\verti,\xi.1,\xi_0)$.
\item Compatibility of the monoidal 2-cells $\alpha^\hori$, $\lambda^\hori$ and $\varrho^\hori$ with the nullary parts of the monoidal  structures of their source and target 1-cells; equivalently, coassociativity and counitality of the opmonoidal 1-cell $(u^\verti,\xi^0,\xi^0_0)$.
\end{itemize}
Such a datum is called a {\em duoidal} or {\em 2-oidal} object of $\mathcal M$.

A 1-cell in either one of the 2-categories $(\mathcal M_{10})_{01}$ and $(\mathcal M_{01})_{10}$ consists of a 1-cell $f:A\to A'$ in $\mathcal M$ together with a monoidal structure 
$(\varphi_2:m^{\verti\prime }.ff \to f.m^\verti,\varphi_0:u^{\verti\prime } \to f.u^\verti)$ and an opmonoidal structure 
$(\varphi^2:f.m^\hori \to m^{\hori\,\prime }.ff,\varphi^0:f.u^\hori \to u^{\hori\,\prime })$
subject to the following compatibility conditions.
\begin{itemize}
\item The monoidal 2-cell $\varphi^2:(f,\varphi_2,\varphi_0).(m^\hori,\xi,\xi^0)\to (m^{\hori\, \prime},\xi',\xi^{0\prime}).(ff,\varphi_2\varphi_2.1,$ $\varphi_0\varphi_0)$ is compatible with the binary parts of the monoidal  structures of its source and target 1-cells; equivalently, the opmonoidal 2-cell $\varphi_2:(m^{\verti\prime},\xi'.1,\xi'_0).$ 
$(ff,$ $\varphi^2\varphi^2.1,\varphi^0\varphi^0) \to (f,\varphi^2,\varphi^0).(m^\verti,\xi.1,\xi_0)$ is compatible with the binary parts of the opmonoidal  structures of its source and target 1-cells.
\item The monoidal 2-cell $\varphi^2$ is compatible with the nullary parts of the monoidal  structures of its source and target 1-cells; equivalently, the opmonoidal 2-cell $\varphi_0:(u^{\verti\, \prime},\xi^{0\prime},$
$\xi^{0\prime}_0)\to (f,\varphi^2,\varphi^0).(u^\verti,\xi^0,\xi^0_0)$ is compatible with the binary parts of the opmonoidal  structures of its source and target 1-cells.
\item The monoidal 2-cell $\varphi^0:(f,\varphi_2,\varphi_0).(u^\hori,\xi_0,\xi^0_0)\to (u^{\hori\,\prime},\xi'_0,\xi^{0\prime}_0)$ is compatible with the binary parts of the monoidal  structures of its source and target 1-cells; equivalently, the opmonoidal 2-cell $\varphi_2$ is compatible with the nullary parts of the opmonoidal  structures of its source and target 1-cells.
\item The monoidal 2-cell $\varphi^0$  is compatible with the nullary parts of the monoidal  structures of its source and target 1-cells; equivalently, the opmonoidal 2-cell $\varphi_0$ is compatible with the nullary parts of the opmonoidal  structures of its source and target 1-cells.
\end{itemize}
Such a datum is termed a {\em $(1,1)$-oidal} 1-cell. In the particular case of $\mathcal M=\mathsf{Cat}$, in \cite{AguiarHaimLopezFranco} it was called a {\em bimonoidal functor}.

Finally, a 2-cell $(f,\varphi_2,\varphi_0,\varphi^2,\varphi^0) \to (g,\gamma_2,\gamma_0,\gamma^2,\gamma^0)$ in either one of the 2-categories $(\mathcal M_{10})_{01}$ and $(\mathcal M_{01})_{10}$ is a 2-cell in $\mathcal M$ which is both 
\begin{itemize}
\item
a monoidal 2-cell $(f,\varphi_2,\varphi_0)\to (g,\gamma_2,\gamma_0)$ and 
\item
an opmonoidal 2-cell $(f,\varphi^2,\varphi^0)\to (g,\gamma^2,\gamma^0)$.
\end{itemize}
We say that it is a {\em $(1,1)$-oidal} 2-cell in $\mathcal M$. 

This proves the isomorphism of the 2-categories $(\mathcal M_{10})_{01}$ and $(\mathcal M_{01})_{10}$. The obtained isomorphism is clearly symmetric strict monoidal and 2-natural.  

(4) Thanks to their commuting verified in item (3) above, we may apply the 2-functor in part (1) $q$ times and the 2-functor in part (2) $p$ times in an arbitrary order, for any non-negative integers $q$ and $p$. Thereby we obtain a 2-functor $(-)_{pq}:\mathsf{sm}\mbox{-}\mathsf{2Cat} \to \mathsf{sm}\mbox{-}\mathsf{2Cat}$ (so that $(-)_{00}$ is the identity 2-functor).
In particular, it takes any symmetric strict monoidal 2-category $\mathcal M$ to a symmetric strict monoidal 2-category $\mathcal M_{pq}$. Recall that a 0-cell in $\mathcal M_{p+1,q}\cong (\mathcal M_{pq})_{10}$ is the same as a 0-cell in $\mathcal M_{p,q+1}\cong (\mathcal M_{pq})_{01}$; namely, a pseudomonoid in $\mathcal M_{pq}$. Thus the notion of 0-cell in $\mathcal M_{p,n-p}$ only depends on the non-negative integer $n$ but not on the integer $0\leq p \leq n$. (For the particular 2-category $\mathcal M=\mathsf{Cat}$ this was discussed in Proposition 7.49 and Remark 7.51 of \cite{AguiarMahajan}.)
\end{example}


\section{$(0,q)$-oidal monads}
\label{sec:0q-monad}

The aim of this section is to prove the commutativity (up-to 2-natural isomorphism) of the diagram
\begin{equation} \label{eq:Mpq}
\xymatrix{
\mathsf{sm}\mbox{-}\mathsf{2Cat} \ar[r]^-{(-)_{01}} \ar[d]_-{\mathsf{Mnd}} &
\mathsf{sm}\mbox{-}\mathsf{2Cat}  \ar[d]^-{\mathsf{Mnd}} \\
\mathsf{sm}\mbox{-}\mathsf{2Cat} \ar[r]_-{(-)_{01}}  &
\mathsf{sm}\mbox{-}\mathsf{2Cat} }
\end{equation}
of the 2-functors of Example \ref{ex:Mnd} and Example \ref{ex:Mpq}~(1). Its iteration yields 2-natural isomorphisms 
$\mathsf{Mnd}(\mathcal M_{0q})\cong \mathsf{Mnd}(\mathcal M)_{0q}$ for any symmetric strict monoidal 2-category $\mathcal M$ and any non-negative integer $q$. 
A 0-cell in $\mathsf{Mnd}(\mathcal M_{0q})\cong \mathsf{Mnd}(\mathcal M)_{0q}$ is called a {\em $(0,q)$-oidal monad}.

For 2-endofunctors on the 2-category $\mathsf{2Cat}_\times$ of {\em Cartesian monoidal} 2-categories (instead of our $\mathsf{sm}\mbox{-}\mathsf{2Cat}$), commutativity of the analogous diagram was proved  (by actually the same steps) in \cite[Lemma 3.1]{Zawadowski}. A possible generalization, discussed here practically without any additional technical difficulty, is proposed in the introduction and the concluding remarks of \cite{Zawadowski}.

In order to prove the commutativity of \eqref{eq:Mpq} (up-to 2-natural isomorphism), we need to compare first the actions of the 2-functors around it on an arbitrary 0-cell; that is, symmetric strict monoidal 2-category $\mathcal M$.
A 0-cell in $\mathsf{Mnd}(\mathcal M_{01})$ is a monad in $\mathcal M_{01}$ thus it consists of the following data.
\begin{itemize}
\item A 0-cell of $\mathcal M_{01}$; that is, a pseudomonoid $(A,m,u,\alpha,\lambda,\varrho)$ in $\mathcal M$.
\item A 1-cell $(A,m,u,\alpha,\lambda,\varrho) \to (A,m,u,\alpha,\lambda,\varrho)$ in $\mathcal M_{01}$; that is, an opmonoidal 1-cell $(t:A\to A, \tau^2:t.m\to m.tt, \tau^0:t.u\to u)$ in $\mathcal M$.
\item Multiplication and unit 2-cells in $\mathcal M_{01}$; that is, opmonoidal 2-cells $\mu:(t,\tau^2,\tau^0).$ $(t,\tau^2,\tau^0)\to (t,\tau^2,\tau^0)$ and $\eta:(1,1,1)\to (t,\tau^2,\tau^0)$ in $\mathcal M$ satisfying the associativity and unit conditions. 
\end{itemize}
On the other hand, a 0-cell in $\mathsf{Mnd}(\mathcal M)_{01}$ is a pseudomonoid in $\mathsf{Mnd}(\mathcal M)$ thus it consists of the following data.
\begin{itemize}
\item A 0-cell in $\mathsf{Mnd}(\mathcal M)$; that is, a monad $(A,t,\mu,\eta)$ in $\mathcal M$.
\item Multiplication and unit 1-cells in $\mathsf{Mnd}(\mathcal M)$; that is, monad morphisms $(m:AA\to A,\tau^2:t.m\to m.tt)$ and $(u:I\to A,\tau^0:t.u\to u)$ in $\mathcal M$.
\item Associativity and unitality 2-cells in $\mathsf{Mnd}(\mathcal M)$; that is, invertible monad transformations 
$\alpha:(m.m1,(1.\tau^21)\updot(\tau^2.1)) \to (m.1m,(1.1\tau^2)\updot(\tau^2.1))$,
$\lambda:(m.u1,$ $ (1.\tau^01)\updot(\tau^2.1))$
$\to (1,1)$ and
$\varrho:(m.1u, (1.1\tau^0)\updot(\tau^2.1))\to (1,1)$
satisfying MacLane's pentagon and triangle conditions.
\end{itemize}
Both sets of data above amount to a pseudomonoid $(A,m,u,\alpha,\lambda,\varrho)$, an opmonoidal 1-cell $(t,\tau^2,\tau^0):(A,m,u,\alpha,\lambda,\varrho) \to (A,m,u,\alpha,\lambda,\varrho)$ and a monad $(A,t,\mu,\eta)$ in $\mathcal M$ with common 1-cell part $t$; subject to the following compatibility conditions.
$$
\xymatrix@C=20pt{
t.t.m\ar[r]^-{1.\tau^2} \ar[d]_-{\mu.1} &
t.m.tt \ar[r]^-{\tau^2.1} &
m.tt.tt \ar[d]^-{1.\mu\mu} \\
t.m \ar[rr]_-{\tau^2} &&
m.tt}\ 
\xymatrix@C=20pt{
t.t.u\ar[r]^-{1.\tau^0} \ar[d]_-{\mu.1} &
t.u \ar[r]^-{\tau^0} &
u \ar@{=}[d] \\
t.u \ar[rr]_-{\tau^0} &&
u}\ 
\xymatrix@C=20pt{
m\ar@{=}[r] \ar[d]_-{\eta.1} &
m\ar[d]^-{1.\eta\eta} \\
t.m \ar[r]_-{\tau^2} &
m.tt}\ 
\xymatrix@C=20pt{
u\ar@{=}[r] \ar[d]_-{\eta.1} &
u \ar@{=}[d] \\
t.u \ar[r]_-{\tau^0} &
u}
$$
The first two diagrams express the opmonoidality of the 2-cell $\mu$; equivalently, the multiplicativity of the monad morphisms $(m,\tau^2)$ and $(u,\tau^0)$, respectively.
The last two diagrams express the  opmonoidality of the 2-cell $\eta$; equivalently, the unitality of the monad morphisms $(m,\tau^2)$ and $(u,\tau^0)$, respectively.
This structure is termed a {\em $(0,1)$-oidal} or {\em opmonoidal monad}.

Both in $\mathsf{Mnd}(\mathcal M_{01})$ and $\mathsf{Mnd}(\mathcal M)_{01}$ a 1-cell is an opmonoidal 1-cell $(f,\varphi^2,\varphi^0):(A,m,u,$
$\alpha,\lambda,\varrho)\to (A',m',u',\alpha',\lambda',\varrho')$ and a monad morphism $(f,\Phi):(A,t,\mu,\eta)\to (A',t',\mu',\eta')$ in $\mathcal M$ with common 1-cell part $f$ such that the following compatibility conditions hold.
$$
\xymatrix{
t'.f.m \ar[d]_-{\Phi.1} \ar[r]^-{1.\varphi^2} &
t'.m'.ff \ar[r]^-{\tau^{\prime 2}.1}  &
m'.t't'.ff \ar[d]^-{1.\Phi\Phi} \\
f.t.m \ar[r]_-{1.\tau^2} &
f.m.tt \ar[r]_-{\varphi^2.1} &
m'.ff.tt} \qquad\qquad
\xymatrix{
t'.f.u \ar[d]_-{\Phi.1} \ar[r]^-{1.\varphi^0} &
t'.u' \ar[r]^-{\tau^{\prime 0}} &
u' \ar@{=}[d] \\
f.t.u \ar[r]_-{1.\tau^0}  &
f.u \ar[r]_-{\varphi^0} &
u'}
$$
The first diagram expresses the compatibility of the opmonoidal 2-cell $\Phi:(t',\tau^{\prime 2},\tau^{\prime 0}).$ $(f,\varphi^2,\varphi^0)\to (f,\varphi^2,\varphi^0).(t,\tau^2,\tau^0)$ with the binary parts of the opmonoidal structures of its source and target 1-cells; equivalently, the requirement that $\varphi^2:(f,\Phi).(m,\tau^2) \to (m',\tau^{\prime 2}).(ff,\Phi\Phi)$ is a monad transformation.
The second diagram expresses the compatibility of the opmonoidal 2-cell $\Phi$ with the nullary parts of the opmonoidal structures of its source and target 1-cells; equivalently, the requirement that $\varphi^0:(f,\Phi).(u,\tau^0) \to (u',\tau^{\prime 0})$ is a monad transformation.

Finally, both in $\mathsf{Mnd}(\mathcal M_{01})$ and $\mathsf{Mnd}(\mathcal M)_{01}$ a 2-cell is a 2-cell in $\mathcal M$ which is both opmonoidal and a monad transformation.

With this we established an isomorphism between the 2-categories $\mathsf{Mnd}(\mathcal M_{01})$ and $\mathsf{Mnd}(\mathcal M)_{01}$. It is clearly symmetric strict monoidal and it is straightforward to see its 2-naturality.


\section{$q$-oidal Eilenberg--Moore objects}
\label{sec:q-EM}

For a symmetric strict monoidal 2-category $\mathcal M$, any symmetric strict monoidal 2-functor $\mathsf H:\mathsf{Mnd}(\mathcal M)\to \mathcal M$ induces a symmetric strict monoidal 2-functor
\begin{equation} \label{eq:H0q}
\xymatrix{
\mathsf{Mnd}(\mathcal M_{0q})\cong \mathsf{Mnd}(\mathcal M)_{0q}
\ar[r]^-{\mathsf H_{0q}} &
\mathcal M_{0q}}
\end{equation}
whose object map sends a $(0,q)$-oidal monad to a $q$-oidal object of $\mathcal M$. 
The aim of this section is to investigate when the Eilenberg-Moore construction in $\mathcal M$ yields such a symmetric strict monoidal 2-functor $\mathsf H:\mathsf{Mnd}(\mathcal M)\to \mathcal M$; hence the $q$-oidal structure of the base object of a $(0,q)$-oidal monad in $\mathcal M$ lifts to the Eilenberg-Moore object along the `forgetful' 1-cells (see \eqref{eq:EMlift}). The results of this section  extend \cite[Theorem 8.2]{AguiarHaimLopezFranco} and place it in a broader context. We apply analogous ideas to those in \cite{Zawadowski}. In the particular case when the strict monoidal structure of $\mathcal M$ is Cartesian, 
a stronger result --- proving also that $\mathsf H_{0q}$ provides Eilenberg-Moore construction on $\mathcal M_{0q}$; cf. Remark \ref{rem:EM_10} --- 
was obtained in \cite[Theorem 5.1]{Zawadowski}.

The Eilenberg-Moore construction for 2-categories was shortly recalled from \cite{Street:FTM} in the Introduction. Namely, a 2-category $\mathcal M$ is said to admit Eilenberg-Moore construction if the inclusion 2-functor $\mathcal M\to \mathsf{Mnd}(\mathcal M)$ possesses a right 2-adjoint $\mathsf{H}$. The image of a monad $(A,t)$ under $\mathsf{H}$ is denoted by $A^t$ and it is called the Eilenberg-Moore object of the monad.
The component of the counit of this 2-adjunction at any monad $(A,t)$ is necessarily of the form $(u^t,1.\epsilon^t:t.u^t = u^t.f^t.u^t \to u^t):(A^t,1) \to (A,t)$, where the 1-cell $u^t:A^t\to A$ has a left adjoint $f^t$ in $\mathcal M$ and $\epsilon^t:f^t.u^t\to 1$ is the counit of this adjunction, see \cite[pages 151-153]{Street:FTM}. The adjunction $f^t\dashv u^t:A^t \to A$ generates the monad $(A,t)$ in the sense that the 1-cell $u^t.f^t$ is equal to $t$, the 2-cell $1.\epsilon^t.1:f^t.u^t.f^t.u^t \to f^t.u^t$ is equal to the multiplication of the monad while the unit $\eta^t:1\to u^t.f^t$ of the adjunction is equal to the unit of the monad, see \cite[Theorem 2]{Street:FTM}. 
By \cite[Theorem 3]{Street:FTM}, for any other adjunction $l\dashv r:B\to A$ generating the same monad $(A,t)$, there is a unique comparison 1-cell $k$ rendering commutative the first diagram of
\begin{equation}\label{eq:comp-counit}
\xymatrix{
A \ar[r]^-l \ar[d]_-{f^t} &
B \ar[d]^-r \ar@{-->}[ld]_-k \\
A^t\ar[r]_-{u^t} &
A}\qquad \qquad \qquad
\xymatrix{
k.l.r\ar[r]^-{1.\epsilon} \ar@{=}[d] &
k \ar@{=}[d] \\
f^t.u^t.k \ar[r]_-{\epsilon^t.1} &
\ k .}
\end{equation}
With its help, the counit $\epsilon$ of the adjunction $l\dashv r$ fits in the second commutative diagram of \eqref{eq:comp-counit}.

Consider now a strict monoidal 2-category which admits Eilenberg-Moore construction. 
The adjunction $1\dashv 1:I\to I$ of identity functors on the monoidal unit $I$ generates the monad $(I,1)$. Hence there is a comparison 1-cell $I\to I^1$.
For any monads $(A,t)$ and $(B,s)$, the adjunction $f^tf^s \dashv u^tu^s:A^tB^s \to AB$ generates the monad $(AB,ts)$. Hence there is a unique comparison 1-cell $A^tB^s \to (AB)^{ts}$.

\begin{definition} \label{def:mon-EM}
We say that a strict monoidal 2-category admits {\em monoidal  Eilenberg--Moore construction} if it admits Eilenberg--Moore construction and the comparison 1-cells $I\to I^1$  and $A^tB^s \to (AB)^{ts}$ are identities for the monoidal unit $I$ and all monads $(A,t)$ and $(B,s)$.
\end{definition}

For example, the strict monoidal 2-category $\mathsf{Cat}$ of Example \ref{ex:Cat} admits monoidal Eilenberg-Moore construction. More generally, monoidality of the Eilenberg-Moore construction in a Cartesian monoidal 2-category is a key assumption also in \cite[Theorem 5.1]{Zawadowski}.

\begin{proposition} \label{prop:mon-EM}
If a strict monoidal 2-category $\mathcal M$ admits monoidal Eilenberg--Moore construction, then the Eilenberg--Moore 2-functor $\mathsf{H}:\mathsf{Mnd}(\mathcal M) \to \mathcal M$ is strict monoidal.
If in addition $\mathcal M$ is symmetric then also $\mathsf H$ is symmetric.
\end{proposition}

\begin{proof}
The monoidal unit $I$ is preserved by the assumption that the comparison 1-cell $I\to I^1=\mathsf{H}(I,1)$ is the identity.
The equality of 2-functors $\mathsf{H}(-\otimes -)=\mathsf{H}(-)\otimes \mathsf{H}(-)$ holds on any 0-cells $(A,t)$ and $(B,s)$ by the assumption that the comparison 1-cell $A^tB^s \to (AB)^{ts}$ is the identity.
By the second diagram of \eqref{eq:comp-counit}, $\epsilon^{ts}=\epsilon^t\epsilon^s$.
Therefore for any 1-cells $(h,\chi):(A,t)\to (A',t')$ and $(g,\gamma):(B,s)\to (B',s')$, the 2-natural isomorphism $\mathcal M((AB)^{ts},(A'B')^{t's'}) \cong \mathsf{Mnd}(\mathcal M)(((AB)^{ts},1),(A'B',t's'))$ sends both objects $\mathsf{H}(hg,\chi\gamma)$ and $\mathsf{H}(h,\chi)\mathsf{H}(g,\gamma)$ to the same object 
$$
(hg.u^{ts},
\xymatrix{
t's'.hg.u^{ts} \ar[r]^-{\chi\gamma.1} &
hg.ts.u^{ts} = hg.u^{ts}.f^{ts}.u^{ts} \ar[r]^-{1.1.\epsilon^{ts}} &
hg.u^{ts}})
$$
which proves their equality. 
For any 2-cells $\omega$ and $\vartheta$, both 2-cells $\mathsf{H}(\omega \vartheta)$ and $\mathsf{H}(\omega )\mathsf{H}(\vartheta)$ are the liftings of the same monad transformation $\omega\vartheta$ hence they are equal.

Assume now the existence of a symmetry $\sigma$ of $\mathcal M$. For any monads $(A,t)$ and $(B,s)$, the 2-natural isomorphism $\mathcal M((AB)^{ts},(BA)^{st})\cong \mathsf{Mnd}(\mathcal M)(((AB)^{ts},1),(BA,st))$ sends both objects $\mathsf{H}(\sigma_{A,B},1)$ and $\sigma_{A^t,B^s}$ to the same object 
$$
(\sigma_{A,B}.u^{ts},
\xymatrix{
st.\sigma_{A,B}.u^{ts}=
\sigma_{A,B}.ts.u^{ts}=
\sigma_{A,B}.u^{ts}.f^{ts}.u^{ts} \ar[r]^-{1.1.\epsilon^{ts}} &
\sigma_{A,B}.u^{ts}})
$$
which proves their equality.
\end{proof}

From \eqref{eq:H0q} we obtain the following generalization of \cite[Theorem 8.2]{AguiarHaimLopezFranco}.

\begin{corollary}
In a symmetric strict monoidal 2-category which admits mono\-idal Eilenberg--Moore construction, the $q$-oidal structure of the base object $A$ of any $(0,q)$-oidal monad $(A,t)$ lifts to the Eilenberg--Moore object $A^t$ along the 1-cell $u^t:A^t\to A$.
\end{corollary}


\section{$(p,0)$-oidal monads}

Without entering the details, in this section we sketch the dual of the situation in Section \ref{sec:0q-monad}. Although the omitted proofs are analogous to those in Section \ref{sec:0q-monad}, they do not seem to follow by any kind of abstract duality.

The 2-functors $\mathsf{Mnd}_{\mathsf{op}}$ of Example \ref{ex:Mnd} and $(-)_{10}$ of part (2) of Example \ref{ex:Mpq} constitute a diagram 
$$
\xymatrix{
\mathsf{sm}\mbox{-}\mathsf{2Cat} \ar[r]^-{(-)_{10}} 
\ar[d]_-{\mathsf{Mnd}_{\mathsf{op}} } &
\mathsf{sm}\mbox{-}\mathsf{2Cat}  \ar[d]^-{\mathsf{Mnd}_{\mathsf{op}}} \\
\mathsf{sm}\mbox{-}\mathsf{2Cat} \ar[r]_-{(-)_{10}}  &
\mathsf{sm}\mbox{-}\mathsf{2Cat} }
$$
which is commutative up-to 2-natural isomorphism.
So whenever there is a symmetric strict monoidal 2-functor $\mathsf V$ from $\mathsf{Mnd}_{\mathsf{op}}(\mathcal M)=\mathsf{Mnd}(\mathcal M_{\mathsf{op}})_{\mathsf{op}}$ to some symmetric strict monoidal 2-category $\mathcal M$, it induces a symmetric strict monoidal 2-functor
\begin{equation}\label{eq:Vp0}
\xymatrix{
\mathsf{Mnd}_{\mathsf{op}}(\mathcal M_{p0}) \cong
\mathsf{Mnd}_{\mathsf{op}}(\mathcal M)_{p0} \ar[r]^-{\mathsf V_{p0}} &
\mathcal M_{p0}}
\end{equation}
for any non-negative integer $p$.
Its object map sends a $(p,0)$-oidal monad to a $p$-oidal object in $\mathcal M$.

\section{$p$-oidal Eilenberg--Moore objects}

In this section we consider a symmetric strict monoidal 2-category $\mathcal M$ and a non-negative integer $p$. Then we describe a setting in which the $p$-oidal structure of the base object $A$ of any $(p,0)$-oidal monad $(A,t)$ lifts to the Eilenberg--Moore object $A^t$ along the 1-cell $f^t:A\to A^t$. (A 1-cell $\overline h:A^t\to B^s$ is said to be a lifting of 1-cell $h:A\to B$ along the left adjoint 1-cells $f^t:A\to A^t$ and $f^s:B\to B^s$ if $f^s.h=\overline h.f^t$.)
The results of this section extend \cite[Theorem 7.1]{AguiarHaimLopezFranco} and give it a different explanation.

Let $\mathcal M$ be a 2-category which admits Eilenberg--Moore construction. For any 1-cell $(h,\chi):(A,t) \to (B,s)$ in $\mathsf{Mnd}_{\mathsf {op}}(\mathcal M)$, consider the reflexive pair of morphisms in $\mathcal M(A^t,B^s)$,
\begin{equation} \label{eq:Linton}
\xymatrix@C=25pt@R=1pt{
& f^s.h.u^t.f^t.u^t \ar@/^.5pc/[rd]^-{1.1.1.\epsilon^t} \\
f^s.h.t.u^t \ar@{=}@/^.5pc/[ru] \ar@/_.5pc/[rd]_-{1.\chi.1} &&
f^s.h.u^t \ar@{-->}[rr]^-{\pi(h,\chi)} &&
\mathsf V(h,\chi). \\
& f^s.s.h.u^t =
f^s.u^s.f^s.h.u^t \ar@/_.5pc/[ru]_-{\epsilon^s.1.1.1}}
\end{equation}
(A common section of the parallel morphisms is $1.1.\eta^t.1$ in terms of the unit $\eta^t$ of the monad $(A,t)$.) Assuming that their coequalizer exists for all 1-cells $(h,\chi)$ in $\mathsf{Mnd}_{\mathsf {op}}(\mathcal M)$, and it is preserved by the horizontal composition on either side with any 1-cell in $\mathcal M$, we construct a 2-functor $\mathsf V:\mathsf{Mnd}_{\mathsf{op}}(\mathcal M)\to \mathcal M$ whose action on any 1-cell $(h,\chi)$ is given by the coequalizer in \eqref{eq:Linton}. Coequalizers of this form were studied in \cite{Linton}.

\begin{lemma} \label{lem:V-hom} 
Consider a 2-category $\mathcal M$ which admits Eilenberg--Moore construction and take any monads $(A,t)$ and $(B,s)$ in $\mathcal M$. If the coequalizer \eqref{eq:Linton} exists for all objects of $\mathsf{Mnd}_{\mathsf {op}}(\mathcal M)((A,t),(B,s))$, then there is a functor $\mathsf{Mnd}_{\mathsf {op}}(\mathcal M)((A,t),(B,s))\to \mathcal M(A^t,B^s)$ whose object map is given by the coequalizer in \eqref{eq:Linton}. 
\end{lemma}

\begin{proof}
Take a morphism $\omega:(h,\chi) \to (k,\kappa)$ in $\mathsf{Mnd}_{\mathsf {op}}(\mathcal M)((A,t),(B,s))$. By the universality of the coequalizer in the top row of the serially commutative diagram
$$
\xymatrix@C=25pt@R=1pt{
& f^s.h.u^t.f^t.u^t \ar@/^.5pc/[rd]^-{1.1.1.\epsilon^t} \\
f^s.h.t.u^t \ar@{=}@/^.5pc/[ru] \ar@/_.5pc/[rd]_-{1.\chi.1} \ar[dddddd]_-{1.\omega.1.1} &&
f^s.h.u^t \ar[rr]^-{\pi(h,\chi)} \ar[dddddd]^-{1.\omega.1} &&
\mathsf V(h,\chi) \ar@{-->}[dddddd]^-{\mathsf V\omega} \\
& f^s.s.h.u^t =
f^s.u^s.f^s.h.u^t \ar@/_.5pc/[ru]_-{\epsilon^s.1.1.1}\\
\\
\\
\\
& f^s.k.u^t.f^t.u^t \ar@/^.5pc/[rd]^-{1.1.1.\epsilon^t} \\
f^s.k.t.u^t \ar@{=}@/^.5pc/[ru] \ar@/_.5pc/[rd]_-{1.\kappa.1} &&
f^s.k.u^t \ar[rr]^-{\pi(k,\kappa)} &&
\mathsf V(k,\kappa) \\
& f^s.s.k.u^t =
f^s.u^s.f^s.k.u^t \ar@/_.5pc/[ru]_-{\epsilon^s.1.1.1}}
$$
the image of $\omega$ under the desired functor $\mathsf V$ occurs in the right vertical. Functoriality follows by construction.
\end{proof}

In order to interpret the functors of Lemma \ref{lem:V-hom} as the hom functors of a 2-functor, we need to see their compatibility with the horizontal composition and the identity 1-cells. The proof of this rests on the following.

\begin{lemma} \label{lem:V}
Consider a 2-category $\mathcal M$ which admits Eilenberg--Moore construction. Assume that the coequalizer \eqref{eq:Linton} exists for some 1-cell $(h,\chi):(A,t)\to (B,s)$ in $\mathsf{Mnd}_{\mathsf {op}}(\mathcal M)$, and it is preserved by the horizontal composition on either side with any 1-cell in $\mathcal M$. Then 
$$
(\xymatrix@C=35pt{
f^s.h.u^t.f^t \ar[r]^-{\pi(h,\chi).1} &
\mathsf V(h,\chi).f^t}) =
(\xymatrix{
f^s.h.t\ar[r]^-{1.\chi} &
f^s.s.h=f^s.u^s.f^s.h \ar[r]^-{\epsilon^s.1.1} &
f^s.h}).
$$
\end{lemma}

\begin{proof}
By assumption $\mathsf V(h,\chi).f^t$ appears in the coequalizer
$$
\xymatrix@C=25pt@R=1pt{
& f^s.h.t \ar@{=}@/^.5pc/[rd] \\
f^s.h.t.t \ar@/^.5pc/[ru]^-{1.1.\mu^t} \ar@/_.5pc/[rd]_-{1.\chi.1} &&
f^s.h.u^t.f^t \ar[rr]^-{\pi(h,\chi).1}  &&
\mathsf V(h,\chi).f^t  \\
& f^s.s.h.t =
f^s.u^s.f^s.h.u^t.f^t \ar@/_.5pc/[ru]_-{\epsilon^s.1.1.1.1}}
$$
(where $\mu^t$ denotes the multiplication of the monad $(A,t)$). So the claim follows by the observation that
$$
\xymatrix@C=25pt@R=1pt{
& f^s.h.t \ar@{=}@/^.5pc/[rd] \\
f^s.h.t.t \ar@/^.5pc/[ru]^-{1.1.\mu^t} \ar@/_.5pc/[rd]^-{1.\chi.1} &&
f^s.h.t \ar[r]^-{1.\chi}   \ar@{-->}@/^4pc/[ll]^-{1.1.1.\eta^t}  &
f^s.s.h=f^s.u^s.f^s.h \ar[r]^-{\epsilon^s.1.1} &
f^s.h \ar@{-->}@/^4pc/[ll]^-{1.1.\eta^t} \\
& f^s.s.h.t =
f^s.u^s.f^s.h.t \ar@/_.5pc/[ru]^-{\epsilon^s.1.1.1\quad  }}
$$
is a split coequalizer (where $\eta^t$ is the unit of the monad $t$).
\end{proof}

\begin{proposition} \label{prop:V}
Consider a 2-category $\mathcal M$ which admits Eilenberg--Moore construction. If the coequalizer \eqref{eq:Linton} exists for all 1-cells in $\mathsf{Mnd}_{\mathsf {op}}(\mathcal M)$, and it is preserved by the horizontal composition on either side with any 1-cell in $\mathcal M$, then the following hold.

(1) There is a 2-functor $\mathsf V:\mathsf{Mnd}_{\mathsf{op}}(\mathcal M) \to \mathcal M$ which sends a monad $(A,t)$ to the Eilenberg--Moore object $A^t$ and whose hom functors are defined as in Lemma \ref{lem:V-hom}.

(2) If moreover $\mathcal M$ is a strict monoidal 2-category which admits monoidal Eilenberg--Moore construction, then the 2-functor $\mathsf V$ of part (1) is strict monoidal. 

(3) If in addition $\mathcal M$ has a symmetry then the 2-functor $\mathsf V$ of part (1) is symmetric as well.
\end{proposition}

\begin{proof}
(1) There is a monad $t.(-)$ on the category $\mathcal M(A^t,A)$ whose Eilenberg--Moore category is isomorphic to $\mathsf{Mnd}(\mathcal M)((A^t,1),(A,t))\cong \mathcal M(A^t,A^t)$; and whose forgetful functor differs by this isomorphism from 
\begin{equation} \label{eq:ut-}
u^t.(-):\mathcal M(A^t,A^t) \to \mathcal M(A^t,A).
\end{equation}
Thus  \eqref{eq:ut-} is monadic and therefore the coequalizers of those forks exist and are preserved by \eqref{eq:ut-} which are sent by \eqref{eq:ut-} to split coequalizers, see Proposition 3.5 on page 95 of \cite{BarrWells:TTT} or Theorem 4.4.4 in \cite{BorceuxII}. 

Now the functor \eqref{eq:ut-} sends the fork
\begin{equation} \label{eq:unit-coeq}
\xymatrix{
f^t.t.u^t=f^t.u^t.f^t.u^t
\ar@<2pt>[rr]^-{1.1.\epsilon^t} \ar@<-2pt>[rr]_-{\epsilon^t.1.1} &&
f^t.u^t \ar[r]^-{\epsilon^t} &
1}
\end{equation}
to the split coequalizer
$$
\xymatrix{
u^t.f^t.t.u^t=u^t.f^t.u^t.f^t.u^t
\ar@<2pt>[rr]^-{1.1.1.\epsilon^t} \ar@<-2pt>[rr]_-{1.\epsilon^t.1.1} &&
u^t.f^t.u^t \ar[r]^-{1.\epsilon^t} \ar@/^2pc/@{-->}[ll]^-{\eta^t.1.1.1}&
u^t  \ar@/^2pc/@{-->}[l]^-{\eta^t.1}}
$$ 
so that \eqref{eq:unit-coeq} is a coequalizer and thus $\mathsf V$ preserves identity 1-cells.

For the preservation of the horizontal composition recall that for arbitrary 1-cells $(h,\chi):(A,t) \to (B,s)$ and $(k,\kappa):(B,s) \to (C,z)$ in $\mathsf{Mnd}_{\mathsf{op}}(\mathcal M)$, $\mathsf V(k,\kappa).\mathsf V(h,\chi)$ and $\mathsf V((k,\kappa).(h,\chi))$ are defined as the coequalizers of the respective pairs of parallel morphisms in
$$
\xymatrix@C=15pt@R=1pt{
& \mathsf V(k,\kappa).f^s.h.u^t.f^t.u^t \ar@/^.5pc/[rd]^-{1.1.1.1.\epsilon^t} \\
\mathsf V(k,\kappa).f^s.h.t.u^t \ar@{=}@/^.5pc/[ru] \ar@/_.5pc/[rd]_(.3){1.1.\chi.1\ } 
\ar@{=}[dddddd] &&
\mathsf V(k,\kappa).f^s.h.u^t  \ar@{=}[dddddd] \\
& \mathsf V(k,\kappa).f^s.s.h.u^t =
\mathsf V(k,\kappa).f^s.u^s.f^s.h.u^t \ar@/_.5pc/[ru]_(.75){1.\epsilon^s.1.1.1} \\
\\
\\
\\
& f^z.k.h.u^t.f^t.u^t \ar@/^.5pc/[rd]^-{1.1.1.1.\epsilon^t} \\
f^z.k.h.t.u^t \ar@{=}@/^.5pc/[ru] \ar@/_.5pc/[rd]_-{1.1.\chi.1} &&
f^z.k.h.u^t  .\\
& f^z.k.s.h.u^t 
\raisebox{-6pt}{$\stackrel{\longrightarrow}{{}_{1.\kappa.1.1}}$} 
f^z.z.k.h.u^t = f^z.u^z.f^z.k.h.u^t
\ar@/_.5pc/[ru]_(.8){\epsilon^z.1.1.1.1}}
$$
The vertical equalities hold by Lemma \ref{lem:V}. The square with the lower ones of the parallel arrows commutes since the following diagram commutes and its dotted arrow denotes an epimorphism.
$$
\xymatrix@C=15pt@R=20pt{
\mathsf V (k,\kappa).f^s.u^s.f^s \ar[rrrr]^-{1.\epsilon^s.1} \ar@{=}[ddddd] &&&&
\mathsf V (k,\kappa).f^s \ar@{=}[ddddd] \\
& f^z.k.u^s.f^s.u^s.f^s \ar[rr]^-{1.1.1.\epsilon^s.1} \ar@{=}[d] \ar@{..>}[lu]_-{\pi(k,\kappa).1.1.1} &&
f^z.k.u^s.f^s \ar@{=}[d] \ar[ru]^-{\pi(k,\kappa).1} \\
\ar@{}[rd]|-{\textrm{Lemma }~\ref{lem:V}} & 
f^z.k.s.s \ar[rr]^-{1.1.\mu^s} \ar[d]^-{1.\kappa.1} &&
f^z.k.s \ar[d]_-{1.\kappa} \ar@{}[rd]|-{\textrm{Lemma }~\ref{lem:V}} \\
& f^z.z.k.s \ar[r]^-{1.1.\kappa} \ar@{=}[d] &
f^z.z.z.k \ar[r]^-{1.\mu^z.1} \ar@{=}[d] &
f^z.z.k \ar@{=}[d]  & \\
& f^z.u^z.f^z.k.s \ar[ld]^-{\epsilon^z.1.1.1} \ar[r]^-{\raisebox{8pt}{${}_{1.1.1.\kappa}$}} &
f^z.u^z.f^z.u^z.f^z.k \ar[d]^-{\epsilon^z.1.1.1.1} \ar[r]^-{\raisebox{8pt}{${}_{1.1.\epsilon^z.1.1}$}} &
f^z.u^z.f^z.k \ar[rd]_-{\epsilon^z.1.1} \\
f^z.k.s \ar[r]_-{1.\kappa} &
f^z.z.k \ar@{=}[r] &
f^z.u^z.f^z.k  \ar[rr]_-{\epsilon^z.1.1} &&
f^z.k}
$$
This also proves the equality 
\begin{equation} \label{eq:V-hori}
(\xymatrix@C=33pt{
\mathsf V(k,\kappa).f^s.h.u^t\ar[r]^-{1.\pi(h,\chi)} &
\mathsf V(k,\kappa).\mathsf V(h,\chi)}\!\!)=
(\xymatrix@C=55pt{
f^z.k.h.u^t \ar[r]^-{\pi((k,\kappa).(h,\chi))} &
\mathsf V((k,\kappa).(h,\chi))}\!\!).
\end{equation}

Part (2) is immediate. For part (3) recall that for the symmetry $\sigma$ of $\mathcal M$, the action of $\mathsf V$ on the symmetry $(\sigma,1):(AB,ts) \to (BA,st)$ of $\mathsf{Mnd}_{\mathsf{op}}(\mathcal M)$ is defined as the coequalizer of the morphisms in the top row of
$$
\xymatrix@C=25pt@R=1pt{
& f^sf^t.\sigma.u^tu^s.f^tf^s.u^tu^s \ar@/^.5pc/[rd]^-{1.1.1.\epsilon^t \epsilon^s} \\
f^sf^t.\sigma.ts.u^tu^s \ar@{=}@/^.5pc/[ru] \ar@{=}@/_.5pc/[rd] \ar@{=}[dddddd] &&
f^sf^t.\sigma.u^tu^s \ar@{=}[dddddd] \\
& f^sf^t.st.\sigma.u^tu^s =
f^sf^t.u^su^t.f^sf^t.\sigma.u^tu^s \ar@/_.5pc/[ru]_-{\epsilon^s\epsilon^t.1.1.1}\\
\\
\\
\\
& \sigma.f^tf^s.u^tu^s.f^tf^s.u^tu^s \ar@/^.5pc/[rd]^-{1.1.1.\epsilon^t \epsilon^s} \\
\sigma.f^tf^s.ts.u^tu^s \ar@{=}@/^.5pc/[ru] \ar@{=}@/_.5pc/[rd] &&
\sigma.f^tf^s.u^tu^s .\\
& \sigma.f^tf^s.ts.u^tu^s =
\sigma.f^tf^s.u^tu^s.f^tf^s.u^tu^s \ar@/_.5pc/[ru]_-{1.\epsilon^t\epsilon^s.1.1}}
$$
Both the vertical equalities and serial commutativity of the diagram follow by the 2-naturality of $\sigma$.
The bottom row is the image under $\sigma.(-)$ of parallel morphisms as in $\eqref{eq:unit-coeq}$ hence their coequalizer is $\sigma.1=\sigma$ as stated.
\end{proof}

In view of Lemma \ref{lem:V}, the action of $\mathsf V$ can be interpreted as a lifting along the left adjoint 1-cells of the kind $f^t:A\to A^t$. So from \eqref{eq:Vp0} we obtain the following.

\begin{corollary} \label{cor:Vlift}
In the setting of Proposition \ref{prop:V}~(3), the $p$-oidal structure of the base object $A$ of any $(p,0)$-oidal monad $(A,t)$ in $\mathcal M$ lifts to the Eilenberg-Moore object $A^t$ along the 1-cell $f^t:A \to A^t$.
\end{corollary}

\begin{remark} \label{rem:reflCat}
The premises of Proposition \ref{prop:V}~(3) hold in a suitable 2-subcategory of the 2-category $\mathsf{Cat}$ of Example \ref{ex:Cat} which we describe next.

Let $\mathcal M$ be the symmetric strict monoidal 2-subcategory of $\mathsf{Cat}$ whose 0-cells are those categories in which the coequalizers of reflexive pairs exist; whose 1-cells are those functors which preserve reflexive coequalizers; and whose 2-cells are all natural transformations between them.

For any monad $(A,t)$ in $\mathcal M$, the coequalizer of any reflexive pair in $A^t$ exists and it is preserved by the forgetful functor $A^t\to A$, see \cite[Corollary 3]{Linton} or \cite[Proposition 4.3.2]{BorceuxII}. Since the forgetful functor is conservative, it also reflects the coequalizers of reflexive pairs.
Then  if some functor $h:A\to B$ preserves reflexive coequalizers, then so does $h.u^t=u^s.\mathsf H(h,\chi)$ for any monad functor $(h,\chi):(A,t) \to (B,s)$ and the Eilenberg--Moore 2-functor $\mathsf H$. So since $u^s$ reflects reflexive coequalizers, $\mathsf H(h,\chi):A^t\to A^s$ preserves them. Because in this way $\mathcal M$ is closed under the monoidal Eilenberg--Moore construction of $\mathsf{Cat}$, it admits monoidal Eilenberg--Moore construction itself.

Evaluating the parallel natural transformations of \eqref{eq:Linton} at an arbitrary $t$-algebra we get a reflexive pair in $B^s$. Their coequalizer exists by the considerations in the previous paragraph and it is evidently preserved by any 1-cell in $\mathcal M$. Then also the coequalizer \eqref{eq:Linton} exists and it is preserved by the horizontal composition. 
\end{remark}

Applying Corollary \ref{cor:Vlift} to the 2-category of Remark \ref{rem:reflCat},  we obtain an alternative proof of related  \cite[Theorem 2.6.4]{Seal} and \cite[Theorem 4.14]{AguiarHaimLopezFranco}; as well as of \cite[Theorem 7.1]{AguiarHaimLopezFranco}.


\section{Multimonoidal structures in double categories}
\label{sec:Dpq}

For a study of the lifting of the $(p+q)$-oidal structure of the base object of a $(p,q)$-oidal monad in a symmetric strict monoidal 2-category, in the case when both $p$ and $q$ are strictly positive integers, we leave the realm of 2-categories and operate, more generally, with double categories. 
Therefore this section begins with the introduction of  a 2-category $\mathsf{sm}\mbox{-}\mathsf{DblCat}$ of symmetric strict monoidal double categories. Its choice is motivated by two requirements. On one hand, it should allow for the interpretation of the symmetry of a symmetric strict monoidal double category in \cite[Definition 5]{BruniMeseguerMontanari} as a suitable 2-cell in the 2-category of double categories. On the other hand, even more importantly, a definition of $\mathsf{sm}\mbox{-}\mathsf{DblCat}$ would serve our purposes if it admits several 2-functors constructed in this section and the next one.

In Example \ref{ex:Dpq} we introduce 2-endofunctors $(-)_{01}$ and $(-)_{10}$ on $\mathsf{sm}\mbox{-}\mathsf{DblCat}$; and show that they commute up-to 2-natural isomorphism.
Then iterating them in an arbitrary order $q$ and $p$ times, respectively, we obtain a 2-endofunctor $(-)_{pq}$ on $\mathsf{sm}\mbox{-}\mathsf{DblCat}$. The 0-cells of $\mathbb D_{pq}$, for any symmetric strict monoidal double category $\mathbb D$, can be interpreted as $(p,q)$-oidal objects in $\mathbb D$.

Ehresmann's {\em square} or {\em quintet} construction \cite{Ehresmann} yields a 2-functor $\Sqr :\mathsf{sm}\mbox{-}\mathsf{2Cat} \to \mathsf{sm}\mbox{-}\mathsf{DblCat}$. Together with the 2-functor $(-)_{pq}:\mathsf{sm}\mbox{-}\mathsf{2Cat} \to \mathsf{sm}\mbox{-}\mathsf{2Cat}$ of Example \ref{ex:Mpq}~(4) for any non-negative integers $p$ and $q$, it fits in the commutative diagram
$$
\xymatrix{
\mathsf{sm}\mbox{-}\mathsf{2Cat} \ar[r]^-{\Sqr } \ar[d]_-{(-)_{pq}} &
\mathsf{sm}\mbox{-}\mathsf{DblCat} \ar[d]^-{(-)_{pq}} \\
\mathsf{sm}\mbox{-}\mathsf{2Cat} \ar[r]_-{\Sqr }  &
\mathsf{sm}\mbox{-}\mathsf{DblCat}. }
$$
The 2-functors around it take a symmetric strict monoidal 2-category $\mathcal M$ to the symmetric strict monoidal double category $\Sqr (\mathcal M)_{pq} \cong \Sqr (\mathcal M_{pq})$. Thus the  $(p,q)$-oidal objects in $\Sqr (\mathcal M)$ are the same as the $(p+q)$-oidal objects of $\mathcal M$.

We begin with quickly recalling some basic definitions. More details can be found e.g. in \cite{BruniMeseguerMontanari}.

\begin{definition} \label{def:doublecat}
A {\em double category} is an internal category in the category of small categories and functors. Thus a double category $\mathbb D$ consists of a category of objects $V$, a category of morphisms $\mathcal V$ and structure functors 
$$
\xymatrix@C=40pt{
V \ar[r]|-{\,i^\hori\,} &
\mathcal V \ar@<-5pt>[l]_-{s^\hori} \ar@<5pt>[l]^-{t^\hori} &
\mathcal V \times_V \mathcal V \ar[l]_{c^\hori} }
$$
which are subject to the usual axioms of internal category.
\end{definition}

The objects of $V$ are called the {\em 0-cells} or {\em objects} of $\mathbb D$ and the morphisms of $V$ are called the {\em vertical 1-cells} of $\mathbb D$. The objects of $\mathcal V$ are called the {\em horizontal 1-cells} of $\mathbb D$ and the morphisms of $\mathcal V$ are called the {\em 2-cells} of $\mathbb D$. The compositions of both categories $V$ and $\mathcal V$ are called the {\em vertical composition} (of 1-cells and 2-cells, respectively).The functor $c^\hori$ is called the {\em horizontal composition}. A general 2-cell is depicted as
$$
\xymatrix{
s^\verti (s^\hori \Theta) =s^\hori (s^\verti \Theta) \ar[r]^-{s^\verti \Theta} \ar@<-20pt>[d]_-{s^\hori \Theta} \ar@{}[rd]|-{\Longdownarrow\Theta} &
t^\hori (s^\verti \Theta) =s^\verti (t^\hori \Theta)\ar@<20pt>[d]^-{t^\hori \Theta} \\
t^\verti (s^\hori \Theta)=s^\hori (t^\verti \Theta) \ar[r]_-{t^\verti \Theta} &
t^\hori (t^\verti \Theta)=t^\verti (t^\hori \Theta)}
$$
where $s^\verti$ denotes the source maps, and $t^\verti$ denotes the target maps of both categories $V$ and $\mathcal V$. In this graphical notation, vertical composition is denoted by vertical juxtaposition and horizontal composition is denoted by horizontal juxtaposition.

In Definition \ref{def:doublecat} the roles of the vertical and horizontal structures turn out to be symmetric. This leads to a symmetric description of the same double category $\mathbb D$. It can be seen equivalently as an internal category in the category of small categories and functors with another category of objects $H$ and another category of morphisms $\mathcal H$. In $H$ the objects are again the 0-cells of $\mathbb D$ but now the morphisms are the horizontal 1-cells. The objects of $\mathcal H$ are the vertical 1-cells and the morphisms are the 2-cells again. The compositions in $H$ and $\mathcal H$ are provided by the horizontal composition of $\mathbb D$. In this description the structure functors of the internal category are
$$
\xymatrix@C=40pt{ 
H \ar[r]|-{\,i^\verti\,} &
\mathcal H \ar@<-5pt>[l]_-{s^\verti} \ar@<5pt>[l]^-{t^\verti} &
\mathcal H \times_H \mathcal H .\ar[l]_{c^\verti} }
$$
Here the functor $i^\verti$ sends any object $A$ of $H$ (equivalently, of $V$) to the corresponding identity morphism  $1:A\to A$ in $V$ (regarded now as an object of $\mathcal H$) and it sends a horizontal morphism $h$ to the identity morphism $h\to h$ in $\mathcal V$ (regarded now as a non-identity morphism in $\mathcal H$).

Any double category $\mathbb D$ contains two 2-categories. In the {\em horizontal 2-category} $\mathsf{Hor}(\mathbb D)$ the 0-cells are the 0-cells, and the 1-cells are the horizontal 1-cells of $\mathbb D$. The 2-cells are those 2-cells in $\mathbb D$ which are taken by $s^\hori$ and $t^\hori$ to identity vertical 1-cells. Both horizontal and vertical compositions are the restrictions of the respective composition in $\mathbb D$.
Symmetrically, in the {\em vertical 2-category} $\mathsf{Ver}(\mathbb D)$ the 0-cells are the 0-cells, and the 1-cells are the vertical 1-cells of $\mathbb D$. The 2-cells are those 2-cells in $\mathbb D$ which are taken by $s^\verti$ and $t^\verti$ to identity horizontal 1-cells. In this 2-category the horizontal composition is the restriction of the vertical composition of $\mathbb D$ and vertical composition is the restriction of the horizontal composition of $\mathbb D$.

\begin{definition} \label{def:doublefunctor}
A {\em double functor} is an internal functor in the category of small categories and functors. Thus a double functor $\mathbf F:\mathbb D \to \mathbb D'$ consists of compatible functors $(V\to V',\mathcal V \to \mathcal V')$. Equivalently, it consists of compatible functors $(H\to H',\mathcal H \to \mathcal H')$.
\end{definition}

The object maps of the functors $V\to V'$ and $H\to H'$ in Definition \ref{def:doublefunctor} are equal. The morphism map of the functor $V\to V'$ is equal to the object map of the functor $\mathcal H \to \mathcal H'$ and the morphism map of the functor $H\to H'$ is equal to the object map of the functor $\mathcal V \to \mathcal V'$. The morphism maps of the functors $\mathcal V \to \mathcal V'$ and $\mathcal H \to \mathcal H'$ are equal again.
Any double functor $\mathbf F$ induces 2-functors $\mathsf{Hor}(\mathbf F)$ and $\mathsf{Ver}(\mathbf F)$ between the horizontal and vertical 2-categories, respectively, which we call the horizontal and the vertical 2-functors, respectively.

Between double functors $\mathbf F,\mathbf G:\mathbb D \to \mathbb D'$, there are two symmetric variants of internal natural transformations.
A {\em horizontal transformation} $x:\mathbf F \to \mathbf G$ is an internal natural transformation in the category of small categories and functors from $(\mathbf F:V\to V',\mathbf F:\mathcal V\to \mathcal V')$ to $(\mathbf G:V\to V',\mathbf G:\mathcal V\to \mathcal V')$. Thus it is a functor $V \to \mathcal V'$ sending any vertical 1-cell $f:A\to B$ to a 2-cell 
$$
\xymatrix{
\mathbf F A \ar[r]^-{x_A} \ar[d]_-{\mathbf F f}  \ar@{}[rd]|-{\Longdownarrow{x_f}}&
\mathbf G A \ar[d]^-{\mathbf G f} \\ 
\mathbf F B \ar[r]_-{x_B} & 
\mathbf G B}
$$
subject to the naturality condition 
\begin{equation} \label{eq:vert}
\xymatrix{
\mathbf F A \ar[r]^-{\mathbf F l} \ar[d]_-{\mathbf F f} \ar@{}[rd]|-{\Longdownarrow{\mathbf F \omega}} &
\mathbf F C \ar[d]|-{\mathbf F g} \ar[r]^-{x_C} \ar@{}[rd]|-{\Longdownarrow{x_g}} &
\mathbf G C \ar[d]^-{\mathbf G g} \\
\mathbf F B \ar[r]_-{\mathbf F k} &
\mathbf F D \ar[r]_-{x_D} & 
\mathbf G D} 
\ \raisebox{-17pt}{$=$} \ 
\xymatrix{
\mathbf F A \ar[r]^-{x_A} \ar[d]_-{\mathbf F f}  \ar@{}[rd]|-{\Longdownarrow{x_f}} &
\mathbf G A \ar[d]|-{\mathbf G f} \ar[r]^-{\mathbf G l} \ar@{}[rd]|-{\Longdownarrow{\mathbf G \omega}}  &
\mathbf G C \ar[d]^-{\mathbf G g} \\ 
\mathbf F B \ar[r]_-{x_B} & 
\mathbf G B \ar[r]_-{\mathbf G k} &
\mathbf G D}
\ \raisebox{-17pt}{for all} \ 
\xymatrix{
A \ar[r]^-l \ar[d]_-f \ar@{}[rd]|-{\Longdownarrow\omega} &
C \ar[d]^-g  \\
B \ar[r]_-k &
D }
\quad \raisebox{-17pt}{in $\mathbb D$.}
\end{equation}
For any horizontal transformation $x:\mathbf F \to \mathbf G$, the components $x_A:\mathbf F A\to \mathbf G A$ constitute a 2-natural transformation $\mathsf{Hor}(x): \mathsf{Hor}(\mathbf F) \to \mathsf{Hor}(\mathbf G)$.

There is then a 2-category in which  the 0-cells are the double categories, the 1-cells are the double functors and the 2-cells are the horizontal transformations. The horizontal composition is induced by the consecutive application of double functors and the vertical composition is induced by the horizontal composition in the target double category. We may see $\mathsf{Hor}$ as a 2-functor from this 2-category to $\mathsf{2Cat}$. 

Symmetrically, a {\em vertical transformation} $y:\mathbf F \to \mathbf G$ is an internal natural transformation in the category of small categories and functors from $(\mathbf F:H\to H',\mathbf F:\mathcal H\to \mathcal H')$ to $(\mathbf G:H\to H',\mathbf G:\mathcal H\to \mathcal H')$. Thus it is a functor $H \to \mathcal H'$ sending any horizontal 1-cell $l:A\to C$ to a 2-cell 
$$
\xymatrix{
\mathbf F A \ar[r]^-{\mathbf F l} \ar[d]_-{y_A} \ar@{}[rd]|-{\Longdownarrow{y_l}}&
\mathbf F C \ar[d]^-{y_C} \\
\mathbf G A \ar[r]_-{\mathbf G l} &
\mathbf G C}
$$
subject to the naturality condition 
\begin{equation} \label{eq:hor}
\xymatrix{
\mathbf F A \ar[r]^-{\mathbf F l} \ar[d]_-{\mathbf F f} \ar@{}[rd]|-{\Longdownarrow{\mathbf F \omega}} &
\mathbf F C \ar[d]^-{\mathbf F g} \\
\mathbf F B \ar[r]^-{\mathbf F k} \ar[d]_-{y_B} \ar@{}[rd]|-{\Longdownarrow{y_k}} &
\mathbf F D  \ar[d]^-{y_D} \\
\mathbf G B \ar[r]_-{\mathbf G k} &
\mathbf G D}
\quad \raisebox{-40pt}{$=$} \quad
\xymatrix{
\mathbf F A \ar[r]^-{\mathbf F l} \ar[d]_-{y_A} \ar@{}[rd]|-{\Longdownarrow{y_l}} &
\mathbf F C \ar[d]^-{y_C} \\
\mathbf G A \ar[d]_-{\mathbf G f} \ar[r]^-{\mathbf G l} \ar@{}[rd]|-{\Longdownarrow{\mathbf G \omega}} &
\mathbf G C \ar[d]^-{\mathbf G g} \\
\mathbf G B \ar[r]_-{\mathbf G k} &
\mathbf G D}
\quad \raisebox{-40pt}{for all} \quad
\raisebox{-20pt}{$\xymatrix{
A \ar[r]^-l \ar[d]_-f \ar@{}[rd]|-{\Longdownarrow\omega} &
C \ar[d]^-g  \\
B \ar[r]_-k &
D }$}
\quad \raisebox{-40pt}{in $\mathbb D$.}
\end{equation}
For any vertical transformation $y:\mathbf F \to \mathbf G$, the components $y_A:\mathbf F A\to \mathbf G A$ constitute a 2-natural transformation $\mathsf{Ver}(y): \mathsf{Ver}(\mathbf F) \to \mathsf{Ver}(\mathbf G)$.

There is then a 2-category of  double categories, double functors and vertical transformations, and a 2-functor $\mathsf{Ver}$ from it to $\mathsf{2Cat}$.

The notion of double natural transformation, which is  relevant for our purposes, symmetrically combines vertical and horizontal transformations. It is a particular case of {\em generalized natural transformations} in \cite{BruniMeseguerMontanari}.

\begin{definition}\label{def:dbl-nattr}
For double functors $\mathbf F,\mathbf G:\mathbb D \to \mathbb D'$, a {\em double natural transformation} $\mathbf F \to \mathbf G$ consists of
\begin{itemize}
\item a horizontal transformation $x:\mathbf F \to \mathbf G$
\item a vertical transformation $y:\mathbf F \to \mathbf G$
\vspace{-43pt}

\item a map $\theta$ associating to each 0-cell $A$ of $\mathbb D$ a 2-cell 
\raisebox{40pt}{$
\xymatrix@R=26pt{
\mathbf F A \ar[r]^-{x_A}  \ar[d]_-{y_A} \ar@{}[rd]|-{\Longdownarrow{\theta_A}}&
\mathbf G A \ar@{=}[d] \\
\mathbf G A \ar@{=}[r] &
\mathbf G A}$} 
in $\mathbb D'$
\end{itemize} 
such that $\theta$ gives rise to natural transformation from the functor $x:V \to \mathcal V'$ to the functor sending a vertical 1-cell $f:A \to B$ to the identity morphism $1:\mathbf Gf\to \mathbf Gf$ in $\mathcal H'$ (regarded as a non-identity morphism of $\mathcal V'$); as well as a natural transformation from the functor $H \to \mathcal H'$ sending a horizontal 1-cell $l:A\to C$ to the identity morphism $1:\mathbf Gl\to \mathbf Gl$ in $\mathcal V'$ (regarded as a non-identity morphism of $\mathcal H'$) to $y$. Equivalently, $\theta$ is a {\em modification} in the sense of \cite[Section 1.6]{GrandisPare}. That is, in addition to the conditions \eqref{eq:vert} and \eqref{eq:hor}, also the further naturality conditions
\begin{equation}\label{eq:doub}
\xymatrix@C=18pt@R=24pt{
\mathbf F A \ar[r]^-{x_A} \ar[d]_-{\mathbf F f} \ar@{}[rd]|-{\Longdownarrow{x_f}} &
\mathbf G A  \ar[d]^-{\mathbf G f} \\
\mathbf F B \ar[r]^-{x_B} \ar[d]_-{y_B} \ar@{}[rd]|-{\Longdownarrow{\theta_B}} &
\mathbf G B \ar@{=}[d] \\
\mathbf G B \ar@{=}[r] &
\mathbf G B }
\raisebox{-40pt}{$=$} 
\xymatrix@C=18pt@R=24pt{
\mathbf F A \ar[r]^-{x_A} \ar[d]_-{y_A} \ar@{}[rd]|-{\Longdownarrow{\theta_A}} &
\mathbf G A \ar@{=}[d] \\
\mathbf G A \ar@{=}[r] \ar[d]_-{\mathbf G f} \ar@{}[rd]|-{\Longdownarrow 1} &
\mathbf G A \ar[d]^-{\mathbf G f} \\
\mathbf G B \ar@{=}[r] &
\mathbf G B}
\xymatrix@C=20pt{
\mathbf F A \ar[r]^-{\mathbf F l} \ar[d]_-{y_A} \ar@{}[rd]|-{\Longdownarrow{y_l}} &
\mathbf F C \ar[d]|-{y_C} \ar[r]^-{x_C} \ar@{}[rd]|-{\Longdownarrow{\theta_C}} &
\mathbf G C \ar@{=}[d] \\
\mathbf G A \ar[r]_-{\mathbf G l} &
\mathbf G C \ar@{=}[r] & 
\mathbf G C}
\raisebox{-17pt}{$=$} 
\xymatrix@C=20pt{
\mathbf F A \ar[d]_-{y_A} \ar[r]^-{x_A} \ar@{}[rd]|-{\Longdownarrow{\theta_A}} &
\mathbf G A \ar@{=}[d] \ar[r]^-{\mathbf G l} \ar@{}[rd]|-{\Longdownarrow 1} &
\mathbf G C \ar@{=}[d]  \\
\mathbf G A \ar@{=}[r] & 
\mathbf G A \ar[r]_-{\mathbf G l} &
\mathbf G C}
\end{equation}
hold, for all vertical 1-cells $f:A\to B$ and horizontal 1-cells $l:A\to C$.
\end{definition}

Double categories are the 0-cells, double functors are the 1-cells and double natural transformations are the 2-cells of the 2-category $\mathsf{DblCat}$. 
For double functors $\mathbf F$, $\mathbf G$ and $\mathbf H:\mathbb D \to \mathbb D'$,
the vertical composite of double natural transformations $(x,\theta,y):\mathbf F \to \mathbf G$ and $(x',\theta',y'):\mathbf G \to \mathbf H$ has the component
\begin{equation} \label{eq:diag_comp}
\xymatrix{
\mathbf F A \ar[r]^-{x_A} \ar[d]_-{y_A} \ar@{}[rd]|-{\Longdownarrow{\theta_A}} &
\mathbf G A \ar@{=}[d] \ar[r]^-{x'_A} \ar@{}[rd]|-{\Longdownarrow 1} &
\mathbf H A \ar@{=}[d] \\
\mathbf G A \ar@{=}[r] \ar[d]_-{y'_A} \ar@{}[rd]|-{\Longdownarrow 1} &
\mathbf G A \ar[d]|-{\, y'_A} \ar[r]^-{x'_A}  \ar@{}[rd]|-{\Longdownarrow{\theta'_A}}  &
\mathbf H A \ar@{=}[d] \\
\mathbf H A \ar@{=}[r] &
\mathbf H A \ar@{=}[r] &
\mathbf H A }
\end{equation}
for any 0-cell $A$ of $\mathbb D$. This is the {\em diagonal composition} of \cite{BruniMeseguerMontanari}.
Double functors $\mathbf F:\mathbb D \to \mathbb D'$ and $\mathbf F':\mathbb D' \to \mathbb D''$ are composed as any internal functors. For another pair of double functors $\mathbf G:\mathbb D \to \mathbb D'$ and $\mathbf G':\mathbb D' \to \mathbb D''$, and for double natural transformations $(x,\theta,y):\mathbf F \to \mathbf G$ and $(x',\theta',y'):\mathbf F' \to \mathbf G'$, the horizontal composite has the component
$$
\xymatrix{
\mathbf F'(\mathbf FA) \ar[r]^-{\mathbf F' x_A} \ar[d]_-{\mathbf F' y_A} \ar@{}[rd]|-{\Longdownarrow{\mathbf F' \theta_A}} &
\mathbf F'(\mathbf GA)\ar[r]^-{x'_{\mathbf G A}} \ar@{=}[d] \ar@{}[rd]|-{\Longdownarrow 1} &
\mathbf G'(\mathbf G A) \ar@{=}[d] \\
\mathbf F'(\mathbf GA) \ar@{=}[r] \ar[d]_-{y'_{\mathbf G A}} \ar@{}[rd]|-{\Longdownarrow 1} &
\mathbf F'(\mathbf GA) \ar[r]^-{x'_{\mathbf GA}} \ar[d]|-{y'_{\mathbf GA}} \ar@{}[rd]|-{\Longdownarrow{\theta'_{\mathbf GA}}}  &
\mathbf G'(\mathbf G A) \ar@{=}[d] \\
\mathbf G'(\mathbf G A) \ar@{=}[r] &
\mathbf G'(\mathbf G A) \ar@{=}[r] &
\mathbf G'(\mathbf G A)}
$$
for any 0-cell $A$ of $\mathbb D$.
We leave it to the reader to check that $\mathsf{DblCat}$ is a 2-category indeed; in particular, to derive the middle four interchange law from \eqref{eq:vert}, \eqref{eq:hor} and \eqref{eq:doub}. There are evident forgetful 2-functors from $\mathsf{DblCat}$ to the 2-category of double categories, double functors and  horizontal or vertical transformations.

Although double natural transformations as in Definition \ref{def:dbl-nattr} were introduced in \cite{BruniMeseguerMontanari} (without using this name for them), we did not manage to find an explicit description of the above 2-category $\mathsf{DblCat}$ anywhere in the literature.

The 2-category $\mathsf{DblCat}$ can be made monoidal with the Cartesian product of double categories as the monoidal product. Then it has a symmetry $\mathbf X$ as well, provided by the flip double functors $\mathbb D \times \mathbb C\to \mathbb C \times \mathbb D$, sending a pair of 2-cells $(\omega,\vartheta)$ to $(\vartheta,\omega)$.

\begin{remark}
The author is grateful to Ross Street for kindly pointing out the following derivation of  the 2-category $\mathsf{DblCat}$ via the `change of enriching category construction' in \cite{EilenbergKelly}.

The category $\mathsf{dblcat}$ of double categories and double functors is Cartesian closed. (In the internal hom double category $[\mathbb A,\mathbb B]$ the 0-cells are the double functors $\mathbb A\to \mathbb B$, the horizontal/vertical 1-cells are the horizontal/vertical transformations and the 2-cells are the modifications in \cite[Section 1.6]{GrandisPare}.) Hence it can be seen as a $\mathsf{dblcat}$-enriched category. 

Now let us consider the following monoidal functor $F$ from $\mathsf{dblcat}$ to the category $\mathsf{cat}$ of small categories and functors. It sends a double category $\mathbb D$ to the category whose objects are the 0-cells of $\mathbb D$; whose morphisms $P \to Q$ are the triples consisting of a horizontal 1-cell $x$, a vertical 1-cell $y$ and a 2-cell 
$$
\xymatrix{
P \ar[r]^-x \ar[d]_-y \ar@{}[rd]|-{\Longdownarrow{\theta}} &
Q \ar@{=}[d] \\ 
Q \ar@{=}[r] &
Q}
$$
in $\mathbb D$; and whose composition is the diagonal composition as in \eqref{eq:diag_comp}. A double functor $\mathbf H$ is taken by $F$ to the functor which acts on the objects as $\mathbf H$ does; and acts on the morphisms componentwise. 

Following \cite{EilenbergKelly}, the monoidal functor $F$ induces a 2-functor from the 2-category of $\mathsf{dblcat}$-enriched categories to the 2-category $\mathsf{2Cat}$ of $\mathsf{cat}$-enriched categories. This induced 2-functor sends the $\mathsf{dblcat}$-enriched category $\mathsf{dblcat}$ precisely to the 2-category $\mathsf{DblCat}$ above.
\end{remark}

\begin{example} \label{ex:D*}
There are several notions of duality for double categories (see e.g. \cite[Section 1.2]{GrandisPare}). In addition to the horizontal and vertical opposites as for 2-categories, there is a further possibility to take the {\em diagonal dual} or {\em transpose} $\mathbb D^*$ of a double category $\mathbb D$. Its 0-cells are the 0-cells of $\mathbb D$, its horizontal 1-cells are the vertical 1-cells of $\mathbb D$ and vice versa, its vertical 1-cells are the horizontal 1-cells of $\mathbb D$. A 2-cell in $\mathbb D^*$ as on the left,
$$
\xymatrix@C=10pt@R=1pt{
A \ar@{..>}[rr]^-h \ar@{..>}[dddd]_-f &
\raisebox{-5pt}{$\ar@2@{..>}[ddd] $} &
C \ar@{..>}[dddd]^-g \\
\\
\\
& \\
B \ar@{..>}[rr]_-k &&
D}\qquad \qquad
\xymatrix@C=27pt@R=27pt{
A \ar[r]^-f\ar[d]_-h \ar@{}[rd]|-{\Longdownarrow{}}  &
B \ar[d]^-k \\
C \ar[r]_-g &
D}
$$
is a 2-cell in $\mathbb D$ on the right. 
The horizontal composition in $\mathbb D^*$ is the vertical composition in $\mathbb D$ and the vertical composition in $\mathbb D^*$ is the horizontal composition in $\mathbb D$.

Any double functor $\mathbf F:\mathbb D\to \mathbb C$ induces a double functor $\mathbf F^*:\mathbb D^*\to \mathbb C^*$. On the 0-cells and on the 2-cells it acts as $\mathbf F$. Its action on the horizontal and vertical 1-cells is the action of $\mathbf F$ on the vertical and horizontal 1-cells, respectively.

The components of a double natural transformation $(x,\theta,y):\mathbf F \to \mathbf G$ on the left
$$
\xymatrix@C=8pt@R=1pt{
\mathbf F A \ar[rr]^-x \ar[dddd]_-y & 
\raisebox{-5pt}{$\ar@2@{>}[ddd]^-{\displaystyle \theta}$} &
\mathbf G A \ar@{=}[dddd] \\
\\
\\
& \\
\mathbf G A \ar@{=}[rr] &&
\mathbf G A} \qquad \qquad
\xymatrix@C=8pt@R=1pt{
\mathbf F A \ar@{..>}[rr]^-y \ar@{..>}[dddd]_-x &
\raisebox{-5pt}{$\ar@2@{..>}[ddd]^-{\displaystyle \theta}$} &
\mathbf G A \ar@2@{..}[dddd] \\
\\
\\
& \\
\mathbf G A \ar@2@{..}[rr] &&
\mathbf G A}
$$
can be seen as the components on the right of a double natural transformation $(x,\theta,y)^*:\mathbf F^* \to \mathbf G^*$.

This defines 2-functors $(-)^*$ from the 2-category of double categories, double functors and horizontal/vertical transformations to the 2-category of double categories, double functors and vertical/horizontal transformations; as well as a 2-endofunctor $(-)^*$ on the 2-category $\mathsf{DblCat}$ of double categories, double functors and double natural transformations.
\end{example}

\begin{definition} \cite[Definition 4]{BruniMeseguerMontanari}
A {\em strict monoidal double category} is a monoid in the category of double categories and double functors considered with the Cartesian product of double categories as the monoidal product. That is, a strict monoidal double category consists of a double category $\mathbb D$ together with double functors $\otimes$ --- the {\em monoidal product} --- from the Cartesian product double category $\mathbb D \times \mathbb D$ to $\mathbb D$ and $I$ --- the {\em monoidal unit} --- from the singleton double category (with a single 0-cell and only identity higher cells) to $\mathbb D$ such that $\otimes$ is strictly associative with the strict unit $I$.
\end{definition}

The action of the monoidal product double functor $\otimes$ on any cells will be denoted by juxtaposition. Note that in a strict monoidal double category $\mathbb D$ all constituent categories $V$, $H$, $\mathcal V$ and $\mathcal H$ are strict monoidal, and so are the horizontal and vertical 2-categories $\mathsf{Hor}(\mathbb D)$ and $\mathsf{Ver}(\mathbb D)$.

\begin{definition} \cite[Definition 4]{BruniMeseguerMontanari}
A {\em strict monoidal double functor} is a monoid morphism in the category of double categories and double functors considered with the Cartesian product of double categories as the monoidal product. That is, a double functor $\mathbf F:\mathbb D \to \mathbb D'$ such that $\mathbf F(-\otimes -)=\mathbf F(-)\otimes' \mathbf F(-)$ and $\mathbf FI=I'$.
\end{definition}

\begin{definition} 
A {\em monoidal horizontal transformation} between strict monoidal double functors $(\mathbb D,\otimes,I) \to (\mathbb D',\otimes',I')$ is a horizontal transformation $z$ such that $z:V \to \mathcal V'$ is a strict monoidal functor. Symmetrically, a {\em monoidal vertical transformation} between strict monoidal double functors is a vertical transformation $v$ such that $v:H \to \mathcal H'$ is a strict monoidal functor.
A {\em monoidal double natural transformation} is a double natural transformation $(z,\theta,v)$ between strict monoidal double functors $\mathbf F$ and $\mathbf G:(\mathbb D,\otimes,I) \to (\mathbb D',\otimes',I')$ such that $z:V \to \mathcal V'$ and $v:H \to \mathcal H'$ are strict monoidal functors and $\theta$, seen as a natural transformation in either way, is monoidal. That is to say, the equalities of 2-cells
$$
\xymatrix{
(\mathbf F A)(\mathbf F B) \ar[r]^-{z_A z_B} \ar[d]_-{v_A v_B} 
\ar@{}[rd]|-{\Longdownarrow{\theta_A \theta_B}} &
(\mathbf G A)(\mathbf G B) \ar@{=}[d] \\
(\mathbf G A)(\mathbf G B) \ar@{=}[r] &
(\mathbf G A)(\mathbf G B) }
\raisebox{-17pt}{$=$}
\xymatrix{
\mathbf F (AB) \ar[r]^-{z_{AB}} \ar[d]_(.6){v_{AB}} 
\ar@{}[rd]|-{\Longdownarrow{\theta_{AB}}} &
\mathbf G (AB) \ar@{=}[d] \\
\mathbf G (AB) \ar@{=}[r] &
\mathbf G (AB) }
\qquad
\xymatrix@R=28pt{
\mathbf F I \ar[r]^-{z_I} \ar[d]_-{v_I} \ar@{}[rd]|-{\Longdownarrow{\theta_I}} &
\mathbf G I \ar@{=}[d] \\
\mathbf G I \ar@{=}[r] &
\mathbf G I}
\raisebox{-17pt}{$=$}
\xymatrix@R=28pt{
I' \ar@{=}[r] \ar@{=}[d] \ar@{}[rd]|-{\Longdownarrow{1}} &
I' \ar@{=}[d] \\
I' \ar@{=}[r] &
I'}
$$
hold for all 0-cells $A$ and $B$ in $\mathbb D$.
\end{definition}

Strict monoidal double categories are the 0-cells and strict monoidal double functors are the 1-cells in various 2-categories. The 2-cells can be chosen to be monoidal horizontal transformations, monoidal vertical transformations, or monoidal double natural transformations.

\begin{definition} \label{def:smdoublecat} \cite[Definition 5]{BruniMeseguerMontanari}
A {\em symmetric strict monoidal double category} consists of a strict monoidal double category $(\mathbb D,\otimes,I)$ together with a double natural transformation $(x,\sigma,y):\otimes\to \otimes.\mathbf X$ (where $\mathbf X$ denotes the relevant component of the symmetry of the 2-category $\mathsf{DblCat}$, that is, the flip double functor $\mathbb D \times \mathbb D\to \mathbb D \times \mathbb D$). It is required to be involutive in the sense that
$$
\xymatrix{
\otimes \ar[r]^-{x} \ar[d]_-{y} \ar@{}[rd]|-{\Longdownarrow \sigma} &
\otimes.\mathbf X \ar[r]^-{x.1} \ar@{=}[d] \ar@{}[rd]|-{\Longdownarrow 1} &
\otimes \ar@{=}[d] \\
\otimes.\mathbf X \ar[d]_-{y.1} \ar@{=}[r] \ar@{}[rd]|-{\Longdownarrow 1} &
\otimes.\mathbf X \ar[d]|-{y.1} \ar[r]^-{x.1}  \ar@{}[rd]|-{\Longdownarrow{\sigma.1}} &
\otimes \ar@{=}[d] \\
\otimes \ar@{=}[r] &
\otimes \ar@{=}[r] &
\otimes  } 
\quad \raisebox{-40pt}{$=$}\quad
\raisebox{-20pt}{$
\xymatrix{
\otimes \ar@{=}[d] \ar@{=}[r] \ar@{}[rd]|-{\Longdownarrow 1} &
\otimes \ar@{=}[d] \\
\otimes \ar@{=}[r] &
\otimes}$}
$$
and the {\em hexagon condition} is required to hold, which says that the 2-cells of Figure \ref{fig:doublebraid} are equal.
\begin{amssidewaysfigure}
\centering
\hspace*{-1.5cm}
\scalebox{.95}{$
\xymatrix@C=1pt{
\otimes.(\otimes \!\times  \! 1) \ar[r]^-{\raisebox{8pt}{${}_{1.(x \times 1)}$}} 
\ar[d]_-{1.(y \times 1)} 
\ar@{}[rrdd]|-{\Longdownarrow {1.(\sigma \times 1)}} &
\otimes.(\otimes  \! \times  \! 1).(\mathbf X  \! \times \! 1) \ar@{=}[r] & 
\otimes.(1  \! \times  \! \otimes).(\mathbf X  \! \times \! 1) \ar@{=}[dd]  
\ar[rr]^-{\raisebox{8pt}{${}_{1.(1 \times x ).1}$}} \ar@{}[rrdd]|-{\Longdownarrow 1} &&
\otimes.(1  \! \times \! \otimes).(1 \! \times \! \mathbf X).(\mathbf X \! \times \! 1) \ar@{=}[dd] \\
\otimes.(\otimes  \! \times \! 1).(\mathbf X  \! \times  \! 1) 
\ar@{=}[d]   &&
\\
\otimes.(1  \! \times  \! \otimes).(\mathbf X  \! \times \! 1) \ar@{=}[rr]  \ar[d]_-{1.(1 \times y ).1} 
\ar@{}[rrd]|-{\Longdownarrow 1} &&
\otimes.(1  \! \times  \!\otimes).(\mathbf X  \!\times \! 1) \ar[d]_-{1.(1 \times y ).1} \ar[rr]^-{1.(1 \times x ).1} 
\ar@{}[rrd]|-{\Longdownarrow {1.(1 \times \sigma).1}} &&
\otimes.(1  \!\times \! \otimes).(1 \!\times \! \mathbf X).(\mathbf X  \!\times \! 1) \ar@{=}[d] \\
\otimes.(1  \!\times \! \otimes).(1 \! \times  \! \mathbf X).(\mathbf X  \!\times \! 1) \ar@{=}[rr] &&
\otimes.(1  \!\times  \!\otimes).(1 \!\times \! \mathbf X).(\mathbf X  \!\times  \!1) \ar@{=}[rr] &&
\otimes.(1  \!\times \! \otimes).(1 \!\times \! \mathbf X).(\mathbf X \! \times \! 1) \\
\\
\\
\otimes.(\otimes \times 1) \ar@{=}[d] \ar@{=}[r] \ar@{}[rdd]|-{\Longdownarrow {1}} &
\otimes.(1 \times \otimes) \ar[r]^-{x.1} \ar[dd]_-{y.1} \ar@{}[rdd]|-{\Longdownarrow {\sigma.1}} &
\otimes.\mathbf X.(1 \times \otimes) \ar@{=}[dd] \ar@{=}[r] &
\otimes.(\otimes \times 1).(1\times \mathbf X).(\mathbf X \times 1) \ar@{=}[r] &
\otimes.(1 \times \otimes).(1\times \mathbf X).(\mathbf X \times 1) \ar@{=}[dddd] \\
\otimes.(1 \times \otimes) \ar[d]_-{y.1} \\
\otimes.\mathbf X.(1 \times \otimes) \ar@{=}[d] \ar@{=}[r] &
\otimes.\mathbf X.(1 \times \otimes) \ar@{=}[r] &
\otimes.\mathbf X.(1 \times \otimes) \ar@{}[rrdd]|-{\Longdownarrow 1} \\
\otimes.(\otimes \times 1).(1\times \mathbf X).(\mathbf X \times 1) \ar@{=}[d] \\ 
\otimes.(1 \times \otimes).(1\times \mathbf X).(\mathbf X \times 1) \ar@{=}[rrrr] &&&&
\otimes.(1 \times \otimes).(1\times \mathbf X).(\mathbf X \times 1).}$}
\caption{Hexagon condition for a symmetric strict monoidal double category}
\label{fig:doublebraid}
\end{amssidewaysfigure}
\end{definition}

The symmetry $(x,\sigma,y)$ of a symmetric strict monoidal double category $\mathbb D$ induces symmetries on all constituent categories $V$, $H$, $\mathcal V$ and $\mathcal H$.

It follows from Definition \ref{def:smdoublecat} that for any 0-cell $A$,
$$
\xymatrix{
IA\ar[r]^-x \ar[d]_-y \ar@{}[rd]|-{\Longdownarrow {\sigma}}  &
AI \ar@{=}[d] \\
AI \ar@{=}[r] &
AI}
\raisebox{-17pt}{$\quad = \quad$}
\xymatrix{
A \ar@{=}[r]  \ar@{=}[d]  \ar@{}[rd]|-{\Longdownarrow {1}}  &
A \ar@{=}[d] \\
A \ar@{=}[r] &
A}
\raisebox{-17pt}{$\quad = \quad$}
\xymatrix{
AI\ar[r]^-x \ar[d]_-y \ar@{}[rd]|-{\Longdownarrow {\sigma}}  &
IA \ar@{=}[d] \\
IA \ar@{=}[r] &
IA.}
$$
Recall from \cite[Section 4]{BruniMeseguerMontanari} that in a strict monoidal double category $(\mathbb D,\otimes,I)$ with symmetry $(x,\sigma,y)$, the 2-cell $\sigma$ has a horizontal inverse $\sigma^h$ on the left, and a vertical inverse $\sigma^v$ on the right of
$$
\xymatrix{
AB \ar[r]^-x \ar[d]_-y \ar@{}[rd]|-{\Longdownarrow \sigma} &
BA \ar@{=}[d] \\
BA \ar@{=}[r] \ar[d]_-y  \ar@{}[rd]|-{\Longdownarrow 1} &
BA \ar[d]^-y \\
AB \ar@{=}[r] &
AB} \qquad\qquad
\xymatrix{
AB \ar[r]^-x \ar[d]_-y \ar@{}[rd]|-{\Longdownarrow \sigma} &
BA \ar[r]^-x \ar@{=}[d] \ar@{}[rd]|-{\Longdownarrow 1} &
AB \ar@{=}[d] \\
BA \ar@{=}[r] &
BA \ar[r]_-x &
AB.}
$$ 
From these explicit expressions it follows for any vertical 1-cells $f:A\to A'$ and $g:B\to B'$ that
\begin{equation} \label{eq:sigma_id}
\xymatrix{
AB \ar[r]^-x \ar@{=}[d] \ar@{}[rd]|-{\Longdownarrow {\sigma^h}} &
BA \ar@{=}[r] \ar[d]^-y \ar@{}[rdd]|-{\Longdownarrow 1} &
BA \ar@{=}[r] \ar[d]_-{gf} \ar@{}[rd]|-{\Longdownarrow 1} &
BA \ar[d]^-{gf} \\
AB \ar@{=}[r] \ar[d]_-{fg} \ar@{}[rd]|-{\Longdownarrow 1} &
AB \ar[d]^-{fg} &
B'A' \ar[d]_-y \ar@{=}[r] \ar@{}[rd]|-{\Longdownarrow {\sigma^v}} &
B'A' \ar@{=}[d] \\
A'B' \ar@{=}[r] &
A'B' \ar@{=}[r] &
A'B' \ar[r]_-x &
B'A'}
\raisebox{-40pt}{$\qquad = \qquad$}
\raisebox{-20pt}{$
\xymatrix{
AB \ar[r]^-x \ar[d]_-{fg} \ar@{}[rd]|-{\Longdownarrow x} &
BA \ar[d]^-{gf} \\
A'B' \ar[r]_-x &
B'A'}$}
\end{equation}
and it also follows for any horizontal 1-cells $f:A\to A'$ and $g:B\to B'$ that
\begin{equation} \label{eq:sigma_y_id}
\xymatrix{
AB \ar[r]^-x \ar@{=}[d] \ar@{}[rd]|-{\Longdownarrow {\sigma^h}} &
BA \ar[r]^-{gf} \ar[d]^-y \ar@{}[rd]|-{\Longdownarrow {y}} &
B'A' \ar@{=}[r] \ar[d]^-y \ar@{}[rd]|-{\Longdownarrow {\sigma^v}} &
B'A' \ar@{=}[d] \\
AB \ar@{=}[r] &
AB \ar[r]_-{fg} &
A'B' \ar[r]_-x &
B'A'}
\raisebox{-17pt}{$\quad=\quad$}
\xymatrix{
AB \ar[r]^-x \ar@{=}[d] \ar@{}[rrd]|-{\Longdownarrow 1} &
BA\ar[r]^-{gf} &
B'A' \ar@{=}[d] \\
AB \ar[r]_-{fg} &
A'B' \ar[r]_-x &
B'A'.}
\end{equation}

\begin{definition}
A {\em symmetric strict monoidal double functor} is a strict monoidal double functor $\mathbf F:(\mathbb D,\otimes,I) \to (\mathbb D',\otimes',I')$ which preserves the symmetry in the sense of the equalities
$$
\mathbf F.x=x'.(\mathbf F \times \mathbf F)\qquad
\mathbf F.y=y'.(\mathbf F \times \mathbf F)\qquad
\mathbf F.\sigma=\sigma'.(\mathbf F\times \mathbf F).
$$
(of functors $V\times V \to \mathcal V'$ and $H\times H \to \mathcal H'$, and of natural transformations, respectively).  
\end{definition}

\begin{definition}
Consider  some symmetric strict monoidal double functors $\mathbf F, \mathbf G:\mathbb D \to \mathbb D'$.
A monoidal horizontal transformation $h:\mathbf F \to \mathbf G$ is said to be {\em symmetric} if the strict monoidal functor $h:V\to \mathcal V'$ is symmetric.
Symmetrically, a monoidal vertical transformation $v:\mathbf F \to \mathbf G$ is said to be {\em symmetric} if the strict monoidal functor $v:H\to \mathcal H'$ is symmetric.
A monoidal double natural transformation $(h,\theta, v):\mathbf F \to \mathbf G$ is said to be {\em symmetric} if $h$ and $v$ are symmetric.
\end{definition}

Symmetric strict monoidal double categories are the 0-cells and symmetric strict mono\-idal double functors are the 1-cells in various 2-categories. The 2-cells can be chosen to be symmetric monoidal horizontal transformations, symmetric monoidal vertical transformations or symmetric monoidal  double natural transformations. We denote this latter 2-category by $\mathsf{sm}\mbox{-}\mathsf{DblCat}$. Its vertical and horizontal compositions are given by the same expressions as in $\mathsf{DblCat}$. 

This finishes our review of the basic theory of double categories. 

\begin{example} \label{ex:Dpq}
(1) For any strict monoidal double category $\mathbb D$ there is a double category $\mathbb D_{01}$ which we describe next. Its 0-cells are the pseudomonoids 
$$
(A,
\xymatrix@C=15pt{AA \ar[r]^-m & A},
\xymatrix@C=15pt{I \ar[r]^-u & A},
\raisebox{17pt}{$\xymatrix@C=15pt@R=20pt{
AAA\ar@{=}[d]\ar[r]^-{m1} 
\ar@{}[rrd]|-{\Longdownarrow \alpha}& 
AA \ar[r]^-m &
A \ar@{=}[d]\\
AAA \ar[r]_-{1m} &
AA \ar[r]_-m &
A}$},
\raisebox{17pt}{$\xymatrix@C=15pt@R=20pt{
A\ar@{=}[d]\ar[r]^-{u1} 
\ar@{}[rrd]|-{\Longdownarrow \lambda}& 
AA \ar[r]^-m &
A \ar@{=}[d]\\
A \ar@{=}[rr] &&
A}$},
\raisebox{17pt}{$\xymatrix@C=15pt@R=20pt{
A\ar@{=}[d]\ar[r]^-{1u} 
\ar@{}[rrd]|-{\Longdownarrow \varrho}& 
AA \ar[r]^-m &
A \ar@{=}[d]\\
A \ar@{=}[rr] &&
A}$})
$$
in the horizontal 2-category $\mathsf{Hor(\mathbb D)}$ of $\mathbb D$ (which we call the {\em horizontal pseudomonoids} in $\mathbb D$), and the horizontal 1-cells are the opmonoidal 1-cells 
$$
(\xymatrix@C=15pt{A \ar[r]^-h &A'},
\raisebox{17pt}{$\xymatrix@C=15pt@R=20pt{
AA\ar@{=}[d]\ar[r]^-{m} 
\ar@{}[rrd]|-{\Longdownarrow {\chi^2}}& 
A \ar[r]^-h &
A' \ar@{=}[d]\\
AA \ar[r]_-{hh} &
A'A' \ar[r]_-{m'} &
A'}$},
\raisebox{17pt}{$\xymatrix@C=15pt@R=20pt{
I\ar@{=}[d]\ar[r]^-{u} 
\ar@{}[rrd]|-{\Longdownarrow {\chi^0}}& 
A \ar[r]^-h &
A' \ar@{=}[d]\\
I \ar[rr]_-{u'}& &
A'}$})
$$ 
in the horizontal 2-category of $\mathbb D$. A vertical 1-cell consists of a vertical 1-cell on the left, together with 2-cells on the right of 
$$
\xymatrix{
A \ar[d]^-f \\
A'}\qquad\qquad
\xymatrix{
AA \ar[r]^-{m} \ar[d]_-{ff} \ar@{}[rd]|-{\Longdownarrow {\varphi^2}} &
A \ar[d]^-f \\
A'A'\ar[r]_-{m'} &
A'} \qquad\qquad
\xymatrix{
I \ar[r]^-{u} \ar@{=}[d] \ar@{}[rd]|-{\Longdownarrow {\varphi^0}} &
A \ar[d]^-f \\
I\ar[r]_-{u'} &
A'}
$$
which are required to satisfy the following coassociativity and counitality conditions.
$$
\xymatrix{
AAA \ar[r]^-{m1} \ar@{=}[d] \ar@{}[rrd]|-{\Longdownarrow \alpha} &
AA \ar[r]^-m &
A \ar@{=}[d] \\
AAA \ar[r]^-{1m} \ar[d]_-{fff} \ar@{}[rd]|-{\Longdownarrow {1\varphi^2}} &
AA \ar[r]^-{m} \ar[d]|-{ff}  \ar@{}[rd]|-{\Longdownarrow {\varphi^2}} &
A \ar[d]^-f \\
A'A'A' \ar[r]_-{1m'} &
A'A' \ar[r]_-{m'} &
A'}
\raisebox{-40pt}{$\quad=\quad$}
\xymatrix{
AAA \ar[r]^-{m1} \ar[d]_-{fff} \ar@{}[rd]|-{\Longdownarrow {\varphi^2 1}} &
AA \ar[r]^-{m} \ar[d]|-{ff}  \ar@{}[rd]|-{\Longdownarrow {\varphi^2}} &
A \ar[d]^-f \\
A'A'A '\ar[r]^-{m'1} \ar@{=}[d] \ar@{}[rrd]|-{\Longdownarrow {\alpha'}} &
A'A' \ar[r]^-{m'} &
A' \ar@{=}[d] \\
A'A'A' \ar[r]_-{1m'} &
A'A' \ar[r]_-{m'} &
A'}
$$
$$
\xymatrix@C=35pt{
A \ar[r]^-{u1} \ar@{=}[d] \ar@{}[rrd]|-{\Longdownarrow \lambda} &
AA \ar[r]^-m &
A \ar@{=}[d] \\
A \ar@{=}[rr] \ar[d]_-{f} \ar@{}[rrd]|-{\Longdownarrow 1} &&
A \ar[d]^-f \\
A'\ar@{=}[rr]   &&
A'}
\raisebox{-40pt}{$\quad=\quad$}
\xymatrix@C=35pt{
A \ar[r]^-{u1} \ar[d]_-{f} \ar@{}[rd]|-{\Longdownarrow {\varphi^0 1}} &
AA \ar[r]^-{m} \ar[d]|-{ff}  \ar@{}[rd]|-{\Longdownarrow {\varphi^2}} &
A \ar[d]^-f \\
A '\ar[r]^-{u'1} \ar@{=}[d] \ar@{}[rrd]|-{\Longdownarrow {\lambda'}} &
A'A' \ar[r]^-{m'} &
A' \ar@{=}[d] \\
A' \ar@{=}[rr] &&
A'}
$$
$$
\xymatrix@C=35pt{
A \ar[r]^-{1u} \ar@{=}[d] \ar@{}[rrd]|-{\Longdownarrow\varrho} &
AA \ar[r]^-m &
A \ar@{=}[d] \\
A \ar@{=}[rr] \ar[d]_-{f} \ar@{}[rrd]|-{\Longdownarrow 1} &&
A \ar[d]^-f \\
A'\ar@{=}[rr]   &&
A'}
\raisebox{-40pt}{$\quad=\quad$}
\xymatrix@C=35pt{
A \ar[r]^-{1u} \ar[d]_-{f} \ar@{}[rd]|-{\Longdownarrow {1\varphi^0}} &
AA \ar[r]^-{m} \ar[d]|-{ff}  \ar@{}[rd]|-{\Longdownarrow {\varphi^2}} &
A \ar[d]^-f \\
A '\ar[r]^-{1u'} \ar@{=}[d] \ar@{}[rrd]|-{\Longdownarrow {\varrho'}} &
A'A' \ar[r]^-{m'} &
A' \ar@{=}[d] \\
A' \ar@{=}[rr] &&
A'}
$$
A 2-cell in $\mathbb D_{01}$ with boundaries on the left of
$$
\xymatrix@C=35pt{
A \ar[r]^-{(h,\chi^2,\chi^0)} \ar[d]_-{(f,\varphi^2,\varphi^0)} &
C \ar[d]^-{(g,\gamma^2,\gamma^0)} \\
B \ar[r]_-{(k,\kappa^2,\kappa^0)} &
D}  \qquad \qquad
\xymatrix{
A \ar[r]^-h \ar[d]_-f \ar@{}[rd]|-{\Longdownarrow \omega} &
C \ar[d]^-g \\
B \ar[r]_-k &
D}
$$
is a 2-cell of $\mathbb D$ on the right, subject to the opmonoidality conditions
$$
\xymatrix{
AA \ar[r]^-m \ar@{=}[d] \ar@{}[rrd]|-{\Longdownarrow {\chi^2}} &
A \ar[r]^-h &
C \ar@{=}[d] \\
AA \ar[r]^-{hh} \ar[d]_-{ff} \ar@{}[rd]|-{\Longdownarrow {\omega\omega}} &
CC \ar[r]^-m \ar[d]|-{gg} \ar@{}[rd]|-{\ \Longdownarrow {\gamma^2}} &
C \ar[d]^-g \\
BB \ar[r]_-{kk} &
DD \ar[r]_-m &
D}
\raisebox{-40pt}{$\quad=\quad$}
\xymatrix{
AA \ar[r]^-m \ar[d]_-{ff} \ar@{}[rd]|-{\ \Longdownarrow {\varphi^2}} &
A \ar[r]^-h \ar[d]|-f \ar@{}[rd]|-{\Longdownarrow \omega} &
C \ar[d]^-g \\
BB \ar[r]^-m \ar@{=}[d] \ar@{}[rrd]|-{\Longdownarrow {\kappa^2}} &
B \ar[r]^-k &
D \ar@{=}[d] \\
BB \ar[r]_-{kk} &
DD \ar[r]_-m &
D}
$$
$$
\xymatrix@C=35pt{
I \ar[r]^-u \ar@{=}[d] \ar@{}[rrd]|-{\Longdownarrow {\chi^0}} &
A \ar[r]^-h &
C \ar@{=}[d] \\
I  \ar[rr]^-u \ar@{=}[d] \ar@{}[rrd]|-{\ \Longdownarrow {\gamma^0}} &&
C \ar[d]^-g \\
I \ar[rr]_-u &&
D}
\raisebox{-40pt}{$\quad=\quad$}
\xymatrix@C=35pt{
I \ar[r]^-u \ar@{=}[d] \ar@{}[rd]|-{\ \Longdownarrow {\varphi^0}} &
A \ar[r]^-h \ar[d]|-f \ar@{}[rd]|-{\Longdownarrow \omega} &
C \ar[d]^-g \\
I \ar[r]^-u \ar@{=}[d] \ar@{}[rrd]|-{\Longdownarrow {\kappa^0}} &
B \ar[r]^-k &
D \ar@{=}[d] \\
I \ar[rr]_-u &&
D}
$$
(where all occurring pseudomonoid structures are denoted by the same symbols $(m,u,$ $\alpha,\lambda,\varrho)$).
Horizontal 1-cells are composed by the composition rule of opmonoidal 1-cells in a 2-category, see Example \ref{ex:Mpq}.
Vertical 1-cells are vertically composed according to the rule
$$
\xymatrix{
A \ar[d]^-f \\
B \ar[d]^-g \\
C}\qquad \qquad
\xymatrix{
AA \ar[r]^-{m} \ar[d]_-{ff} \ar@{}[rd]|-{\Longdownarrow {\varphi^2}} &
A \ar[d]^-f \\
BB\ar[r]^-{m} \ar[d]_-{gg} \ar@{}[rd]|-{\Longdownarrow {\gamma^2}}&
B\ar[d]^-g \\
CC \ar[r]_-m &
C} \qquad\qquad
\xymatrix{
I \ar[r]^-{u} \ar@{=}[d] \ar@{}[rd]|-{\Longdownarrow {\varphi^0}} &
A \ar[d]^-f \\
I\ar[r]^-{u} \ar@{=}[d] \ar@{}[rd]|-{\Longdownarrow {\gamma^0}}&
B\ar[d]^-g \\
I \ar[r]_-u &
C.}
$$
Horizontal and vertical compositions of 2-cells are computed in $\mathbb D$. The horizontal 2-category of the resulting double category $\mathbb D_{01}$ can be obtained by applying the 2-functor $(-)_{01}$ of Example \ref{ex:Mpq} to the horizontal 2-category of $\mathbb D$. That is, $\mathsf{Hor}(\mathbb D)_{01}\cong \mathsf{Hor}(\mathbb D_{01})$. (The analogous statement fails to hold for the vertical 2-category.)

If in addition the strict monoidal double category $(\mathbb D,\otimes,I)$ has a symmetry $(x,\sigma,y)$, then the above double category $\mathbb D_{01}$ can be equipped with a strict monoidal structure. The monoidal unit is the monoidal unit $I$ of $\mathbb D$ seen as a trivial pseudomonoid, see Example \ref{ex:Mpq}.
The monoidal product of 0-cells and of horizontal 1-cells is their monoidal product as pseudomonoids and opmonoidal 1-cells, respectively,  in the horizontal 2-category, see again Example \ref{ex:Mpq}. The monoidal product of vertical 1-cells $(f,\varphi^2,\varphi^0):A\to A'$ and $(g,\gamma^2,\gamma^0):B\to B'$ is
$$
\xymatrix{
AB \ar[d]^-{fg} \\
A'B'} \qquad 
\xymatrix{
ABAB \ar[r]^-{1x1} \ar[d]_-{fgfg} \ar@{}[rd]|-{\Longdownarrow {1x1}} &
AABB \ar[r]^-{mm} \ar[d]|-{ffgg} \ar@{}[rd]|-{\Longdownarrow {\varphi^2\gamma^2}} &
AB \ar[d]^{fg} \\
A'B'A'B' \ar[r]_-{1x1} &
A'A'B'B' \ar[r]_-{m'm'} &
A'B'}
\qquad 
\xymatrix{
I\ar[r]^-{uu} \ar@{=}[d] \ar@{}[rd]|-{\Longdownarrow {\varphi^0\gamma^0}} &
AB \ar[d]^{fg} \\
I' \ar[r]_-{u'u'} &
A'B'.}
$$
The monoidal product of 2-cells is their monoidal product as 2-cells of $\mathbb D$.
What is more, the above strict monoidal double category $\mathbb D_{01}$ has a symmetry with components 
$$
\xymatrix{
AB \ar[r]^-{(x,1,1)} \ar[d]_-{(y,\nu^2,\nu^0)} \ar@{}[rd]|-{\Longdownarrow \sigma} &
BA\ar@{=}[d] \\
BA \ar@{=}[r] &
BA}
$$
where the binary part $\nu^2$ and the nullary part $\nu^0$ of the opmonoidal 1-cell $y$ are the 2-cells
$$
\xymatrix{
ABAB \ar[rrr]^-{1x1} \ar@{=}[d] \ar@{}[rrrd]|-{\Longdownarrow 1} &&&
AABB \ar[r]^-{mm} \ar@{=}[d] \ar@{}[rd]|-{\Longdownarrow 1} &
AB \ar@{=}[d] \\
ABAB \ar[r]^-{xx} \ar@{=}[d] \ar@{}[rrd]|-{\Longdownarrow 1} &
BABA \ar[r]^-{1x1} & 
BBAA \ar[r]^-{x_{22}} \ar@{=}[d] \ar@{}[rrd]|-{\Longdownarrow 1} &
AABB \ar[r]^-{mm} &
AB \ar@{=}[d] \\
ABAB \ar[r]^-{xx} \ar[d]_{yy} \ar@{}[rd]|-{\Longdownarrow {\sigma\sigma}} &
BABA \ar[r]^-{1x1} \ar@{=}[d] \ar@{}[rrd]|-{\Longdownarrow 1} & 
BBAA \ar[r]^-{mm}  &
BA \ar@{=}[d] \ar[r]^-x \ar@{}[rd]|-{\Longdownarrow {\sigma^h}} &
AB \ar[d]^-y \\
BABA \ar@{=}[r] &
BABA \ar[r]_-{1x1} &
BBAA \ar[r]_-{mm} &
BA \ar@{=}[r] &
BA}\qquad
\xymatrix@R=43pt{
I \ar[rr]^-{uu} \ar@{=}[d] \ar@{}[rrd]|-{\Longdownarrow 1} &&
AB \ar@{=}[d] \\
I \ar[r]^-{uu} \ar@{=}[d] \ar@{}[rd]|-{\Longdownarrow 1} &
BA \ar[r]^-x \ar@{=}[d]\ar@{}[rd]|-{\Longdownarrow {\sigma^h}} &
AB \ar[d]^-y \\
I \ar[r]_-{uu} &
BA \ar@{=}[r] &
BA}
$$
respectively, (where $\sigma^h$ stands for the horizontal inverse of $\sigma$ in $\mathbb D$).

For any strict monoidal double functor $\mathbf F:\mathbb D \to \mathbb D'$ there is a double functor $\mathbf F_{01}:\mathbb D_{01} \to \mathbb D'_{01}$. On the horizontal 2-category it acts as the 2-functor obtained by applying the 2-functor $(-)_{01}$ of Example \ref{ex:Mpq} to the 2-functor induced by $\mathbf F$ between the horizontal 2-categories. It takes a vertical 1-cell $(f,\varphi^2,\varphi^0)$ to
$$
\xymatrix{
\mathbf F A \ar[d]^-{\mathbf F f} \\
\mathbf F B} \qquad
\xymatrix{
(\mathbf F A) (\mathbf F A) \ar[d]_-{(\mathbf F f)(\mathbf F f)} \ar@{=}[r] \ar@{}[rd]|-{\Longdownarrow 1} &
\mathbf F (AA) \ar[d]|-{\mathbf F (ff)} \ar[r]^-{\mathbf F m} \ar@{}[rd]|-{\Longdownarrow {\mathbf F \varphi^2}} &
\mathbf F A \ar[d]^-{\mathbf F f} \\
(\mathbf F B)(\mathbf F B) \ar@{=}[r] &
\mathbf F (BB) \ar[r]_-{\mathbf F m} &
\mathbf F B} \qquad
\xymatrix{
I \ar@{=}[d] \ar@{=}[r] \ar@{}[rd]|-{\Longdownarrow 1} &
\mathbf F I \ar@{=}[d] \ar[r]^-{\mathbf F u}\ar@{}[rd]|-{\Longdownarrow {\mathbf F \varphi^0}}  &
\mathbf F A \ar[d]^-{\mathbf F f} \\
I \ar@{=}[r] &
\mathbf F I \ar[r]_-{\mathbf F u} &
\mathbf F B} 
$$
and it acts on the 2-cells as $\mathbf F$.

If the strict monoidal double categories $(\mathbb D,\otimes,I)$ and $(\mathbb D',\otimes',I') $ are symmetric and $\mathbf F$ preserves the symmetry, then  $\mathbf F_{01}$ is symmetric strict monoidal with respect to the symmetric strict monoidal structures of $\mathbb D_{01}$ and $\mathbb D' _{01}$ in the previous paragraphs.

Finally, take any strict monoidal double functors $\mathbf F$ and $\mathbf G$.
 For any monoidal horizontal transformation $z:\mathbf F \to \mathbf G$, there is a horizontal transformation $z_{01}:\mathbf F_{01} \to \mathbf G_{01}$ with the components $(z_A,1,1)$ at any pseudomonoid $(A,m,u,\alpha,\lambda,$
$\varrho)$ in $\mathbb D$, and $x_f$ at any vertical 1-cell $f$ in $\mathbb D$. 
Whenever $\mathbf F$ and $\mathbf G$ are symmetric as well --- so that $\mathbf F_{01}$ and $\mathbf G_{01}$ are strict monoidal too --- $z_{01}$ is monoidal too; and it is symmetric if $z$ is so.

For any monoidal vertical transformation $v:\mathbf F \to \mathbf G$ there is a vertical transformation $v_{01}:\mathbf F_{01} \to \mathbf G_{01}$ with the components $(v_A,v_m,v_u)$ at any pseudomonoid $(A,m,u,\alpha,\lambda,\varrho)$  in $\mathbb D$; and $v_h$ at any horizontal 1-cell $h$ in $\mathbb D$. 
It is symmetric monoidal  whenever $v$ is a symmetric monoidal vertical transformation between  symmetric strict monoidal double functors.

For any monoidal double natural transformation $(z,\theta,v):\mathbf F \to \mathbf G$ there is a double natural transformation $\mathbf F_{01} \to \mathbf G_{01}$ with the components $z_{01}$, $v_{01}$ and
$$
\xymatrix{
\mathbf F A \ar[r]^-{(z,1,1)} \ar[d]_-{(v,v_m,v_u)} \ar@{}[rd]|-{\Longdownarrow \theta} &
\mathbf G A \ar@{=}[d] \\
\mathbf G A \ar@{=}[r] &
\mathbf G A} 
$$
at any pseudomonoid $(A,m,u,\alpha,\lambda,\varrho)$  in $\mathbb D$. It is monoidal  whenever $\mathbf F$, $\mathbf G$ and $v$ are symmetric as well; and symmetric if so is $z$ in addition.

These maps define various 2-functors $(-)_{01}$:
\begin{enumerate}[(i)]
\item from the 2-category of strict monoidal double categories, strict monoidal double functors and monoidal horizontal transformations to the 2-category of double categories, double functors and horizontal transformations,
\item from the 2-category of strict monoidal double categories, strict monoidal double functors and monoidal vertical transformations to the 2-category of double categories, double functors and vertical transformations,
\item from the 2-category of strict monoidal double categories, strict monoidal double functors and monoidal double natural transformations to the 2-category $\mathsf{DblCat}$ of double categories, double functors and double natural transformations,
\item from the 2-category of symmetric strict monoidal double categories, symmetric strict monoidal double functors and symmetric monoidal horizontal transformations to itself,
\item from the 2-category of symmetric strict monoidal double categories, symmetric strict monoidal double functors and symmetric monoidal vertical transformations  to itself,
\item from the 2-category $\mathsf{sm}\mbox{-}\mathsf{DblCat}$ of symmetric strict  monoidal double categories, symmetric strict monoidal double functors and  symmetric   monoidal double natural transformations to itself.
\end{enumerate}
 
(2) The `diagonal duality' 2-functors $(-)^*$ (see Example \ref{ex:D*}) induce 2-functors between the (domain, respectively, codomain) 2-categories in items (i) and (ii), as well as those in items (iv) and (v); and they induce 2-endofunctors on the 2-categories in items (iii) and (vi) of the list of part (1) above.
Hence symmetrically to part (1), we define 2-functors $(-)_{10}:=(((-)^*)_{01})^*$ of all kinds in this list. 
They send a strict monoidal double category $\mathbb D$ to the double category $\mathbb D_{10}$ whose 0-cells are the pseudomonoids in the vertical 2-category of $\mathbb D$ (i.e. the {\em vertical pseudomonoids} in $\mathbb D$), and the vertical 1-cells are the monoidal morphisms in the vertical 2-category of $\mathbb D$. The horizontal 1-cells are monoidal horizontal 1-cells with respect to the vertical pseudomonoids and the 2-cells are the 2-cells of $\mathbb D$ which are monoidal in a suitable sense. If in addition the strict monoidal double category $\mathbb D$ is symmetric then also $\mathbb D_{10}$ is symmetric strict monoidal.

(3) Next we show that the diagram
$$
\xymatrix{
\mathsf{sm}\mbox{-}\mathsf{DblCat}  \ar[r]^-{(-)_{01}} \ar[d]_-{(-)_{10}} &
\mathsf{sm}\mbox{-}\mathsf{DblCat} \ar[d]^-{(-)_{10}}  \\
\mathsf{sm}\mbox{-}\mathsf{DblCat} \ar[r]_-{(-)_{01}} & 
\mathsf{sm}\mbox{-}\mathsf{DblCat}}
$$
of the 2-functors in parts (1) and (2) commutes up-to 2-natural isomorphism. (The same proof shows, in fact, the commutativity of the 2-endofunctors $(-)_{01}$ and $(-)_{10}$ also in items (iv) and (v) of the list of part (1).)

For any symmetric strict monoidal double category $\mathbb D$, a  0-cell of $(\mathbb D_{01})_{10}$ is a vertical pseudomonoid in $\mathbb D_{01}$, 
$$
((A,m^\hori,u^\hori,\alpha^\hori,\lambda^\hori,\varrho^\hori),
(m^\verti,\zeta,\zeta_0),
(u^\verti,\zeta^0,\zeta^0_0),
\alpha^\verti,\lambda^\verti,\varrho^\verti).
$$
Symmetrically, a  0-cell of $(\mathbb D_{10})_{01}$ is a horizontal pseudomonoid in $\mathbb D_{10}$, 
$$
((A,m^\verti,u^\verti,\alpha^\verti,\lambda^\verti,\varrho^\verti),
(m^\hori,\xi,\xi^0),
(u^\hori,\xi_0,\xi^0_0),
\alpha^\hori,\lambda^\hori,\varrho^\hori).
$$
A bijective correspondence between them is given by $\zeta^0_0=\xi^0_0$, $\zeta_0=\xi_0$, $\zeta^0=\xi^0$ and $\zeta$ equal to
$$
\xymatrix@C=35pt@R=20pt{
AAAA \ar[r]^-{1x1} \ar[d]_-{1y1} \ar@{}[rd]|-{\Longdownarrow {1\sigma1}} &
AAAA \ar[r]^-{m^\hori m^\hori} \ar@{}[d] \ar@{=}[d] \ar@{}[rd]|-{\Longdownarrow{1}} &
AA \ar@{=}[d] \\
AAAA \ar@{=}[r] \ar[d]_-{1y1} \ar@{}[rrdd]|-{\Longdownarrow{\xi}} &
AAAA \ar[r]_-{m^\hori m^\hori} &
AA \ar[dd]^-{m^\verti} \\
AAAA \ar[d]_-{m^\verti m^\verti} \\
AA \ar[rr]_-{m^\hori} &&
A.}
$$
Similarly to Example \ref{ex:Mpq}~(3), the following properties are pairwise equivalent (for checking it use also \eqref{eq:sigma_id}).
\begin{itemize}
\item
Associativity and unitality of the monoidal horizontal 1-cell $(m^\hori,\xi,\xi^0)$ are equivalent to the compatibilities of the opmonoidal 2-cells $\alpha^\verti$, $\lambda^\verti$ and $\varrho^\verti$ with the binary parts of their source and target opmonoidal vertical 1-cells.
\item
Associativity and unitality of the monoidal horizontal 1-cell $(u^-,\xi_0,\xi^0_0)$ are equivalent to the compatibilities of the opmonoidal 2-cells $\alpha^\verti$, $\lambda^\verti$ and $\varrho^\verti$ with the nullary parts of their source and target opmonoidal  vertical 1-cells.
\item
Compatibilities of the monoidal 2-cells $\alpha^\hori$, $\lambda^\hori$ and $\varrho^\hori$  with the binary parts of their source and target monoidal  horizontal 1-cells are  equivalent to the coassociativity and counitality of the opmonoidal vertical 1-cell $(m^\verti,\zeta,\zeta_0)$.
\item
Compatibilities of the monoidal 2-cells $\alpha^\hori$, $\lambda^\hori$ and $\varrho^\hori$ with the nullary parts of their source and target monoidal  horizontal 1-cells are equivalent to the coassociativity and counitality of the opmonoidal vertical 1-cell $(u^\verti,\zeta^0,\zeta^0_0)$.
\end{itemize}

A horizontal 1-cell in $(\mathbb D_{01})_{10}$ is a monoidal horizontal 1-cell in $\mathbb D_{01}$,
$$
((f,\varphi^2,\varphi^0),\varphi_2,\varphi_0)
$$
while a horizontal 1-cell in $(\mathbb D_{10})_{01}$ is an opmonoidal horizontal 1-cell in $\mathbb D_{10}$,
$$
((f,\varphi_2,\varphi_0),\varphi^2,\varphi^0).
$$
Both of them amount to a monoidal structure $(\varphi_2,\varphi_0)$ on $f$ with respect to the vertical pseudomonoids and an opmonoidal structure $(\varphi^2,\varphi^0)$ on $f$ with respect to the horizontal pseudomonoids, which are subject to four compatibility conditions. 
\begin{itemize}
\item Compatibility of the opmonoidal 2-cell $\varphi_0$ with the nullary parts of its surrounding opmonoidal 1-cells coincides with the compatibility of the monoidal 2-cell $\varphi^0$ with the nullary parts of its surrounding monoidal 1-cells. 
\item Compatibility of the opmonoidal 2-cell $\varphi_0$ with the binary parts of its surrounding opmonoidal 1-cells is the same condition as the compatibility of the monoidal 2-cell $\varphi^2$ with the nullary parts of surrounding monoidal 1-cells. 
\item Symmetrically, compatibility of the opmonoidal 2-cell $\varphi_2$ with the nullary parts of its surrounding opmonoidal 1-cells coincides with the compatibility of the monoidal 2-cell $\varphi^0$ and the binary parts of its surrounding monoidal 1-cells. 
\item Compatibility of the opmonoidal 2-cell $\varphi_2$ with the binary parts of 
its surrounding opmonoidal 1-cells, and compatibility of the mono\-idal 2-cell $\varphi^2$ with the binary parts of its surrounding monoidal 1-cells are equivalent to each other (for its proof also \eqref{eq:sigma_y_id} is used).
\end{itemize}
So a bijective correspondence between the horizontal 1-cells of $(\mathbb D_{01})_{10}$ and $(\mathbb D_{10})_{01}$ is obtained by re-ordering the constituent 2-cells.

A bijection between the vertical 1-cells of $(\mathbb D_{01})_{10}$ and $(\mathbb D_{10})_{01}$ is obtained symmetrically.

Finally, the 2-cells both in $(\mathbb D_{01})_{10}$ and $(\mathbb D_{10})_{01}$ are those 2-cells of $\mathbb D$ which are both 
\begin{itemize}
\item monoidal (with respect to the monoidal structures of the surrounding 1-cells),
\item opmonoidal (with respect to the opmonoidal structures of these 1-cells). 
\end{itemize}
So there is a trivial (identity) bijection between them. 

The above bijections combine into an iso double functor $(\mathbb D_{01})_{10}\to (\mathbb D_{10})_{01}$ which is symmetric strict monoidal and 2-natural.

(4) By the commutativity of the first diagram of part (3), we may apply  in any order the 2-functor in its columns $p$ times and the 2-functor in its rows $q$ times. This yields a 2-functor $(-)_{pq}: \mathsf{sm}\mbox{-}\mathsf{DblCat} \to \mathsf{sm}\mbox{-}\mathsf{DblCat}$. 
Analogously we obtain 2-endofunctors $(-)_{pq}$ on the 2-categories in items (iv) and (v) of the list of part (1). 
For any symmetric strict monoidal double category $\mathbb D$, we term the 0-cells of $\mathbb D_{pq}$ as the {\em (p,q)-oidal objects} of $\mathbb D$. In particular, the $(0,1)$-oidal objects are the horizontal pseudomonoids in $\mathbb D$ while the $(1,0)$-oidal objects are the vertical pseudomonoids. Thus in contrast to part (4) of Example \ref{ex:Mpq}, the structure of $(p,q)$-oidal objects in $\mathbb D$ depends both on $p$ and $q$ not only their sum.
\end{example}

\begin{example} \label{ex:Sqr}
We may regard Ehresmann's {\em square} or {\em quintet} construction \cite{Ehresmann} as a 2-functor
$\Sqr :\mathsf{sm}\mbox{-}\mathsf{2Cat} \to \mathsf{sm}\mbox{-}\mathsf{DblCat}$.

Recall that for any 2-category $\mathcal M$, the 0-cells of the double category $\Sqr (\mathcal M)$ are the 0-cells of $\mathcal M$. Both the horizontal and the vertical 1-cells in $\Sqr (\mathcal M)$ are the 1-cells of $\mathcal M$. The 2-cells of the form
$$
\xymatrix{
A \ar[r]^-h \ar[d]_-f \ar@{}[rd]|-{\Longdownarrow \omega}& 
C\ar[d]^-g \\
B \ar[r]_-k &
D}
$$
in $\Sqr (\mathcal M)$ are the 2-cells $\omega:g.h \to k.f$ in $\mathcal M$. 
Both the horizontal composition of horizontal 1-cells, and the vertical composition of vertical 1-cells in $\Sqr (\mathcal M)$ are given by the composition of 1-cells in $\mathcal M$.
The horizontal composition results in
\begin{equation} \label{eq:Sqr_hor}
\xymatrix{
t.n.h \ar[r]^-{\vartheta.1} &
p.g.h \ar[r]^-{1.\omega} &
p.k.f}
\quad \textrm{for 2-cells \quad}
\raisebox{17pt}{$
\xymatrix{
A \ar[r]^-h \ar[d]_-f \ar@{}[rd]|-{\Longdownarrow \omega}& 
C\ar[d]^-g \\
B \ar[r]_-k &
D} \quad
\xymatrix{
C \ar[r]^-n \ar[d]_-g \ar@{}[rd]|-{\Longdownarrow \vartheta}& 
E\ar[d]^-t \\
D \ar[r]_-p &
F}$}
\end{equation}
while the vertical composition results in
\begin{equation} \label{eq:Sqr_vert}
\xymatrix{
p.g.h \ar[r]^-{1.\omega} &
p.k.f \ar[r]^-{\vartheta.1} &
t.n.f}
\quad \textrm{for 2-cells \quad}
\raisebox{17pt}{$
\xymatrix{
A \ar[r]^-h \ar[d]_-f \ar@{}[rd]|-{\Longdownarrow \omega}& 
C\ar[d]^-g \\
B \ar[r]_-k &
D} \quad
\xymatrix{
B \ar[d]_-n \ar[r]^-k \ar@{}[rd]|-{\Longdownarrow \vartheta}& 
D\ar[d]^-p \\
E \ar[r]_-t &
F.}$}
\end{equation}
Both the horizontal and the vertical 2-categories of $\Sqr(\mathcal M)$ are isomorphic to $\mathcal M$.

Whenever $\mathcal M$ possesses a strict monoidal structure, it induces an evident strict mono\-idal structure on $\Sqr (\mathcal M)$. If moreover $\mathcal M$ has a symmetry $x$ then a symmetry on $\Sqr (\mathcal M)$ is given by the horizontal and vertical transformations in the first two diagrams, and 2-cell part in the third diagram of
$$
\xymatrix{
AB \ar[r]^-x \ar[d]_-{fg} \ar@{}[rd]|-{\Longdownarrow 1} &
BA \ar[d]^-{gf} \\
A'B' \ar[r]_-x &
B'A'}\qquad
\xymatrix{
AB \ar[r]^-{fg} \ar[d]_-x \ar@{}[rd]|-{\Longdownarrow 1} &
A'B' \ar[d]^-x  \\
BA \ar[r]_{gf} &
B'A'} \qquad
\xymatrix{
AB \ar[r]^-x \ar[d]_-x \ar@{}[rd]|-{\Longdownarrow 1} &
BA \ar@{=}[d] \\
BA \ar@{=}[r] &
BA}
$$
for any 1-cells $f:A \to A'$ and $g:B \to B'$ in $\mathcal M$.

For any 2-functor $\mathsf F:\mathcal M \to \mathcal N$, there is a double functor $\Sqr (\mathsf F):\Sqr (\mathcal M) \to \Sqr (\mathcal N)$. Its action on the 0-cells is equal to the action of $\mathsf F$ on the 0-cells. Its action both on the horizontal and vertical 1-cells is equal to the action of $\mathsf F$ on the 1-cells. Its action on the 2-cells is equal to the action of $\mathsf F$ on the 2-cells. Whenever $\mathsf F$ is strict monoidal, clearly $\Sqr (\mathsf F)$ is strict monoidal; and whenever $\mathsf F$ is symmetric, so is $\Sqr (\mathsf F)$.

Finally, a 2-natural transformation $\omega:\mathsf F \to \mathsf G$ induces a double natural transformation $\Sqr (\omega):\Sqr (\mathsf F) \to \Sqr (\mathsf G)$ with the horizontal and vertical transformations in the first two diagrams, and 2-cell part in the third diagram of
$$
\xymatrix{
\mathsf F A \ar[r]^-{\omega_A} \ar[d]_-{\mathsf F f} \ar@{}[rd]|-{\Longdownarrow 1} &
\mathsf G A \ar[d]^-{\mathsf G f} \\
\mathsf F B \ar[r]_-{\omega_B} &
\mathsf G B} \qquad
\xymatrix{
\mathsf F A \ar[d]_-{\omega_A} \ar[r]^-{\mathsf F f} \ar@{}[rd]|-{\Longdownarrow 1} &
\mathsf F B \ar[d]^-{\omega_B} \\
\mathsf G A \ar[r]_-{\mathsf G f} &
\mathsf G B} \qquad
\xymatrix{
\mathsf F A \ar[r]^-{\omega_A} \ar[d]_-{\omega_A} \ar@{}[rd]|-{\Longdownarrow 1}  &
\mathsf G A \ar@{=}[d] \\
\mathsf G A \ar@{=}[r] &
\mathsf G A }
$$
for any 1-cell $f: A\to B$ in $\mathcal M$. It is symmetric monoidal whenever $\omega$ is monoidal.

The so constructed 2-functor $\Sqr :\mathsf{sm}\mbox{-}\mathsf{2Cat} \to \mathsf{sm}\mbox{-}\mathsf{DblCat}$, together with the 2-functors of Example \ref{ex:Mpq} and  Example \ref{ex:Dpq} in the columns, render strictly commutative the diagrams of 2-functors
$$
\xymatrix{
\mathsf{sm}\mbox{-}\mathsf{2Cat} \ar[r]^-{\Sqr } \ar[d]_-{(-)_{01}} &
\mathsf{sm}\mbox{-}\mathsf{DblCat} \ar[d]^-{(-)_{01}} \\
\mathsf{sm}\mbox{-}\mathsf{2Cat} \ar[r]_-{\Sqr }  &
\mathsf{sm}\mbox{-}\mathsf{DblCat} }\qquad
\xymatrix{
\mathsf{sm}\mbox{-}\mathsf{2Cat} \ar[r]^-{\Sqr } \ar[d]_-{(-)_{10}} &
\mathsf{sm}\mbox{-}\mathsf{DblCat} \ar[d]^-{(-)_{10}} \\
\mathsf{sm}\mbox{-}\mathsf{2Cat} \ar[r]_-{\Sqr }  &
\mathsf{sm}\mbox{-}\mathsf{DblCat}. }
$$
Then $\Sqr $ commutes also with the 2-functors $(-)_{pq}$ of Example \ref{ex:Mpq}~(4) and  Example \ref{ex:Dpq}~(4) for any non-negative integers $p$ and $q$.
\end{example}
 

\section{$(p,q)$-oidal monads}
\label{sec:pq-monad}

In this section we describe a 2-functor $\Mnd :\mathsf{sm}\mbox{-}\mathsf{2Cat} \to \mathsf{sm}\mbox{-}\mathsf{DblCat}$. Its object map sends a symmetric strict monoidal 2-category $\mathcal M$ to the double category of monads in $\Sqr (\mathcal M)$ in the sense of \cite{FioreGambinoKock} (so the horizontal 2-category of $\Mnd (\mathcal M)$ is $\mathsf{Mnd}(\mathcal M)$ and its vertical 2-category is $\mathsf{Mnd}_{\mathsf{op}}(\mathcal M)$). 
Together with the 2-functor $(-)_{pq}:\mathsf{sm}\mbox{-}\mathsf{2Cat} \to \mathsf{sm}\mbox{-}\mathsf{2Cat}$ of Example \ref{ex:Mpq}~(4) and the 2-functor $(-)_{pq}: \mathsf{sm}\mbox{-}\mathsf{DblCat} \to \mathsf{sm}\mbox{-}\mathsf{DblCat}$ of Example \ref{ex:Dpq}~(4) for any non-negative integers $p$ and $q$, it fits in the diagram
$$
\xymatrix{
\mathsf{sm}\mbox{-}\mathsf{2Cat} \ar[r]^-{\Mnd } \ar[d]_-{(-)_{pq}} &
\mathsf{sm}\mbox{-}\mathsf{DblCat} \ar[d]^-{(-)_{pq}} \\
\mathsf{sm}\mbox{-}\mathsf{2Cat} \ar[r]_-{\Mnd }  &
\mathsf{sm}\mbox{-}\mathsf{DblCat} }
$$
which commutes up-to 2-natural isomorphism. The 2-functors around it take a symmetric strict monoidal 2-category $\mathcal M$ to a symmetric strict monoidal double category $\Mnd (\mathcal M)_{pq} \cong \Mnd (\mathcal M_{pq})$. Thus the  $(p,q)$-oidal objects in $\Mnd (\mathcal M)$ are the same as the $(p,q)$-oidal monads in $\mathcal M$. Consequently, any symmetric strict monoidal double functor $\mathbf K:\Mnd (\mathcal M)\to \Sqr (\mathcal M)$ induces a symmetric strict monoidal double functor
$$
\xymatrix{
\Mnd (\mathcal M_{pq})\cong \Mnd (\mathcal M)_{pq}  \ar[r]^-{\mathbf K_{pq}} &
\Sqr (\mathcal M)_{pq}\cong \Sqr (\mathcal M_{pq})}
$$
whose object map sends a $(p,q)$-oidal monad in $\mathcal M$ to a $(p+q)$-oidal object in $\mathcal M$.

As anticipated above, for any 2-category $\mathcal M$ we take the double category $\Mnd (\mathcal M)$ of monads in the sense of \cite{FioreGambinoKock} in the double category $\Sqr (\mathcal M)$ of Example \ref{ex:Sqr}. Thus a 0-cell in $\Mnd (\mathcal M)$ is a monad in $\mathcal M$. Below it will be written as a pair $(A,t)$ where $A$ is the 0-cell part and $t:A\to A$ is the 1-cell part. Whenever needed, the multiplication and unit 2-cells will be denoted by $\mu:t.t \to t$ and $\eta:1\to t$, respectively, but in most cases they will not be explicitly written.
A horizontal 1-cell in $\Mnd (\mathcal M)$ is a 1-cell in $\mathsf{Mnd}(\mathcal M)$, see Example \ref{ex:Mnd}.
A vertical 1-cell in $\Mnd (\mathcal M)$ is a 1-cell in $\mathsf{Mnd}_{\mathsf{op}}(\mathcal M)$, see again Example \ref{ex:Mnd}. A 2-cell in $\Mnd (\mathcal M)$ with boundaries as in the first diagram of
$$
\xymatrix@R=23pt{
(A,t) \ar[r]^-{(h,\Xi)} \ar[d]_-{(f,\Phi)} \ar@{}[rd]|-{\Longdownarrow \omega} &
(C,z) \ar[d]^-{(g,\Gamma)} \\
(B,s) \ar[r]_-{(k,\Theta)} &
(D,v)} \qquad
\xymatrix@R=28pt{
A \ar[r]^-h \ar[d]_-f \ar@{}[rd]|-{\Longdownarrow \omega} &
C \ar[d]^-g \\
B \ar[r]_-k &
D} \qquad
\xymatrix@R=26pt{
g.z.h \ar[r]^-{1.\Xi} \ar[d]_-{\Gamma.1} &
g.h.t \ar[r]^-{\omega.1} &
k.f.t \ar[d]^-{1.\Phi} \\
v.g.h \ar[r]_-{1.\omega} &
v.k.f \ar[r]_-{\Theta.1} &
k.s.f}
$$
is a 2-cell in $\mathcal M$ as in the second diagram, such that the third diagram of 2-cells in $\mathcal M$ commutes.
The horizontal composition of horizontal 1-cells is the composition of 1-cells in $\mathsf{Mnd}(\mathcal M)$ and
the vertical composition of vertical 1-cells is the composition of 1-cells in $\mathsf{Mnd}_{\mathsf{op}}(\mathcal M)$.
The horizontal and vertical compositions of 2-cells are given by the same formulae as in \eqref{eq:Sqr_hor} and \eqref{eq:Sqr_vert}, respectively.

A strict monoidal structure $(\otimes,I)$ of $\mathcal M$ induces a strict monoidal structure on $\Mnd (\mathcal M)$ analogously to Example \ref{ex:Mnd}, and a symmetry $x$ on $(\mathcal M,\otimes, I)$ induces a symmetry of $\Mnd (\mathcal M)$ with horizontal and vertical transformations in the first two diagrams, and 2-cell part in the third diagram of 
$$
\xymatrix@C=20pt{
(AB,ts) \ar[r]^-{(x,1)} \ar[d]_-{(fg,\Phi\Gamma)} \ar@{}[rd]|-{\Longdownarrow 1} &
(BA,st) \ar[d]^-{(gf,\Gamma\Phi)} \\
(A'B',t's') \ar[r]_-{(x,1)} &
(B'A',s't')} \quad
\xymatrix@C=28pt{
(AB,ts) \ar[d]_-{(x,1)} \ar[r]^-{(hk,\Xi\Theta)} \ar@{}[rd]|-{\Longdownarrow 1} &
(A'B',t's') \ar[d]^-{(x,1)} \\
(BA,st) \ar[r]_-{(kh,\Theta\Xi)} &
(B'A',s't')} \quad
\xymatrix@C=20pt{
(AB,ts) \ar[r]^-{(x,1)} \ar[d]_-{(x,1)} \ar@{}[rd]|-{\Longdownarrow 1} &
(BA,st) \ar@{=}[d] \\
(BA,st) \ar@{=}[r] &
(BA,st) }
$$
for any 1-cells $(h,\Xi):(A,t)\to (A',t')$ and $(k,\Theta):(B,s) \to (B',s')$ in $\mathsf{Mnd}(\mathcal M)$ and $(f,\Phi):(A,t)\to (A',t')$ and $(g,\Gamma):(B,s) \to (B',s')$ in $\mathsf{Mnd}_{\mathsf{op}}(\mathcal M)$.

Any 2-functor $\mathsf F:\mathcal M\to \mathcal N$ gives rise to a double functor $\Mnd (\mathsf F):\Mnd (\mathcal M) \to \Mnd (\mathcal N)$. On the 0-cells it acts as $\mathsf{Mnd}(\mathsf F)$ of Example \ref{ex:Mnd}; equivalently, as $\mathsf{Mnd}_{\mathsf{op}}(\mathsf F)$ of Example \ref{ex:Mnd}. On the horizontal and vertical 1-cells, $\Mnd (\mathsf F)$ acts as $\mathsf{Mnd}(\mathsf F)$ and $\mathsf{Mnd}_{\mathsf{op}}(\mathsf F)$, respectively. On the 2-cells $\Mnd (\mathsf F)$ acts as $\mathsf F$. 
The double functor $\Mnd (\mathsf F)$ is strict monoidal whenever $\mathsf F$ is strict monoidal and it is symmetric whenever $\mathsf F$ is so.

A 2-natural transformation $\omega: \mathsf F \to \mathsf G$ induces a double natural transformation $\Mnd (\omega):\Mnd (\mathsf F) \to \Mnd (\mathsf G)$ with horizontal and vertical transformations in the first two diagrams, and 2-cell part in the third diagram of 
$$
\xymatrix{
(\mathsf F A,\mathsf F t) \ar[r]^-{(\omega_A,1)} \ar[d]_-{(\mathsf F f,\mathsf F \Phi)} \ar@{}[rd]|-{\Longdownarrow 1} &
(\mathsf G A,\mathsf G t) \ar[d]^-{(\mathsf G f,\mathsf G \Phi)} \\
(\mathsf F B,\mathsf F s) \ar[r]_-{(\omega_B,1)} &
(\mathsf G B,\mathsf G s) }\ 
\xymatrix{
(\mathsf F A,\mathsf F t) \ar[d]_-{(\omega_A,1)} \ar[r]^-{(\mathsf F h,\mathsf F \Xi)} \ar@{}[rd]|-{\Longdownarrow 1} &
(\mathsf F B,\mathsf F s) \ar[d]^-{(\omega_B,1)} \\
(\mathsf G A,\mathsf G t) \ar[r]_-{(\mathsf G h,\mathsf G \Xi)} &
(\mathsf G B,\mathsf G s) }\ 
\xymatrix{
(\mathsf F A,\mathsf F t) \ar[r]^-{(\omega_A,1)} \ar[d]_-{(\omega_A,1)}  \ar@{}[rd]|-{\Longdownarrow 1} &
(\mathsf G A,\mathsf G t)  \ar@{=}[d] \\
(\mathsf G A,\mathsf G t)  \ar@{=}[r] &
(\mathsf G A,\mathsf G t)  }
$$
for any 1-cells $(h,\Xi):(A,t)\to (B,s)$ in $\mathsf{Mnd}(\mathcal M)$ and $(f,\Phi):(A,t)\to (B,s)$ in $\mathsf{Mnd}_{\mathsf{op}}(\mathcal M)$. It is symmetric monoidal whenever $\omega$ is monoidal.

With this we constructed a 2-functor $\Mnd :\mathsf{sm}\mbox{-}\mathsf{2Cat} \to \mathsf{sm}\mbox{-}\mathsf{DblCat}$ which we want to relate next to the 2-functors $(-)_{01}$ of Example \ref{ex:Mpq} and Example \ref{ex:Dpq} .

For any symmetric strict monoidal 2-category $\mathcal M$, a 0-cell of $\Mnd (\mathcal M)_{01}$ is a horizontal pseudomonoid in $\Mnd (\mathcal M)$. Since the horizontal 2-category of $\Mnd (\mathcal M)$ is $\mathsf{Mnd}(\mathcal M)$, this is the same as a pseudomonoid 
$((A,t,\mu,\eta),(m,\tau^2),(u,\tau^0),\alpha,\lambda,\varrho)$
in $\mathsf{Mnd}(\mathcal M)$; equivalently, a 0-cell of $\mathsf{Mnd}(\mathcal M)_{01}$. Thus the isomorphism of Section \ref{sec:0q-monad} takes it to the 0-cell 
$((A,m,u,\alpha,\lambda,\varrho),(t,\tau^2,\tau^0),\mu,\eta)$
of $\mathsf{Mnd}(\mathcal M_{01})$; that is, a 0-cell of $\Mnd (\mathcal M_{01})$ (which was termed an opmonoidal monad in $\mathcal M$).

In the same way, a horizontal 1-cell of $\Mnd (\mathcal M)_{01}$ is a 1-cell of $\mathsf{Mnd}(\mathcal M)_{01}$, taken by the isomorphism of Section \ref{sec:0q-monad} to a 1-cell of $\mathsf{Mnd}(\mathcal M_{01})$; that is, a horizontal 1-cell of $\Mnd (\mathcal M_{01})$.

A vertical 1-cell of $\Mnd (\mathcal M)_{01}$ is an opmonoidal vertical 1-cell of $\Mnd (\mathcal M)$,
$$
\xymatrix{
(A,t) \ar[d]^-{(f,\Phi)} \\
(A',t')} \qquad
\xymatrix{
(AA,tt) \ar[d]_-{(ff,\Phi\Phi)} \ar[r]^-{(m,\tau^2)} \ar@{}[rd]|-{\Longdownarrow{\varphi^2}} &
(A,t) \ar[d]^-{(f,\Phi)} \\
(A'A',t't') \ar[r]_-{(m',\tau^{\prime 2})} & 
(A',t')} \qquad
\xymatrix{
(I,1) \ar@{=}[d] \ar[r]^-{(u,\tau^0)} \ar@{}[rd]|-{\Longdownarrow{\varphi^0}} &
(A,t) \ar[d]^-{(f,\Phi)} \\
(I,1) \ar[r]_-{(u',\tau^{\prime 0})} & 
(A',t').}
$$
It means a 1-cell $(f,\Phi):(A,t) \to (A',t')$ in $\mathsf{Mnd}_{\mathsf{op}}(\mathcal M)$ and a 1-cell $(f,\varphi^2,\varphi^0):(A,m,u,\alpha,$
$\lambda,\varrho) \to (A',m',u',\alpha',\lambda',\varrho')$ in $\mathcal M_{01}$ such that in addition $\varphi^2$ and $\varphi^0$ are 2-cells of $\Mnd (\mathcal M)$; that is, the following diagrams of 2-cells in $\mathcal M$ commute.
$$
\xymatrix{
f.t.m \ar[r]^-{1.\tau^2} \ar[d]_-{\Phi.1} &
f.m.tt  \ar[r]^-{\varphi^2.1} &
m'.ff.tt \ar[d]^-{1.\Phi\Phi} \\
t'.f.m \ar[r]_-{1.\varphi^2} &
t'.m'.ff \ar[r]_-{\tau^{\prime 2}.1} &
m'.t't'.ff} \qquad\qquad
\xymatrix{
f.t.u \ar[r]^-{1.\tau^0} \ar[d]_-{\Phi.1} &
f.u \ar[r]^-{\varphi^0} &
u' \ar@{=}[d] \\
t'.f.u \ar[r]_-{1.\varphi^0} &
t'.u' \ar[r]_-{\tau^{\prime 0}} &
u'.}
$$
Now these conditions can be interpreted, equivalently, as the opmonoidality of the 2-cell $\Phi:(f,\varphi^2,\varphi^0).(t,\tau^2,\tau^0) \to (t',\tau^{\prime 2},\tau^{\prime 0}) .(f,\varphi^2,\varphi^0)$ in $\mathcal M_{01}$. This amounts to saying that $((f,\Phi),\varphi^2,\varphi^0)$ is a vertical 1-cell of $\Mnd (\mathcal M)_{01}$ if and only if $((f,\varphi^2,\varphi^0),\Phi)$ is a vertical 1-cell of $\Mnd (\mathcal M_{01})$.

Finally, a 2-cell of $\Mnd (\mathcal M)_{01}$ is an opmonoidal 2-cell in $\Mnd (\mathcal M)$; that is, a 2-cell of $\mathcal M$ as in the first diagram of
$$
\xymatrix{
A \ar[r]^-h \ar[d]_-f \ar@{}[rd]|-{\Longdownarrow \omega} & 
C  \ar[d]^-g \\
B \ar[r]_-k &
D} \qquad
\xymatrix@R=20pt{
(A,t) \ar[r]^-{(h,\Xi)} \ar[d]_-{(f,\Phi)} \ar@{}[rd]|-{\Longdownarrow \omega} & 
(C,z)  \ar[d]^-{(g,\Gamma)} \\
(B,s) \ar[r]_-{(k,\Theta)} &
(D,v)} \qquad
\xymatrix@C=40pt{
A \ar[r]^-{(h,\chi^2,\chi^0)} \ar[d]_-{(f,\varphi^2,\varphi^0)} \ar@{}[rd]|-{\Longdownarrow \omega} & 
C  \ar[d]^-{(g,\gamma^2,\gamma^0)} \\
B \ar[r]_-{(k,\kappa^2,\kappa^0)} &
D}
$$
which is both a 2-cell in $\Mnd (\mathcal M)$ as in the second diagram and a 2-cell in $\mathcal M_{01}$ as in the third diagram. Now these are the same conditions characterizing a 2-cell of $\Mnd (\mathcal M_{01})$.

With this we constructed an iso double functor $\Mnd (\mathcal M)_{01} \to \Mnd (\mathcal M_{01})$ which is symmetric strict monoidal and 2-natural. This proves the commutativity, up-to the above 2-natural isomorphism, of the first diagram of 2-functors in
$$
\xymatrix{
\mathsf{sm}\mbox{-}\mathsf{2Cat} \ar[r]^-{\Mnd } \ar[d]_-{(-)_{01}} &
\mathsf{sm}\mbox{-}\mathsf{DblCat} \ar[d]^-{(-)_{01}} \\
\mathsf{sm}\mbox{-}\mathsf{2Cat} \ar[r]_-{\Mnd }  &
\mathsf{sm}\mbox{-}\mathsf{DblCat} }\qquad
\xymatrix{
\mathsf{sm}\mbox{-}\mathsf{2Cat} \ar[r]^-{\Mnd } \ar[d]_-{(-)_{10}} &
\mathsf{sm}\mbox{-}\mathsf{DblCat} \ar[d]^-{(-)_{10}} \\
\mathsf{sm}\mbox{-}\mathsf{2Cat} \ar[r]_-{\Mnd }  &
\mathsf{sm}\mbox{-}\mathsf{DblCat}. }
$$
Commutativity of the second diagram follows by symmetric steps (although it does not seem to result from any kind of abstract duality).
Then $\Mnd $ commutes also with the 2-functors $(-)_{pq}$ of Example \ref{ex:Mpq}~(4) and  Example \ref{ex:Dpq}~(4) for any non-negative integers $p$ and $q$.


\section{$(p+q)$-oidal Eilenberg--Moore objects}
\label{sec:K}

In this final section we consider a symmetric strict monoidal 2-category $\mathcal M$ which satisfies the assumptions in all parts of Proposition \ref{prop:V}. Then we construct a symmetric strict monoidal double functor $\mathbf K:\Mnd (\mathcal M)\to \Sqr (\mathcal M)$. Its horizontal 2-functor is the Eilenberg--Moore 2-functor $\mathsf H:\mathsf{Mnd}(\mathcal M)\to \mathcal M$ and its vertical 2-functor is $\mathsf V:\mathsf{Mnd}_{\mathsf{op}}(\mathcal M)\to \mathcal M$ of Proposition \ref{prop:V}. 
Throughout, the hypotheses of Proposition \ref{prop:V} (including those in parts (2) and (3)) are assumed to hold and the notation of Proposition \ref{prop:V} (and its proof) is used.

A 0-cell of $\Mnd (\mathcal M)$ is a monad $(A,t)$ in $\mathcal M$ hence it can be seen as a 0-cell in either 2-category $\mathsf{Mnd}(\mathcal M)$ or $\mathsf{Mnd}_{\mathsf{op}}(\mathcal M)$. Then we may apply to it either one of the Eilenberg--Moore 2-functor $\mathsf H:\mathsf{Mnd}(\mathcal M)\to \mathcal M$ or the 2-functor $\mathsf V:\mathsf{Mnd}_{\mathsf{op}}(\mathcal M)\to \mathcal M$ of Proposition \ref{prop:V}. Both of them yield the Eilenberg-Moore object $\mathbf K(A,t):=A^t$.

A horizontal 1-cell of $\Mnd (\mathcal M)$ is a 1-cell $(h,\Xi)$ of $\mathsf{Mnd}(\mathcal M)$ so it makes sense to put $\mathbf K(h,\Xi):=\mathsf H(h,\Xi)$. Similarly, a vertical 1-cell of $\Mnd (\mathcal M)$ is a 1-cell $(f,\Phi)$ of $\mathsf{Mnd}_{\mathsf{op}}(\mathcal M)$ so we may put $\mathbf K(f,\Phi):=\mathsf V(f,\Phi)$. 

For the definition of the action of the desired double functor $\mathbf K$ on a 2-cell
$$
\xymatrix{
(A,t) \ar[r]^-{(h,\Xi)} \ar[d]_-{(n,\Phi)} \ar@{}[rd]|-{\Longdownarrow \omega} &
(C,z) \ar[d]^-{(g,\Gamma)} \\
(B,s) \ar[r]_-{(k,\Theta)} &
(D,v),} 
$$
universality of the coequalizer in the top row of
\begin{equation} \label{eq:K}
\xymatrix@C=1pt@R=5pt{
& f^v.g.u^z.f^z.u^z.\mathsf H(h,\Xi) \ar@/^.8pc/[rd]^-{1.1.1.\epsilon^z.1} \\ 
f^v.g.z.u^z.\mathsf H(h,\Xi) \ar@{=}@/^.8pc/[ru]  \ar@/_.8pc/[rd]_-{1.\Gamma.1.1} \ar@{=}[dd] &&
f^v.g.u^z.\mathsf H(h,\Xi) \ar[rrrrr]^-{\pi(g,\Gamma).1} \ar@{=}[dd]  &&&&&
\mathsf V(g,\Gamma).\mathsf H(h,\Xi) \ar@{-->}[dddddddd]^-{\mathbf K \omega} \\
&  \stackrel{\displaystyle f^v.v.g.u^z.\mathsf H(h,\Xi)=}{f^v.u^v.f^v.g.u^z.\mathsf H(h,\Xi)} \ar@/_.8pc/[ru]_(.65){\epsilon^v.1.1.1.1} \\
f^v.g.z.h.u^t  \ar[dd]_-{1.1.\Xi.1} && 
f^v.g.h.u^t \ar[dddd]^-{1.\omega.1}  \\
\\
f^v.g.h.t.u^t \ar[dd]_-{1.\omega.1.1} \\
\\
f^v.k.n.t.u^t \ar[dd]_-{\mathsf H\Theta.1.1.1} &&
f^v.k.n.u^t \ar[dd]^-{\mathsf H\Theta.1.1} \\
& \mathsf H(k,\Theta).f^s.n.u^t.f^t.u^t \ar@/^.8pc/[rd]^(.65){1.1.1.1.\epsilon^t} \\
\mathsf H(k,\Theta).f^s.n.t.u^t \ar@{=}@/^.8pc/[ru] \ar@/_.8pc/[rd]_-{1.1.\Phi.1} &&
\mathsf H(k,\Theta).f^s.n.u^t \ar[rrrrr]^-{1.\pi(n,\Phi)} &&&&&
\mathsf H(k,\Theta).\mathsf V(n,\Phi) \\
& \stackrel{\displaystyle \mathsf H(k,\Theta).f^s.s.n.u^t=}{\mathsf H(k,\Theta).f^s.u^s.f^s.n.u^t} \ar@/_.8pc/[ru]_(.65){1.\epsilon^s.1.1.1}}
\end{equation}
is used, see \eqref{eq:Linton}.
Note that for any 1-cell $(k,\Theta):(B,s) \to (D,v)$ in $\mathsf{Mnd}(\mathcal M)$ the multiplicativity condition 
\begin{equation} \label{eq:monmor}
\xymatrix{
v.v.k \ar[rr]^-{\mu.1} \ar[d]_-{1.\Theta} &&
v.k \ar[d]^-\Theta \\
v.k.s \ar[r]_-{\Theta.1} &
k.s.s \ar[r]_-{1.\mu} &
k.s}
\end{equation}
holds, where $\mu:v.v\to v$ is the multiplication of the monad $(D,v)$ and $\mu:s.s\to s$ is the multiplication of the monad $(B,s)$. This can be interpreted as $\Theta$ being a 2-cell in $\mathsf{Mnd}(\mathcal M)$ as on the left of
$$
\xymatrix@R=20pt{
(B,1) \ar[r]^-{(k,1)} \ar[d]_-{(s,\mu)} \ar@{}[rd]|-{\Longdownarrow \Theta} &
(D,1) \ar[d]^-{(v,\mu)} \\
(B,s) \ar[r]_-{(k,\Theta)} &
(D,v)} \qquad \qquad
\xymatrix{
B \ar[r]^-{k} \ar[d]_-{f^s} \ar@{}[rd]|-{\Longdownarrow {\mathsf H\Theta}} &
D \ar[d]^-{f^v} \\
B^s \ar[r]_-{\mathsf H(k,\Theta)} &
D^v.}
$$
Thus the Eilenberg--Moore 2-functor $\mathsf H$ takes it to a 2-cell of $\mathcal M$ on the right. The resulting 2-cell $\mathsf H\Theta$ occurs at the bottom of the vertical paths of \eqref{eq:K}.

In order for the 2-cell $\mathbf K \omega$ to be well-defined, the diagram of \eqref{eq:K} should serially commute. The square on its left commutes if we take the upper ones of the parallel arrows: use that for the Eilenberg--Moore 2-functor $\mathsf H$ the diagram
\begin{equation} \label{eq:EM}
\xymatrix{
z.u^z.\mathsf H(h,\Xi) \ar@{=}[r] \ar@{=}[d] &
u^z.f^z.u^z.\mathsf H(h,\Xi) \ar[r]^-{1.\epsilon^z.1} &
u^z.\mathsf H(h,\Xi)  \ar@{=}[d] \\
z.h.u^t \ar[r]_-{\Xi.1} &
h.t.u^t=h.u^t.f^t.u^t \ar[r]_-{1.1.\epsilon^t} &
h.u^t}
\end{equation}
commutes by the 2-naturality of the counit $(u^t,1.\epsilon^t)$ of the Eilenberg--Moore 2-ad\-junction, and use the middle four interchange law in the 2-category $\mathcal M$. Commutativity of the square on the left of \eqref{eq:K} with the lower ones of the parallel arrows is slightly more involved; it is checked in Figure \ref{fig:K}. The region marked by $(\ast)$ commutes because $\omega$ is a 2-cell of $\Mnd (\mathcal M)$. The image of the square marked by $(\ast\ast)$ under the functor $u^v.(-):\mathcal M(A^t,D^v)\to \mathcal M(A^t,D)$ commutes by \eqref{eq:monmor}. Hence this square commutes by the faithfulness of $u^v.(-):\mathcal M(A^t,D^v)\to \mathcal M(A^t,D)$.
\begin{amssidewaysfigure}
\centering
\scalebox{1}{
\xymatrix@C=65pt@R=25pt{
f^v.g.z.u^z.\mathsf H(h,\Xi) \ar[r]^-{1.\Gamma.1.1} \ar@{=}[d] &
f^v.v.g.u^z.\mathsf H(h,\Xi)  \ar@{=}[r] &
f^v.u^v.f^v.g.u^z.\mathsf H(h,\Xi) \ar[r]^-{\epsilon^v.1.1.1.1} &
f^v.g.u^z.\mathsf H(h,\Xi) \ar@{=}[d] \\
f^v.g.z.h.u^t \ar[r]^-{1.\Gamma.1.1} \ar[d]_-{1.1.\Xi.1} \ar@{}[rdd]|-{(\ast)} &
f^v.v.g.h.u^t \ar@{=}[r] \ar[d]^-{1.1.\omega.1} &
f^v.u^v.f^v.g.h.u^t \ar[r]^-{\epsilon^v.1.1.1.1} &
f^v.g.h.u^t \ar[d]^-{1.\omega.1} \\
f^v.g.h.t.u^t  \ar[d]_-{1.\omega.1.1} &
f^v.v.k.n.u^t \ar@{=}[r] \ar[d]^-{1.\Theta.1.1} \ar@{}[rrdd]|-{(\ast\ast)} &
f^v.u^v.f^v.k.n.u^t \ar[r]^-{\epsilon^v.1.1.1.1} & 
f^v.k.n.u^t  \ar[dd]^-{\mathsf H \Theta.1.1} \\
f^v.k.n.t.u^t \ar[r]^-{1.1.\Phi.1} \ar[d]_-{\mathsf H \Theta.1.1.1} &
f^v.k.s.n.u^t \ar[d]^-{\mathsf H \Theta.1.1.1} && \\
\mathsf H(k,\Theta).f^s.n.t.u^t \ar[r]_-{1.1.\Phi.1} &
\mathsf H(k,\Theta).f^s.s.n.u^t  \ar@{=}[r] &
\mathsf H(k,\Theta).f^s.u^s.f^s.n.u^t \ar[r]_-{1.\epsilon^s.1.1.1} &
\mathsf H(k,\Theta).f^s.n.u^t}}
\caption{
Serial commutativity of \eqref{eq:K}}
\label{fig:K}
\end{amssidewaysfigure}

We turn to showing that the so defined maps combine into a double functor $\mathbf K$. 
The maps in question preserve horizontal and vertical identity 1-cells, as well as the horizontal composition of horizontal 1-cells and the vertical composition of vertical 1-cells since the 2-functors $\mathsf H$ and $\mathsf V$ preserve identity 1-cells and compositions of 1-cells (see Proposition \ref{prop:V} about $\mathsf V$).
Preservation of identity 2-cells as on the left of
$$
\xymatrix{
(A,t) \ar@{=}[r] \ar[d]_-{(n,\Phi)} \ar@{}[rd]|-{\Longdownarrow 1} &
(A,t) \ar[d]^-{(n,\Phi)} \\
(B,s) \ar@{=}[r] &
(B,s)} \qquad\qquad
\xymatrix{
(A,t) \ar@{=}[d] \ar[r]^-{(h,\Xi)} \ar@{}[rd]|-{\Longdownarrow 1} &
(C,z) \ar@{=}[d] \\
(A,t) \ar[r]_-{(h,\Xi)} &
(C,z)}
$$
is obvious since then also the middle column of \eqref{eq:K} is an identity 2-cell. For an identity 2-cell as on the right, \eqref{eq:K} takes the form
$$
\xymatrix@C=35pt@R=10pt{
f^z.z.u^z.\mathsf H (h,\Xi) \ar@<2pt>[r] \ar@<-2pt>[r] \ar@{=}[d] &
f^z.u^z.\mathsf H (h,\Xi) \ar[r]^-{\epsilon^z.1}  \ar@{=}[d] &
\mathsf H (h,\Xi)  \ar@{=}[ddd] \\
f^z.z.h.u^t \ar[d]_-{1.\Xi.1} &
f^z.h.u^t \ar[dd]^-{\mathsf H \Xi.1} \\
f^z.h.t.u^t \ar[d]_-{\mathsf H \Xi.1.1} \\ 
\mathsf H (h,\Xi).f^t.t.u^t \ar@<2pt>[r] \ar@<-2pt>[r]  &
\mathsf H (h,\Xi).f^t.u^t \ar[r]_-{1.\epsilon^t} &
\mathsf H (h,\Xi).}
$$
The square on its right is taken by the functor $u^z.(-):\mathcal M(A^t,C^z) \to \mathcal M(A^t,C)$ to the commutative square of
 \eqref{eq:EM}. Hence it commutes by the faithfulness of $u^z.(-):\mathcal M(A^t,C^z) \to \mathcal M(A^t,C)$ so that also identity 2-cells of this kind are preserved by $\mathbf K$.

For 2-cells
$$
\xymatrix{
(A,t) \ar[r]^-{(h,\Xi)} \ar[d]_-{(n,\Phi)} \ar@{}[rd]|-{\Longdownarrow \omega} &
(C,z) \ar[d]^-{(g,\Gamma)} \\
(B,s) \ar[r]_-{(k,\Theta)} &
(D,v)}
\raisebox{-17pt}{\quad and \quad}
 \xymatrix{
(C,z) \ar[r]^-{(h',\Xi')} \ar[d]_-{(g,\Gamma)} \ar@{}[rd]|-{\Longdownarrow {\omega'}} &
(E,y) \ar[d]^-{(l,\Lambda)} \\
(D,v) \ar[r]_-{(k',\Theta')} &
(F,w)}
$$
of $\Mnd (\mathcal M)$, the 2-cells $(1.\mathbf K \omega)\updot(\mathbf K \omega'.1)$ and $\mathbf K((1.\omega)\updot(\omega'.1))$ are defined as the unique 2-cells making the respective diagrams of Figure \ref{fig:Khorizontal} commute.
The top and bottom rows of the diagrams of Figure \ref{fig:Khorizontal} are equal epimorphisms. So from the equality of their left columns we infer the equality of their right columns. This proves that $\mathbf K$ preserves the horizontal composition of 2-cells.
\begin{amssidewaysfigure}
\centering
\xymatrix@C=35pt@R=26pt{
f^w.l.u^y.\mathsf H(h',\Xi').\mathsf H(h,\Xi) \ar[r]^-{\pi(l,\Lambda).1} \ar@{=}[d] &
\mathsf V(l,\Lambda).\mathsf H(h',\Xi').\mathsf H(h,\Xi) \ar[ddd]_-{\mathbf K \omega'.1} 
&
f^w.l.u^y.\mathsf H(h',\Xi').\mathsf H(h,\Xi) \ar[r]^-{\pi(l,\Lambda).1} \ar@{=}[d] &
\mathsf V(l,\Lambda).\mathsf H(h',\Xi').\mathsf H(h,\Xi) \ar[dddddd]_-{\mathbf K((1.\omega)\updot(\omega'.1))}
\\
f^w.l.h'.u^z.\mathsf H(h,\Xi) \ar[d]^-{1.\omega'.1.1}  &
&
f^w.l.u^y.\mathsf H((h',\Xi').(h,\Xi)) \ar@{=}[d] 
\\
f^w.k'.g .u^z.\mathsf H(h,\Xi) \ar[d]^-{\mathsf H \Theta'.1.1.1} &
&
f^w.l.h'.h.u^t \ar[d]^-{1.\omega'.1.1}
\\
\mathsf H(k',\Theta').f^v.g.u^z.\mathsf H(h,\Xi) \ar[r]^-{1.\pi(g,\Gamma).1} \ar@{=}[d] &
\mathsf H(k',\Theta').\mathsf V(g,\Gamma).\mathsf H(h,\Xi) \ar[ddd]_-{1.\mathbf K \omega} 
&
f^w.k'.g.h.u^t \ar[d]^-{1.1.\omega.1} 
\\
\mathsf H(k',\Theta').f^v.g.h.u^t \ar[d]^-{1.1.\omega.1} &
&
f^w.k'.k.n.u^t \ar[d]^-{\mathsf H((1.\Theta)\updot(\Theta'.1)).1.1} 
\\
\mathsf H(k',\Theta').f^v.k.n.u^t \ar[d]^-{1.\mathsf H \Theta.1.1} &
&
\mathsf H((k',\Theta').(k,\Theta)).f^s.n.u^t \ar@{=}[d] 
\\
\mathsf H(k',\Theta').\mathsf H(k,\Theta).f^s.n.u^t \ar[r]_-{1.1.\pi(n,\Phi)} &
\mathsf H(k',\Theta').\mathsf H(k,\Theta).\mathsf V(n,\Phi)
&
\mathsf H(k',\Theta').\mathsf H(k,\Theta).f^s.n.u^t \ar[r]_-{1.1.\pi(n,\Phi)} &
\mathsf H(k',\Theta').\mathsf H(k,\Theta).\mathsf V(n,\Phi)} 
\caption{$\mathbf K$ preserves the horizontal composition}
\label{fig:Khorizontal}
\end{amssidewaysfigure}

For 2-cells
$$
\xymatrix{
(A,t) \ar[r]^-{(h,\Xi)} \ar[d]_-{(n,\Phi)} \ar@{}[rd]|-{\Longdownarrow \omega} &
(C,z) \ar[d]^-{(g,\Gamma)} \\
(B,s) \ar[r]_-{(k,\Theta)} &
(D,v)}
\raisebox{-17pt}{\quad and \quad}
 \xymatrix{
(B,s) \ar[r]^-{(k,\Theta)} \ar[d]_-{(n',\Phi')} \ar@{}[rd]|-{\Longdownarrow {\omega'}} &
(D,v) \ar[d]^-{(g',\Gamma')} \\
(E,y) \ar[r]_-{(l,\Lambda)} &
(F,w)}
$$
of $\Mnd (\mathcal M)$, the 2-cells $\mathbf K((\omega'.1)\updot(1.\omega))$ and $(\mathbf K \omega'.1)\updot(1.\mathbf K \omega)$ are defined as the unique 2-cells making commute the diagram of Figure \ref{fig:Kvertical_1} and the diagram of Figure \ref{fig:Kvertical_2}, respectively. 
The top and bottom rows of the diagrams of Figure \ref{fig:Kvertical_1} and Figure \ref{fig:Kvertical_2} are equal epimorphisms. So from the equality of their left columns we infer the equality of their right columns. This proves that $\mathbf K$ preserves the vertical composition of 2-cells. Here again, in Figure \ref{fig:Kvertical_1} the region marked by $(\ast)$ commutes since $\omega'$ is a 2-cell of $\Mnd (\mathcal M)$ and the region marked by $(\ast\ast)$ commutes by the same reason as the region $(\ast\ast)$ of Figure \ref{fig:K}.

\begin{amssidewaysfigure}
\centering
\vspace*{2cm}
\scalebox{.85}{$
\xymatrix@C=8pt@R=25pt{
f^w.g'.u^v.f^v.g.u^z.\mathsf H(h,\Xi) \ar[rrrr]^-{\raisebox{8pt}{${}_{\pi(g',\Gamma').1.1.1.1}$}} \ar@{=}[dd] \ar@{=}[rd] \ar@{}[rrrrd]|-{\textrm{Lemma}~\ref{lem:V}} &&&&
\mathsf V (g',\Gamma').f^v.g.u^z.\mathsf H(h,\Xi) \ar[r]^-{\raisebox{8pt}{${}_{1.\pi(g,\Gamma).1}$}} \ar@{=}[d] \ar@{}[rd]|-{\eqref{eq:V-hori}} &
\mathsf V (g',\Gamma').\mathsf V (g,\Gamma).\mathsf H(h,\Xi) \ar@{=}[d] \\
& f^w.g'.v.g.u^z.\mathsf H(h,\Xi) \ar[r]_-{\raisebox{-8pt}{${}_{1.\Gamma'.1.1.1}$}} &
f^w.w.g'.g.u^z.\mathsf H(h,\Xi) \ar@{=}[r] &
f^w.u^w.f^w.g'.g.u^z.\mathsf H(h,\Xi) \ar[r]_-{\raisebox{-8pt}{${}_{\epsilon^w.1.1.1.1.1}$}} &
f^w.g'.g.u^z.\mathsf H(h,\Xi) \ar[r]_-{\raisebox{-8pt}{${}_{\pi((g',\Gamma').(g,\Gamma)).1}$}} \ar@{=}[d] &
\mathsf V ((g',\Gamma').(g,\Gamma)).\mathsf H(h,\Xi) \ar[dddddd]_-{\mathbf K((\omega'.1)\updot(1.\omega))} \\
f^w.g'.u^v.f^v.g.h.u^t \ar@{=}[r]  \ar[d]^-{1.1.1.1.\omega.1} &
f^w.g'.v.g.h.u^t \ar[r]^-{1.\Gamma'.1.1.1}  &
f^w.w.g'.g.h.u^t \ar@{=}[r] &
f^w.u^w.f^w.g'.g.h.u^t \ar[r]^-{\epsilon^w.1.1.1.1.1} &
f^w.g'.g.h.u^t \ar[d]^-{1.1.\omega.1} \\
f^w.g'.u^v.f^v.k.n.u^t \ar@{=}[r] \ar[d]^-{1.1.1.\mathsf H\Theta.1.1} &
f^w.g'.v.k.n.u^t \ar[r]^-{1.\Gamma'.1.1.1} \ar[dd]^-{1.1.\Theta.1.1} \ar@{}[rddd]|-{(\ast)} &
f^w.w.g'.k.n.u^t \ar@{=}[r] \ar[dd]^-{1.1.\omega'.1.1} &
f^w.u^w.f^w.g'.k.n.u^t \ar[r]^-{\epsilon^w.1.1.1.1.1} &
f^w.g'.k.n.u^t \ar[dd]^-{1.\omega'.1.1}  \\
f^w.g'.u^v.\mathsf H(k,\Theta).f^s.n.u^t \ar@{=}[d] \\
f^w.g'.k.u^s.f^s.n.u^t \ar@{=}[r] \ar[d]^-{1.\omega'.1.1.1.1} &
f^w.g'.k.s.n.u^t \ar[d]^-{1.\omega'.1.1.1} &
f^w.w.l.n'.n.u^t \ar@{=}[r] \ar[d]^-{1.\Lambda.1.1.1} \ar@{}[rrdd]|-{(\ast\ast)}& 
f^w.u^w.f^w.l.n'.n.u^t \ar[r]^-{\epsilon^w.1.1.1.1.1} &
f^w.l.n'.n.u^t \ar[dd]^-{\mathsf H\Lambda.1.1.1} \\
f^w.l.n'.u^s.f^s.n.u^t \ar@{=}[r] \ar[dd]^-{\mathsf H\Lambda.1.1.1.1.1} &
f^w.l.n'.s.n.u^t \ar[r]^-{1.1.\Phi'.1.1} &
f^w.l.y.n'.n.u^t \ar[d]^-{\mathsf H\Lambda.1.1.1.1} \\
\ar@{}[rrrrd]|-{\textrm{Lemma}~\ref{lem:V}} 
& \mathsf H(l,\Lambda).f^y.n'.s.n.u^t \ar@{=}[ld] \ar[r]^-{\raisebox{8pt}{${}_{1.1.\Phi'.1.1}$}} &
\mathsf H(l,\Lambda).f^y.y.n'.n.u^t \ar@{=}[r] &
\mathsf H(l,\Lambda).f^y.u^y.f^y.n'.n.u^t \ar[r]^-{\raisebox{8pt}{${}_{1.\epsilon^y.1.1.1.1}$}} &
\mathsf H(l,\Lambda).f^y.n'.n.u^t \ar@{=}[d] \ar[r]^-{\raisebox{8pt}{${}_{1.\pi((n',\Phi').(n,\Phi))}$}} \ar@{}[rd]|-{\eqref{eq:V-hori}} &
\mathsf H(l,\Lambda) \mathsf V((n',\Phi').(n,\Phi)) \ar@{=}[d] \\
\mathsf H(l,\Lambda).f^y.n'.u^s.f^s.n.u^t \ar[rrrr]_-{\raisebox{-8pt}{${}_{1.\pi(n',\Phi').1.1.1}$}} &&&&
\mathsf H(l,\Lambda).\mathsf V(n',\Phi').f^s.n.u^t \ar[r]_-{\raisebox{-8pt}{${}_{1.1.\pi(n,\Phi)}$}} &
\mathsf H(l,\Lambda).\mathsf V(n',\Phi').\mathsf V(n,\Phi)}$}
\caption{$\mathbf K$ preserves the vertical composition --- first part}
\label{fig:Kvertical_1}
\end{amssidewaysfigure}

\begin{amssidewaysfigure}
\centering
\xymatrix@C=100pt@R=25pt{
f^w.g'.u^v.f^v.g.u^z.\mathsf H(h,\Xi) \ar[r]^-{\pi(g',\Gamma').1.1.1.1} \ar@{=}[d] &
\mathsf V (g',\Gamma').f^v.g.u^z.\mathsf H(h,\Xi) \ar[r]^-{1.\pi(g,\Gamma).1} \ar@{=}[d] &
\mathsf V (g',\Gamma').\mathsf V (g,\Gamma).\mathsf H(h,\Xi) \ar[ddd]^-{1.\mathbf K \omega} \\
f^w.g'.u^v.f^v.g.h.u^t \ar[d]_-{1.1.1.1.\omega.1} &
\mathsf V (g',\Gamma').f^v.g.h.u^t \ar[d]^-{1.1.\omega.1} \\
f^w.g'.u^v.f^v.k.n.u^t \ar[d]_-{1.1.1.\mathsf H\Theta.1.1} &
\mathsf V (g',\Gamma').f^v.k.n.u^t \ar[d]^-{1.\mathsf H\Theta.1.1} \\
f^w.g'.u^v.\mathsf H(k,\Theta).f^s.n.u^t \ar[r]^-{\pi(g',\Gamma').1.1.1.1} \ar@{=}[d] &
\mathsf V (g',\Gamma').\mathsf H(k,\Theta).f^s.n.u^t \ar[r]^-{1.1.\pi(n,\Phi)} \ar[ddd]^-{\mathbf K \omega'.1.1.1} &
\mathsf V (g',\Gamma').\mathsf H(k,\Theta).\mathsf V (n,\Phi) \ar[ddd]^-{\mathbf K \omega'.1} \\
f^w.g'.k.u^s.f^s.n.u^t \ar[d]_-{1.\omega'.1.1.1.1} \\
f^w.l.n'.u^s.f^s.n.u^t \ar[d]_-{\mathsf H \Lambda.1.1.1.1.1} \\
\mathsf H(l,\Lambda).f^y.n'.u^s.f^s.n.u^t \ar[r]_-{1.\pi(n',\Phi').1.1.1} &
\mathsf H(l,\Lambda).\mathsf V(n',\Phi').f^s.n.u^t \ar[r]_-{1.1.\pi(n,\Phi)} &
\mathsf H(l,\Lambda).\mathsf V(n',\Phi').\mathsf V(n,\Phi)}
\caption{$\mathbf K$ preserves the vertical composition --- second part}
\label{fig:Kvertical_2}
\end{amssidewaysfigure}

\begin{theorem} \label{thm:K}
Consider a 2-category $\mathcal M$ which admits Eilenberg--Moore construction. If the coequalizer \eqref{eq:Linton} exists for all 1-cells in $\mathsf{Mnd}_{\mathsf {op}}(\mathcal M)$, and it is preserved by the horizontal composition on either side with any 1-cell in $\mathcal M$, then the following hold.

(1) The maps, constructed preceding the theorem, constitute a double functor $\mathbf K:\Mnd (\mathcal M)\to \Sqr (\mathcal M)$.

(2) If moreover $\mathcal M$ is a strict monoidal 2-category which admits monoidal Eilenberg--Moore construction, then the double functor $\mathbf K$ of part (1) is strict monoidal. 

(3) If in addition the strict monoidal 2-category $\mathcal M$ has a symmetry then the strict monoidal double functor $\mathbf K$ of part (2) is symmetric as well.
\end{theorem}

\begin{proof}
We checked part (1) during the construction. Parts (2) and (3) clearly follow from the strict monoidality and the symmetry, respectively, of the 2-functors $\mathsf H$ and $\mathsf V$  under the stated conditions (see Proposition \ref{prop:mon-EM} and Proposition \ref{prop:V}).
\end{proof}

In the setting of Theorem \ref{thm:K}~(3), the object map of the induced symmetric strict monoidal double functor
$$
\xymatrix{
\Mnd (\mathcal M_{pq})\cong \Mnd (\mathcal M)_{pq}  \ar[r]^-{\mathbf K_{pq}} &
\Sqr (\mathcal M)_{pq}\cong \Sqr (\mathcal M_{pq})}
$$
sends a $(p,q)$-oidal monad $(A,t)$ in $\mathcal M$ to its $(p+q)$-oidal Eilenberg--Moore object $A^t$. While $q$ ones of the pseudomonoid structures of $A^t$ are liftings of the respective pseudomonoid structures of $A$ along the 1-cell $u^t:A^t\to A$,  the other $p$ pseudomonoid structures of $A^t$ are liftings of the corresponding pseudomonoid structures of $A$ along the 1-cell $f^t:A\to A^t$.

Applying this to the particular symmetric strict monoidal 2-category of Remark \ref{rem:reflCat}, we obtain very similar results to \cite[Corollary 6.4 and Theorem 9.1]{AguiarHaimLopezFranco}.
The only difference is that our assumptions are slightly stronger. Recall that a $(p,q)$-oidal object in the symmetric strict monoidal 2-category of Remark \ref{rem:reflCat} is in fact a $(p+q)$-odial category with some further properties. They include, in particular, that
all monoidal products preserve reflexive coequalizers --- not only the first $p$ ones as required in \cite[Theorem 9.1]{AguiarHaimLopezFranco}.

\begin{remark} \label{rem:K_adjoint}
As already discussed in the Introduction, a 2-category $\mathcal M$ is said to admit  Eilenberg--Moore construction if the inclusion 2-functor $\mathcal M\to \mathsf{Mnd}(\mathcal M)$ possesses a right 2-adjoint $\mathsf H$. 
As it is explained below, under the standing assumptions of Theorem \ref{thm:K} (see also Proposition \ref{prop:V}), there is an analogous interpretation of the double functor $\mathbf K:\Mnd (\mathcal M)\to \Sqr (\mathcal M)$ of Theorem \ref{thm:K}~(1) as the right adjoint of an inclusion type double functor, in the sense of \cite{FioreGambinoKock:free}.  In the terminology of \cite{FioreGambinoKock:free} this means that $\mathbf K$ provides Eilenberg--Moore construction for the double category $\Sqr (\mathcal M)$.  We are grateful to an anonymous referee for raising this question.

(1) Any 0-cell of a 2-category $\mathcal M$ can be seen as a monad with identity 1-cell part and identity multiplication and unit 2-cells. Any 1-cell in $\mathcal M$ can be regarded as a monad morphism between these trivial monads --- that is, as a horizontal 1-cell in $\Mnd (\mathcal M)$ --- with identity 2-cell part. Symmetrically, it can be regarded as a vertical 1-cell in $\Mnd (\mathcal M)$ with identity 2-cell part. Finally, any 2-cell of $\Sqr (\mathcal M)$ yields a 2-cell in $\Mnd (\mathcal M)$ surrounded by the above trivial 1-cells. These maps constitute the inclusion double functor $\mathbf I:\Sqr (\mathcal M) \to \Mnd (\mathcal M)$. 

Assume now that  $\mathcal M$ admits Eilenberg--Moore construction; and that in $\mathcal M$ the coequalizer \eqref{eq:Linton} exists for all 1-cells in $\mathsf{Mnd}_{\mathsf {op}}(\mathcal M)$ and it is preserved by the horizontal composition on either side with any 1-cell in $\mathcal M$. Then the double functor $\mathbf K$ of Theorem \ref{thm:K}~(1) is available and it turns out to be the right adjoint of the above inclusion double functor $\mathbf I$ in the 2-category of double categories, double functors and {\em horizontal} transformations. The unit of the adjunction is trivial while the counit has the horizontal 1-cell part and the 2-cell part
\begin{equation} \label{eq:K_counit}
\xymatrix@C=30pt{
(A^t,1) \ar[r]^-{(u^t,1.\epsilon^t)}  &
(A,t)}  \qquad \qquad
\xymatrix@C=30pt{
g.u^t \ar[r]^-{\eta^s.1.1} &
u^s.f^s.g.u^t \ar[r]^-{1.\pi(g,\gamma)} &
u^s. \mathsf V(g,\gamma)}
\end{equation}
respectively, for any monads $(A,t)$ and $(B,s)$ in $\mathcal M$, and any vertical 1-cell $(g,\gamma):(A,t) \to (B,s)$ in $\Mnd(\mathcal M)$. 

However, we were not able to construct vertical transformations for the unit and counit of the adjunction $\mathbf I \dashv \mathbf K$, let alone double natural transformations.

(2) If in addition $\mathcal M$ is a strict monoidal 2-category which admits monoidal Eilenberg--Moore construction then $\mathbf I \dashv \mathbf K$ of part (1) is an adjunction in the 2-category of strict monoidal double categories, strict monoidal double functors and monoidal horizontal transformations.

(3) If, moreover, the strict monoidal 2-category $\mathcal M$ has a symmetry then $\mathbf I \dashv \mathbf K$ of part (2) is an adjunction in the 2-category of symmetric strict monoidal double categories, symmetric strict monoidal double functors and symmetric monoidal horizontal transformations.

(4) Recall that for any monad $(A,t)$ in a 2-category $\mathcal M$ admitting Eilenberg-Moore construction, the forgetful 1-cell $u^t:A^t \to A$ possesses a left adjoint $f^t$ in $\mathcal M$ (with unit and counit denoted by $\eta^t:1\to u^t.f^t$ and $\epsilon^t:f^t.u^t\to 1$, respectively). Extending this observation, for the double functors of part (1) there is a vertical transformation $1\to \mathbf I.\mathbf K$ with vertical 1-cell part and 2-cell part
\begin{equation} \label{eq:K_f}
\raisebox{17pt}{$\xymatrix@R=15pt{
(A,t) \ar[d]^-{(f^t,\epsilon^t.1)} \\
(A^t,1)}$}\qquad 
\xymatrix{
f^s.h \ar[r]^-{1.1.\eta^t} &
f^s.h.u^t.f^t \ar@{=}[r] &
f^s.u^s.\mathsf H (h,\chi).f^t \ar[r]^-{\epsilon^s.1.1} &
\mathsf H (h,\chi).f^t }
\end{equation}
respectively, for any monads $(A,t)$ and $(B,s)$ in $\mathcal M$, and any horizontal 1-cell $(h,\chi):(A,t) \to (B,s)$ in $\Mnd(\mathcal M)$. The transformations of \eqref{eq:K_counit} and \eqref{eq:K_f} turn out to be {\em orthogonal adjoints} in the sense of \cite{GrandisPare:dbladj} --- termed {\em conjoints} in \cite{DawsonPare Pronk} --- in the double category whose 0-cells are the double functors, horizontal and vertical 1-cells are the horizontal and vertical transformations, respectively, and whose 2-cells are the modifications in the sense of \cite[Section 1.6]{GrandisPare:dbladj}. By \cite[Section 1.3]{GrandisPare:dbladj} or \cite[page 314]{DawsonPare Pronk} this means the existence of modifications
$$
\xymatrix{
(A^t,1) \ar[r]^-{(u^t,1.\epsilon^t)} \ar@{=}[d] 
\ar@{}[rd]|-{\Longdownarrow {\epsilon^t}} &
(A,t) \ar[d]^-{(f^t,\epsilon^t.1)}  \\
(A^t,1) \ar@{=}[r] &
(A^t,1)} \qquad \qquad
\xymatrix{
(A,t) \ar@{=}[r] \ar[d]_-{(f^t,\epsilon^t.1)} 
\ar@{}[rd]|-{\Longdownarrow {\eta^t}} &
(A,t) \ar@{=}[d]  \\
(A^t,1) \ar[r]_-{(u^t,1.\epsilon^t)} &
(A,t)}
$$
satisfying the triangle conditions requiring that
$$
\xymatrix{
(A^t,1) \ar[r]^-{(u^t,1.\epsilon^t)} \ar@{=}[d] 
\ar@{}[rd]|-{\Longdownarrow {\epsilon^t}} &
(A,t) \ar@{=}[r] \ar[d]|(.45){(f^t,\epsilon^t.1)} 
\ar@{}[rd]|-{\Longdownarrow {\eta^t}} &
(A,t) \ar@{=}[d]  \\
(A^t,1) \ar@{=}[r] &
(A^t,1)\ar[r]_-{(u^t,1.\epsilon^t)} &
(A,t)}
\qquad \raisebox{-20pt}{$\textrm{and}$} \qquad
\xymatrix{
(A,t) \ar@{=}[r] \ar[d]_-{(f^t,\epsilon^t.1)} 
\ar@{}[rd]|-{\raisebox{28pt}{$\Longdownarrow {\eta^t}$}} &
(A,t) \ar@{=}[d]  \\
(A^t,1) \ar[r]^-{(u^t,1.\epsilon^t)} \ar@{=}[d] 
\ar@{}[rd]|-{\Longdownarrow {\epsilon^t}} &
(A,t)\ar[d]^-{(f^t,\epsilon^t.1)}  \\
(A^t,1) \ar@{=}[r] &
(A^t,1)}
$$
are identity 2-cells of $\mathcal M$. 
\end{remark}

In Section \ref{sec:Dpq} we described a 2-functor $\mathsf{Hor}$ from the 2-category of double categories, double functors and horizontal transformations to $\mathsf{2Cat}$ (it assigns to a double category its horizontal 2-category).  Applying it to the adjunction $\mathbf I \dashv \mathbf K$ of Remark \ref{rem:K_adjoint}~(1) we re-obtain the inclusion 2-functor $\mathcal M \to \mathsf{Mnd}(\mathcal M)$ and its right 2-adjoint $\mathsf H$ in Section \ref{sec:q-EM}; that is, the Eilenberg--Moore construction in $\mathcal M$.
However, since from the same 2-category of double categories, double functors, and horizontal transformations there seems to be no analogous 2-functor taking the vertical 2-category of a double category, Remark \ref{rem:K_adjoint} should have no message about the 2-functor $\mathsf V$ of Proposition \ref{prop:V} being a right 2-adjoint or not.

\begin{remark} \label{rem:EM_10}
Whenever a strict monoidal 2-category $\mathcal M$ admits monoidal Eilenberg--Moore construction, the 2-functor $\mathsf H$ of Section \ref{sec:q-EM} is right adjoint of the inclusion 2-functor $\mathsf J:\mathcal M \to \mathsf{Mnd}(\mathcal M)$ in the 2-category of strict monoidal 2-categories, strict monoidal 2-functors and monoidal natural transformations.
Then  we may apply to this adjunction the 2-functor $(-)_{01}$ in item (i) of the list in Example \ref{ex:Mpq}~(1) so to infer that  $\mathsf H_{01}$ is the right 2-adjoint of the inclusion 
$$
\xymatrix{
\mathcal M_{01} \ar[r]^-{\mathsf J_{01}} &
\mathsf{Mnd}(\mathcal M)_{01} \ar[r]^-\cong  &
\mathsf{Mnd}(\mathcal M_{01} ).}
$$
This shows that under the stated assumptions also $\mathcal M_{01}$ admits Eilenberg--Moore construction described by $\mathsf H_{01}$. This argument can be found in \cite[Theorem 5.1]{Zawadowski} for Cartesian monoidal 2-categories $\mathcal M$.

On the contrary, since the 2-functors $\mathsf{Mnd}$ and $(-)_{10}$ do not commute, no similarly simple argument seems available showing that under suitable circumstances also $\mathcal M_{10}$ admits Eilenberg--Moore construction. However, an application of the adjunction of Remark \ref{rem:K_adjoint}~(2) yields an easy proof. The so obtained result extends Theorem 4.1 of \cite{SzawielZawadowski}.
The author is grateful to an anonymous referee for indicating this question.
 
As in Remark \ref{rem:K_adjoint}~(2), consider a strict monoidal 2-category $\mathcal M$ which admits monoidal Eilenberg--Moore construction, and in which the coequalizer \eqref{eq:Linton} exists for all 1-cells in $\mathsf{Mnd}_{\mathsf {op}}(\mathcal M)$ and it is preserved by the horizontal composition on either side with any 1-cell in $\mathcal M$.

In Example \ref{ex:Dpq}~(2) we constructed a 2-functor $(-)_{10}$ from the 2-category of strict monoidal double categories, strict monoidal double functors and monoidal  horizontal transformations to the 2-category of double categories, double functors and horizontal transformations (see item (ii) of the list in Example \ref{ex:Dpq}~(1)). It takes the adjunction $\mathbf I \dashv \mathbf K$ of Remark \ref{rem:K_adjoint}~(2) to an adjunction $\mathbf I_{10} \dashv \mathbf K_{10}$ which induces a further adjunction
$$
\xymatrix{
\Mnd(\mathcal M_{10})\ar[r]^-\cong &
\Mnd(\mathcal M)_{10}
\ar@/^2pc/[r]^-{\mathbf K_{10}} \ar@{}[r]|-{\displaystyle\top} &
\Sqr(\mathcal M)_{10}
\ar@/^2pc/[l]^-{\mathbf I_{10}} \ar[r]^-\cong &
\Sqr(\mathcal M_{10})}
$$
in the 2-category of double categories, double functors and horizontal transformations.
Applying to it the 2-functor $\mathsf{Hor}$ to $\mathsf{2Cat}$ in Section \ref{sec:Dpq}, we obtain a right 2-adjoint of the inclusion 2-functor $\mathcal M_{10} \to \mathsf{Mnd}(\mathcal M_{10})$, proving that under the stated assumptions $\mathcal M_{10}$ admits Eilenberg--Moore construction. 

Even the explicit form of this right 2-adjoint can be read off this proof. Its action on a 0-cell; that is, the Eilenberg--Moore object of a monad 
$((A,m,u,\alpha,\lambda,\varrho),(t,\tau_2,$
$\tau_0),\mu,\eta)$
in $\mathcal M_{10}$ is equal to its image 
$(A^t=\mathsf V(A,t,\mu,\eta),$
$\mathsf V(m,\tau_2),\mathsf V(u,\tau_0),\mathsf V\alpha,\mathsf V\lambda,\mathsf V\varrho)$
under the 2-functor
$$
\xymatrix{
\mathsf{Mnd}_{\mathsf{op}}(\mathcal M_{10}) \ar[r]^-\cong &
\mathsf{Mnd}_{\mathsf{op}}(\mathcal M)_{10} \ar[r]^-{\mathsf V_{10}} &
\mathcal M_{10}.}
$$
A 1-cell $((h,\chi_2,\chi_0),\Xi)$ of $\mathsf{Mnd}(\mathcal M_{10})$ is sent to the 1-cell $(\mathbf K(h,\Xi)= \mathsf H(h,\Xi),\mathbf K \chi_2,\mathbf K \chi_0)$ of $\mathcal M_{10}$ and a 2-cell $\omega$ is taken to $\mathbf K \omega=\mathsf H \omega$. Thus it describes a lifting along the forgetful 1-cell $u^t:A^t\to A$.

If, furthermore, the strict monoidal 2-category $\mathcal M$ has a symmetry then, based on Remark \ref{rem:K_adjoint}~(3), the above arguments can be iterated to obtain Eilenberg-Moore construction in $\mathcal M_{pq}$ for any non-negative integers $p,q$.
\end{remark}


\bibliographystyle{plain}

\end{document}